\documentclass{amsart}
\usepackage{amssymb}
\usepackage{amsmath}
\usepackage{amsthm}
\usepackage{amsxtra}
\usepackage{verbatim}
\usepackage[all]{xy}
\usepackage{diagrams}

\newcommand{\Nilp}{\operatorname{Nilp}}
\newcommand{\Inv}{\operatorname{Inv}}

\newcommand{\ft}{{\frak t}}
\newcommand{\diag}{\operatorname{diag}}

\newcommand{\bG}{{\bf G}}
\newcommand{\bH}{{\bf H}}

\newcommand{\Id}{\operatorname{Id}}

\renewcommand{\mod}{\operatorname{mod}}

\newcommand{\Tor}{\operatorname{Tor}}
\newcommand{\und}{\underline}

\newcommand{\ti}{\tilde}
\newcommand{\pr}{\operatorname{pr}}

\newcommand{\bL}{{\bf L}}
\newcommand{\bP}{{\bf P}}
\newcommand{\II}{{\cal I}}

\newcommand{\Sh}{\operatorname{Sh}}

\newcommand{\G}{{\Bbb G}}
\newcommand{\A}{{\Bbb A}}

\newcommand{\hra}{\hookrightarrow}
\newcommand{\lan}{\langle}
\newcommand{\ran}{\rangle}
\newcommand{\Coh}{\operatorname{Coh}}

\newcommand{\Spec}{\operatorname{Spec}}

\newcommand{\Sp}{\operatorname{Sp}}

\renewcommand{\P}{{\Bbb P}}

\newcommand{\si}{\sigma}

\newcommand{\de}{\delta}
\newcommand{\Si}{\Sigma}

\renewcommand{\ker}{\operatorname{ker}}
\newcommand{\im}{\operatorname{im}}

\numberwithin{equation}{subsection}

\newcommand{\GL}{\operatorname{GL}}

\newtheorem{thm}{Theorem}[subsection]
\newtheorem{prop}[thm]{Proposition}
\newtheorem{lem}[thm]{Lemma}
\newtheorem{cor}[thm]{Corollary}
{  \theoremstyle{definition}
\newtheorem{defi}[thm]{Definition}
\newtheorem{ex}[thm]{Example}

\newtheorem{rem}[thm]{Remark}
\newtheorem{rems}[thm]{Remarks}
}

\newcommand{\Pf}{\noindent {\it Proof}}
\newcommand{\id}{\operatorname{id}}
\newcommand{\Lie}{\operatorname{Lie}}
\newcommand{\ov}{\overline}

\renewcommand{\Im}{\operatorname{Im}}

\newcommand{\rk}{\operatorname{rk}}

\newcommand{\Ab}{{\cal A}b}
\newcommand{\bS}{{\bf S}}
\newcommand{\bB}{{\bf B}}
\newcommand{\bT}{{\bf T}}
\newcommand{\Du}{{\Bbb D}}
\newcommand{\Om}{\Omega}

\newcommand{\St}{\operatorname{St}}
\newcommand{\Hom}{\operatorname{Hom}}

\newcommand{\Ext}{\operatorname{Ext}}
\newcommand{\End}{\operatorname{End}}

\newcommand{\SO}{\operatorname{SO}}
\newcommand{\Ort}{\operatorname{O}}

\renewcommand{\b}{\beta}
\newcommand{\om}{\omega}
\newcommand{\De}{\Delta}
\newcommand{\la}{\lambda}

\newcommand{\C}{{\Bbb C}}

\newcommand{\R}{{\Bbb R}}
\newcommand{\Z}{{\Bbb Z}}
\newcommand{\Q}{{\Bbb Q}}
\newcommand{\dL}{{\Bbb L}}

\newcommand{\Ga}{\Gamma}

\newcommand{\wt}{\widetilde}
\newcommand{\ot}{\otimes}
\newcommand{\Spr}{\operatorname{Spr}}

\newcommand{\sub}{\subset}
\newcommand{\ed}{\qed\vspace{3mm}}

\newcommand{\Irr}{\operatorname{Irr}}

\newcommand{\fg}{{\frak g}}
\newcommand{\fn}{{\frak n}}
\newcommand{\one}{{\mathbf 1}}

\newcommand{\be}{{\bf e}}

\newcommand{\red}{\operatorname{red}}

\let\cal\mathcal
\def\Ascr{{\cal A}}
\def\Bscr{{\cal B}}
\def\Cscr{{\cal C}}
\def\Dscr{{\cal D}}

\def\Fscr{{\cal F}}
\def\Gscr{{\cal G}}

\def\Iscr{{\cal I}}

\def\Lscr{{\cal L}}

\def\Nscr{{\cal N}}
\def\Oscr{{\cal O}}
\def\Pscr{{\cal P}}

\def\Sscr{{\cal S}}
\def\Tscr{{\cal T}}

\def\Xscr{{\cal X}}
\def\Yscr{{\cal Y}}

\let\blb\mathbb
\def\CC{{\blb C}}

\def\FF{{\blb F}} 
\def\QQ{{\blb Q}}
\def\GG{{\blb G}}

\def \AA{{\blb A}}
\def \ZZ{{\blb Z}}

\def \NN{{\blb N}}

\def\id{\text{id}}
\def\Id{\operatorname{id}}
\def\pr{\mathop{\text{pr}}\nolimits}
\let\st\ast

\def\Ab{\mathbb{Ab}}

\def\Lotimes{\overset{L}{\otimes}}

\def\Mod{\operatorname{Mod}}
\def\mod{\operatorname{mod}}
\def\Gr{\operatorname{Gr}}

\def\gr{\operatorname{gr}}
\def\Lie{\mathop{\text{Lie}}}

\def\gr{\operatorname {gr}}
\def\Spec{\operatorname {Spec}}

\def\GL{\operatorname {GL}}

\def\diag{\operatorname {diag}}
\def\Ext{\operatorname {Ext}}
\def\Hom{\operatorname {Hom}}
\def\uHom{\operatorname {\mathcal{H}\mathit{om}}}

\def\End{\operatorname {End}}
\def\uRHom{\operatorname {R\mathcal{H}\mathit{om}}}

\def\im{\operatorname {im}}

\def\ker{\operatorname {ker}}

\def\Tor{\operatorname {Tor}}
\def\End{\operatorname {End}}

\def\id{{\operatorname {id}}}

\def\rk{\operatorname {rk}}

\def\r{\rightarrow}
\def\l{\leftarrow}

\DeclareMathOperator{\Ind}{Ind}

\DeclareMathOperator{\Ad}{Ad}

\let\dirlim\injlim

\newdimen\uboxsep \uboxsep=1ex
\def\uboxn#1{\vtop to 0pt{\hrule height 0pt depth 0pt\vskip\uboxsep
\hbox to 0pt{\hss #1\hss}\vss}}

\def\uboxs#1{\vbox to 0pt{\vss\hbox to 0pt{\hss #1\hss}
\vskip\uboxsep\hrule height 0pt depth 0pt}}

\def\Ind{\operatorname{Ind}}
\def\Perf{\operatorname{Perf}}

\def\gr{\operatorname{gr}}
\def\grmod{\operatorname{grmod}}
\def\dgmod{\operatorname{dgmod}}
\def\Ab{\operatorname{Ab}}
\def\dg{\operatorname{dg}}
\newcommand{\Obscr}{\bf \Oscr}

\def\Ob{\operatorname{Ob}}
\def\uTor{\Tscr\!\mathit{or}}
\def\uExt{\operatorname {\mathcal{E}\mathit{xt}}}

\title[Semiorthogonal decompositions for some reflection groups]{Semiorthogonal decompositions of the categories of equivariant coherent sheaves for some reflection groups}
\author{Alexander Polishchuk \and Michel Van den Bergh}
\thanks{The first author is supported in part by the NSF grant DMS-1400390}
\thanks{The second author is a senior researcher at the FWO and was supported by the FWO grant 1503512N}
\begin{document}
\begin{abstract}
We consider the derived category $D^b_G(V)$
of coherent sheaves on a complex vector space $V$ equivariant with respect to an action of
a finite reflection group $G$. In some cases, including Weyl groups of type $A$, $B$, $G_2$, $F_4$,
as well as the groups $G(m,1,n)=(\mu_m)^n\rtimes S_n$, we construct a semiorthogonal decomposition of this category,
indexed by the conjugacy classes of $G$. The pieces of this decompositions are equivalent to the derived
categories of coherent sheaves on the quotient-spaces $V^g/C(g)$, where $C(g)$ is the centralizer subgroup of $g\in G$.
In the case of the Weyl groups the construction uses some key results about the Springer correspondence, due to Lusztig,
along with some formality statement generalizing a result of Deligne in \cite{deligneweil2}.
We also construct global analogs of some of these semiorthogonal decompositions involving derived categories
of equivariant coherent sheaves on $C^n$, where $C$ is a smooth curve.  
\end{abstract}
\maketitle

\section*{Introduction}

Let $k$ be a field of characteristic zero.
Let $X$ be a smooth quasiprojective variety over $k$, equipped with an action of a finite group $G$.
Then the Hochschild homology of the category of $G$-equivariant coherent sheaves on $X$ has a decomposition
(see \cite{Bar})
\begin{equation}\label{mot-dec2-eq}
HH_*([X/G])\simeq \bigoplus_{g\in G/\sim} HH_*(X^g)^{C(g)},
\end{equation}
where $G/{\sim}$ is the set of conjugacy classes of $G$, $C(g)$ is the centralizer of $g$, $X^g\sub X$ is the
invariant subvariety of $g$.

A similar decomposition exists\footnote{Here we are talking about
the decomposition of additive groups. In order to get a decomposition compatible with the ring structures one needs to tensor
with $\C$.} for the equivariant $K$-theory groups, tensored with $\Q$
 (see \cite{AS}, \cite{Vistoli}),
and on the level of Chow motives (see \cite{Toen}). In fact, it is shown in \cite{Tabuada} that \eqref{mot-dec2-eq} is 
``motivic'' in a suitable sense and hence it holds for any invariant satisfying a few reasonable axioms.
%Because of this we will loosely refer to \eqref{mot-dec2-eq} as
%the {\it motivic decomposition}.

The question we would like to pose is whether some of these decompositions can be lifted to semiorthogonal decompositions
of the derived category $D^b_G(X)$ of $G$-equivariant coherent sheaves on $X$ (we remind what is a 
semiorthogonal decomposition in \S\ref{semiorth-sec}). Of course, one needs some restriction on the action of $G$ on $X$.
For example, if $X$ is a connected projective Calabi-Yau variety and the action of the group preserves the volume form then
the category $D^b_G(X)$ does not have any non-trivial semiorthogonal decompositions. 
The following conjecture says that for a certain class of $G$-actions \eqref{mot-dec2-eq} does
lift to a semiorthogonal decomposition of $D^b_G(X)$.

%If for every cyclic subgroup $C\sub G$ all elements of maximal order in $C$ are conjugate in $G$
%then we have an identification $h(X^C\times s(C))^{N(C)}\simeq h(X^g)^{C(g)}$,
%where $g$ is any generator of $C$,  defined over $\Q$???

\vspace{2mm}

\noindent
{\bf Conjecture A}. {\it Assume that $G$ acts on $X$ effectively, 
and all the geometric quotients $X^g/C(g)$ are smooth for $g\in G$.
Then there exists a semiorthogonal decomposition of the derived category $D^b_G(X)$
of $G$-equivariant coherent sheaves on $X$ such that the pieces $\Cscr_{[g]}$ of this decomposition are in
bijection with conjugacy classes in $G$ and $\Cscr_{[g]}\simeq D^b(X^g/C(g))$.}

\vspace{2mm}

One of the motivations for the above conjecture is the observation
that whenever $X/G$ is smooth, there exists a semiorthogonal decomposition of $D^b_G(X)$ with $D^b(X/G)$
as one of the pieces. Indeed, it is easy to see that the pull-back functor $D^b(X/G)\to D^b_G(X)$ is
fully faithful, so this follows from the fact that $D^b(X/G)$ is saturated (see \cite{BVdB}).

In the case when $\dim X=1$ the required semiorthogonal decomposition was constructed in \cite{P-orbifold}
(see Theorem \ref{CP-thm} and Remark \ref{curves-rem}). 
Note that already in the case $\dim X=1$, the assumption that $G$ acts effectively
on $X$ is important: in the paper \cite{BGLL} the authors give an example of a non-effective action of a finite group
on an elliptic curve for which the conclusion of Conjecture A is false.

We should also add that in all the examples we know, the full embeddings of $D^b(X^g/C(g))$ into $D^b_G(X)$ are
compatible with the natural $D^b(X/G)$-module structures  on these categories, so perhaps this property holds in general.

Our main motivating example for Conjecture A is the case when $X=C^n$, where $C$ is a smooth curve over $\CC$, and
$G=S_n$, the symmetric group acting on the cartesian power $C^n$ by permutations of factors.
We prove Conjecture A in this case by constructing an explicit semiorthogonal decomposition of $D^b_{S_n}(C^n)$,
numbered by partitions of $n$. Note that a less refined decomposition of  $D^b_{S_n}(C^n)$ was 
constructed by Krug in \cite[Cor.\ B, Sec.\ 5.7]{Krug}, however, our methods are quite different.

Let us describe our construction in more detail.
For a partition $\la$ of $n$ consider the 
closed subvariety $C[\la]\sub C^n$ consisting of $(x_1,\ldots,x_n)$ such that
\begin{equation}\label{C-lambda-eq}
x_1=\ldots=x_{\la_1}, x_{\la_1+1}=\ldots=x_{\la_1+\la_2}, \ldots
\end{equation}
($C[\la]\simeq C^k$ where $k$ is the number of parts in $\la$).
Now let 
$$Z_\la(C)\sub C[\la]\times C^n$$ 
be the union over $w\in S_n/S_\lambda$ of the closed subvarieties
$\{(x,wx)\ |\ x\in C[\la]\}\sub C[\la]\times C^n$. We view $Z_\la$ as a closed subscheme of $C[\la]\times C^n$
equipping it with the reduced scheme structure.
Let $r_i$ denote the multiplicity of $i$ as a part occurring in $\la$. Note that the group $W_\la=\prod_i S_{r_i}$
acts on $C[\la]$ by permuting groups of variables corresponding to equal parts in $\la$, and the geometric quotient
$$C^{(\la)}:=C[\la]/W_\la$$ 
is isomorphic to the product of the symmetric powers of $C$, $\prod_i C^{(r_i)}$, so it is smooth.
We let $W_\la\times S_n$ act on the product $C[\la]\times C^n$, where $W_\la$ acts on the first factor
and $S_n$ acts on the second factor, and we set
$$\ov{Z}_\la(C)=Z_{\la}(C)/W_\la.$$ 
We consider the diagram
\begin{equation}\label{N-lambda-diagram}
\begin{diagram}
&&[\ov{Z}_\la(C)/S_n]&&\\
&\ldTo{q_\la}&&\rdTo{\ov{f}_\la}\\
C^{(\la)}&&&&[C^n/S_n]
\end{diagram}
\end{equation}
where $[X/G]$ denotes the stack quotient, $q_\la$ is induced by the projection to the first component,
and the $S_n$-equivariant morphism
$\ov{f}_\la:\ov{Z}_\la(C)\to C^n$ is given by the projection to the second component.
It turns out that the morphism $q_\la$ is flat, and we consider the functor 
$$\Nscr_\la=R(\ov{f}_\la)_*\circ q_\la^*:D^b(C^{(\la)})\to D^b_{S_n}(C^n)$$

For example, for $\la=(n)$, $C[\la]$ is the small diagonal in $C^n$, $\ov{Z}_\la(C)=Z_\la(C)$ is the graph
of the diagonal embedding $C\to C^n$ and $\Nscr_\la$ is simply the push-forward functor 
$$\De_*:D^b(C)\to D^b_{S_n}(C^n).$$
For $\la=(1)^n=(1,\ldots,1)$, $C[\la]=C^n$ and we get the pull-back functor with respect to the projection
$[C^n/S_n]\to C^{(n)}$, where $C^{(n)}$ is the $n$th symmetric power of $C$.
In general, one can think of the functors $\Nscr_\la$ as modifications of the naive induction functors.

\vspace{2mm}

\noindent
{\bf Theorem B}. {\it For each $\la$, $|\la|=n$, the functor $\Nscr_\la$ is fully faithful. One has
$$\Hom_{D^b_{S_n}(C^n)}(\Nscr_\la(\cdot),\Nscr_\mu(\cdot))=0 \ \ \text{ for } \la\not\le\mu,$$
where $\le$ is the dominance partial order on partitions.
For any total ordering $\la_1<\ldots<\la_p$ of partitions of $n$, refining the dominance order we have a semiorthogonal decomposition
$$D^b_{S_n}(C^n)=\lan \Nscr_{\la_1}(D^b(C^{(\la_1)})),\ldots, \Nscr_{\la_p}(D^b(C^{(\la_p)}))\ran.$$}

\vspace{2mm}

Even though the functors $\Nscr_\la$ are given by explicit kernels, our proof of Theorem B is rather indirect.
There are two main steps in the proof. First, we reduce the problem to the case of linear actions, or equivalently,
to the case $C=\A^1$. This step is not hard: one just has to use an explicit local linear model for the action
of $S_n$ on $C^n$ near every point. The second step is to establish the semiorthogonal decomposition in the case $C=\A^1$.
We were not able to find a direct proof of this, staying within the realm equivariant coherent sheaves (except for small $n$).
Instead we use a nontrivial connection between $S_n$-equivariant coherent sheaves on $\A^n$ and constructible sheaves
on the nilpotent cone of the general linear group, provided by the Springer correspondence (see the comments after Theorem C
below).

The natural context for the ``local case" of Conjecture A is to 
study semiorthogonal decompositions of $D^b_G(X)$
is the case when $G$ is a finite complex reflection group acting linearly on the corresponding affine space $X$ (because
these are precisely linear actions for which $X/G$ is smooth).
%(one can think of this class of actions as the ``local case" of Conjecture A).
Our proof of the local case of Theorem B via the Springer correspondence generalizes to the case when $G$ is the Weyl
group $W$ attached to a maximal torus in a simple algebraic group $\bG$, acting on the Lie algebra $\ft$ of the
maximal torus.
In this case the assumption of smoothness of all the quotients $X^g/C(g)$, required for our semiorthogonal decomposition,
is satisfied for Weyl groups which are not of type $D$ or $E$. The construction of the desired semiorthogonal
decomposition is a part of the following theorem about linear actions.

\vspace{2mm}

\noindent
{\bf Theorem C}. {\it Assume $k=\C$. Conjecture A is true for Weyl groups of types 
$A_n$, $B_n$ (equivalently $C_n$), $G_2$ and $F_4$, as well as for 
the complex reflection groups $G(m,1,n)=(\mu_m)^n\rtimes S_n$, and their natural linear representations. 
More precisely, in these cases all the quotients $X^g/C(g)$ are smooth,
and the required semiorthogonal decompositions exist. 
%For type $A$ the corresponding semiorthogonal decomposition is numbered by partitions that should be ordered compatibly with the dominance order. 
}

\vspace{2mm}

It is easy to see 
that for Weyl groups of type $D_n$ (and their natural representations) there exists $g\in G$ such that $X^g/C(g)$ is singular.
Our results in Section \ref{Springer-sec} 
suggest that the same happens for types $E_6$, $E_7$ and $E_8$ although we did not check this.

%Surprisingly, we could not find a direct construction of the semiorthogonal decomposition even for type $A$.
%Rather, we use a nontrivial connection to the category of constructible sheaves on the nilpotent cone $\Nscr\sub\fg$ in the %corresponding Lie algebra provided by the Springer correspondence. 

In the case of the Weyl groups the proof of Theorem C is based on
the isomorphism of algebras
$$\Ext^*_\bG(\Ascr,\Ascr)\simeq W\ltimes S(\ft^*),$$
where 
%$\bG$ is a connected reductive group, $\ft$ is the Lie algebra of the maximal torus, $W$ the Weyl group, and 
$\Ascr$ is the {\it Springer sheaf}, which is
a certain $\bG$-equivariant perverse sheaf on the nilpotent cone $\Nscr$ in the Lie algebra of $\bG$, playing a key role in the Springer correspondence.
The above isomorphism, explicitly stated in \cite{Kato}, 
is a consequence of Lusztig's identification of the $\bG\times\G_m$-equivariant $\Ext$-algebra
with the graded Hecke algebra in \cite{Lusztig-HAII}.
Together with an appropriate formality statement, this leads to an equivalence of categories
\begin{equation}\label{constr-coh-eq}
D_{\bG,\Spr}(\Nscr)\simeq \Perf(dg-A_W-\mod),
\end{equation}
where the left-hand side is the subcategory of
$\bG$-equivariant constructible sheaves on $\Nscr$, split-generated by $\Ascr$, and the right-hand side is the perfect
derived category of DG-modules over the graded algebra $A_W=W\ltimes S(\ft^*)$. 
We exploit this equivalence of categories to construct the required semiorthogonal decomposition: the pieces of
the semiorthogonal decomposition have a natural geometric definition in terms of constructible sheaves on $\Nscr$.
Note that the equivalence of categories \eqref{constr-coh-eq} was first proved in \cite{Rider} in a different way, using
mixed sheaves. However, our construction of this equivalence has an additional bonus that we can compute explicitly
dg-modules over $A_W$ associated with various constructible sheaves, which eventually leads us to an explicit
form of the semiorthogonal decomposition in Theorem B.

There are two technical issues in the proof of Theorem C, that are of a quite general nature and are dealt with in 
Appendices \ref{dg-sec} and \ref{formality-sec}. 
First, we need to construct a DG-enhancement of the $\bG$-equivariant derived category of $\ell$-adic sheaves on $\Nscr$
and to lift the Frobenius action on $\Ext^*_\bG(\Ascr,\Ascr)$ to the chain level.
Secondly, we prove a general statement of the form ``purity of Frobenius on cohomology implies formality" of a DG-algebra,
generalizing some well known particular cases (such as the one employed by Deligne in \cite[Cor.\ 5.3.7]{deligneweil2}).
%show that properties of such a lifted Frobenius guarantee formality of the corresponding DG-algebra.

%We also prove Conjecture A for some other nonlinear actions. ??? 
%The primary example is the action of the symmetric group $S_n$ on $C^n$, the $n$th cartesian power of a smooth curve
%(see Theorem \ref{global-A-thm}). 
%The semiorthogonal decomposition we construct in this case refines the
%one constructed by Krug in \cite[Cor.\ B, Sec.\ 5.7]{Krug}.

We should mention one more important ingredient in the proof of Theorem B. 
It turns out that the semiorthogonal decomposition in Theorem C can be computed in terms of the equivariant cohomology of Springer fibers. To deduce from this the explicit form of the functors in Theorem B
we use the computation of these equivariant cohomology algebras (in type A) in terms of certain linear arrangements, 
esttablished in \cite{GM} and \cite{KP}.

%Unfortunately, this description requires the assumption of surjectivity of the restriction map $H^*(X,\C)\to H^*(X_a,\C)$,
%where $X_a\sub X$ is the embedding of the Springer fiber into the flag variety. This is the reason why we are able
%to globalize our construction only in type $A$. 
%However, we are then able to use the construction in type $A$ to
%obtain a semiorthogonal decomposition for a global analog of type $B$, associated with
%a smooth curve equipped with a non-trivial involution. 

Using Theorem B we are able to prove Conjecture A for some other group actions.
Namely, we assume that a finite group $G$ acts effectively on a smooth curve $C$, and consider 
the induced action of $G^n\rtimes S_n$ on $C^n$. Combining our semiorthogonal decomposition for the action
of the symmetric group with the decomposition of $D^b_G(C)$ constructed in \cite{CP}, we obtain a semiorthogonal
decomposition of $D^b_{G^n\rtimes S_n}(C^n)$ which is sometimes even finer than the one in Conjecture A
(see Theorem \ref{Gn-thm}).
In the case $G=\Z/m$ and $C=\A^1$ this gives a proof of Theorem C for the complex reflection group $G(m,1,n)$.
In particular, for $G=\Z/2$ and $C=\A^1$ we get a semiorthogonal decomposition for the action of the Weyl group of type $B$,
which is different from the one obtained using the Springer correspondence for the orthogonal (or symplectic) group.

There are other cases of complex reflection groups for which the assumptions of Conjecture A are satisfied 
(see Proposition \ref{H3-H4-prop}). We leave the problem of constructing semiorthogonal decompositions in these
cases for a future work.

Note that although Theorem C does not include the cases of the Weyl groups of
types $D_n$ and $E_n$,
in these cases we still obtain some semiorthogonal decompositions of $D^b_W(\ft)$, which have a
smaller number of pieces than the motivic decomposition \eqref{mot-dec2-eq} (see Theorem \ref{Springer-thm}). 
In some sense in these decompositions nonsmooth pieces $X^g/C(g)$ get replaced by their noncommutative resolutions,
which however ``absorb" some of the other pieces of \eqref{mot-dec2-eq} (for type $D_n$
this is discussed in Section \ref{type-D-sec}).
It would be interesting to study similar less refined semiorthogonal decompositions of $D^b_G(X)$
in other cases when some of the quotients $X^g/C(g)$ are not smooth.
%One can hope that pieces of the categorical decomposition that
%correspond to several motivic pieces provide noncommutative resolutions of singularities of $X^g/C(g)$.

%Discuss reflection groups of rank $2$.

The paper is organized as follows.
In Section \ref{prelim-sec} we gather some preliminaries on semiorthogonal decompositions,
Hochschild (co)homology, DG-modules, and equivariant sheaves.
In Section \ref{Hochschild-sec} we study the canonical decomposition \eqref{mot-dec2-eq}
of the equivariant Hochschild homology
in the case when $G$ is a reflection group acting on a complex vector space $V$.
We show that in the case of real reflection groups, as well as for some complex reflection groups (including $G(m,1,n)$),
it is a decomposition into indecomposable graded $\Oscr(V)^G$-modules.

In Section \ref{Springer-sec} we present the construction of semiorthogonal decompositions of the derived category of
$W$-equivariant coherent sheaves $D^b_W(\ft)$, where $W$ is a Weyl group. 
After reminding the geometric setup of the Springer correspondence in
\S\ref{Springer-prelim-sec} we prove the equivalence of the category of DG-modules 
over the graded algebra $A_W=W\ltimes S(\ft^*)$ with the full subcategory in the derived category of constructible sheaves on the nilpotent cone
generated by the summands of the Springer sheaf (see \S\ref{modules-constr-sec}). 
The main result of this section, Theorem \ref{Springer-thm},
gives a semiorthogonal decomposition of $D^b_W(\ft)$ into subcategories indexed by nilpotent orbits and describes these subcategories as derived categories of modules over certain algebras. In \S\ref{KP-sec}
we discuss the description (in some cases) 
of modules generating the subcategories of our semiorthogonal decomposition, in terms of
some natural linear arrangements, which follows from a similar description of the equivariant cohomology of
Springer fibers in \cite{GM} and \cite{KP}. In Sections \ref{local-A-sec} and \ref{induced-sec} we discuss in more details
the $A_W$-modules giving our semiorthogonal decomposition in the case of type $A$, i.e., for the standard
action of $S_n$ on $\A^n$.

In Section \ref{type-A-sec} we consider some possibly nonlinear actions.  
In \S\ref{global-A-n-sec} we consider the global type $A$ setting, i.e., we construct a semiorthogonal
decomposition of $D^b_{S_n}(C^n)$, where $C$ is a smooth curve, thus proving Theorem B. 
In \S\ref{A1-sec} we consider the natural global analog of type $B$ in the case when the $C$ is equipped with an involution,
and in \S\ref{global-B-sec} we consider the more general case of the $G^n\rtimes S_n$-equivariant sheaves on $C^n$,
where $G$ is any finite group acting effectively on $C$. This includes in particular the proof of Conjecture A for the complex reflection groups $G(m,1,n)$ (see Corollary \ref{Gm1n-cor}).
Note that Theorem C follows from this and from Proposition \ref{exceptional-prop} dealing with types $G_2$ and $F_4$.

Finally, in Section \ref{BD-sec} we consider in detail the semiorthogonal decompositions for the Weyl groups of type
$B$ and $D$. We start with the decompositions constructed in Section \ref{Springer-sec} using the Springer correspondence
and then refine them further. Note that for type $B$ we get two different semiorthogonal decompositions of $D^b_W(\ft)$,
given by Theorems \ref{global-Bn-thm} and \ref{type-B-semiorth-dec-thm}.

The two Appendices, \ref{dg-sec} and \ref{formality-sec}, contain two technical results needed to prove the key statement
that the DG-algebra of endomorphisms of the Springer sheaf is formal (see Theorem \ref{A-W-purity-thm}).
In Section \ref{dg-sec} we show how to lift the Frobenius action on certain $\Ext^*$-algebras in finite characteristic to 
an action on endomorphisms in the DG-version of the (equivariant) derived category of sheaves in characteristic zero.
Then in Section \ref{formality-sec}, generalizing some known results of this kind, we show that the purity of the Frobenius action on the cohomology of a DG-algebra over $\C$ implies its formality. In fact, we show that this statement holds for DG-algebras over operads (see Theorem \ref{ref-1.1-1}).

\medskip

\noindent
{\it Conventions}.
Starting from Sec.\ \ref{Springer-sec} we work over $\CC$. 
If $X$ is a scheme of finite type over $\CC$ then we denote by
$X(\CC)$ the corresponding space of $\CC$-points with the classical topology. 
By a {\it constructible} sheaf on $X$ we mean a sheaf of $\C$-vector spaces on $X(\CC)$, which is constructible with respect 
to an algebraic stratification. Outside sections \S\ref{Hoch-dec-sec} and \S\ref{global-A-n-sec}-\ref{global-B-sec},
an equivariant sheaf means an object of the equivariant derived category of constructible sheaves on $X(\CC)$ (defined in
\cite{BerLun}), whereas in the specified sections we deal with equivariant coherent sheaves. 
When we discuss coherent sheaves
we denote by $D^b(X)$ (resp., $D^b_G(X)$) the bounded derived category of ($G$-equivariant)
coherent sheaves on $X$. For a ring $R$ we denote by $D^f(R)$ the full subcategory in the derived category of $R$-modules,
consisting of bounded complexes of finitely generated $R$-modules.
For a group $G$ acting on a ring $R$ we denote by $G\ltimes R$ the
corresponding crossed product ring. For a collection of subcategories $\Cscr_1,\ldots,\Cscr_n$ in
a triangulated category $\Cscr$ we denote by $\lan \Cscr_1,\ldots,\Cscr_n\ran$ the thick
subcategory of $\Cscr$ generated by $\Cscr_1,\ldots,\Cscr_n$.

\medskip

\noindent
{\it Acknowledgments}. First of all, we would like to acknowledge that we learned key ideas about the
relation of $W$-equivariant coherent sheaves on $\ft$ with constructible sheaves on the nilpotent cone
from the works of Syu Kato \cite{Kato} and Laura Rider \cite{Rider}. In particular, the key equivalence
\eqref{constr-coh-eq} was first proved in \cite{Rider} (we reprove it in a different way to get a more explicit functor
realizing this equivalence).
Also, many of the arguments we use (and some of the semiorthogonalities) 
are present already in the work of Syu Kato \cite{Kato} on Kostka systems.
We started working on this project while visiting the MSRI's special program on Noncommutative Geometry in 2013.
We'd like to extend our gratitide to this institution and to the organizers of this program.
We are grateful to Catharina Stroppel for discussions about the Springer correspondence, 
to Victor Ostrik for showing us the proof of Proposition \ref{star-reflection-prop}, to Gary Seitz for useful discussions
of the centralizers of nilpotent elements in type $E_6$, and to Ben Young
for providing a derivation of the identity \eqref{Bn-comb-id}. We also thank Valery Lunts for informing us about
a counterexample to our original version of Conjecture A from \cite{BGLL}, 
which made us add the assumption of effectivity of the action.

\section{Preliminaries}\label{prelim-sec}

\subsection{Semiorthogonal decompositions}\label{semiorth-sec}

Recall that a {\it semiorthogonal decomposition} of a triangulated category $\Ascr$ is
a pair of triangulated subcategories $\Bscr,\Cscr$ such that
$\Hom(\Cscr,\Bscr)=0$ and every object $A$ of $\Ascr$ fits into an exact triangle
$$C\to A\to B\to C[1]$$
with $B\in\Bscr$ and $C\in\Cscr$. The latter condition can be replaced by the condition  
that $\Ascr$ is classically generated by $\Bscr$ and $\Cscr$, i.e., $\Ascr$ is the smallest thick
subcategory of $\Ascr$ containing $\Bscr$ and $\Cscr$. We
can also require that $\Ascr$ is the smallest triangulated
subcategory of $\Ascr$ containing $\Bscr$ and $\Cscr$. The equivalence of all these conditions
follows from \cite[Lem.\ 1.20]{BLL}.
We write a semiorthogonal decomposition as
$\Ascr=\langle \Bscr,\Cscr\rangle$.

We say that $\Bscr\subset \Ascr$ is \emph{admissible} if there
are semiorthogonal decompositions of $\Ascr$ of the form
$\langle \Bscr,\Cscr\rangle$, $\langle \Cscr',\Bscr\rangle$. 
In this case
$$\Cscr={}^\perp\Bscr:=\{X\in\Ascr\ |\ \Hom(X,B)=0 \ \text{ for all } B\in\Bscr\},$$
$$\Cscr'=\Bscr^\perp:=\{Y\in\Ascr\ |\ \Hom(B,Y)=0 \ \text{ for all } B\in\Bscr\}.$$
The unique projection functor $\Ascr\to {}^\perp\Bscr$ (resp., $\Ascr\to \Bscr^\perp$)
sending $\Bscr$ to $0$, is called {\it left (resp., right) orthogonalization with respect to $\Bscr$.}

Let $R$ be a commutative noetherian ring. We will say that $\Ascr$
is $R$-finite if $\Ascr$ is $R$-linear and for any two objects $A,B\in\Ascr$ one
has that $\bigoplus_i \Hom(A,B[i])$ is a finitely generated $R$-module.

We will use the following criterion for ``intrinsic admissibility''.
\begin{lem}
\label{intrinsic}
  Assume that $\Bscr$ is an $R$-finite dg-enhanced triangulated category
  generated by objects $(M)_{1,\ldots,n} \in\Bscr$ such that $\Hom_{\Bscr}(M_i,M_i[n])=0$ for $n\neq 0$, $\Hom_\Bscr(M_i,M_j[n])=0$ for $i<j$, $n\in \ZZ$
 and $\End_\Bscr(M_i)$ has finite
  global dimension for all $i$. Assume we have a full inclusion $\Bscr\subset
  \Ascr$ of $R$-finite enhanced triangulated categories. Then $\Bscr$
  is admissible inside $\Ascr$.
\end{lem}
\begin{proof} The enhancement allows us to talk 
about $R\Hom$. Set $S=\End_{\Bscr}(M_1)$.
Let $A$ be an object in $\Ascr$. Then $R\Hom_\Ascr(M_1,A)$
has finitely generated cohomology over $R$, and hence it is a perfect
complex of $S$-modules. 
Consider the following
distinguished triangle 
\[
R\Hom(M_1,A)\Lotimes_{S} M_1\xrightarrow{\phi} A\r A'\r \ldots
\]
where $\phi$ is obtained in the usual way from the enhancement. We find
\[
R\Hom_\Ascr(M_1,A')=0,
\]
and so we have a semiorthogonal decomposition $\Ascr=\langle \Bscr_1^\perp, \Bscr_1\rangle$, where
$\Bscr_1$ is generated by $M_1$ and 
$M_2,\ldots,M_n\in \Bscr_1^\perp$. Continuing in this way, we find a semiorthogonal decomposition
$\Ascr=\langle\Bscr^{\perp}, \Bscr\rangle$.

 By considering the opposite categories
$\Bscr^{\circ} \subset\Ascr^{\circ}$ we obtain a semiorthogonal decomposition $\Ascr=\langle \Bscr, {}^\perp\Bscr\rangle$.
\end{proof}

\subsection{Hochschild homology and cohomology}

For a DG-category $\Cscr$ over a field $k$ we denote by $HH_*(\Cscr)$ and $HH^*(\Cscr)$ the Hochschild homology and cohomology, respectively.

Recall that the Hochschild homology is functorial: a DG-functor $\Phi:\Cscr_1\to \Cscr_2$ induces a map of Hochschild homology
$$\Phi_*:HH_*(\Cscr_1)\to HH_*(\Cscr_2).$$ 
On the other hand, 
if $\Phi:\Ascr\sub\Cscr$ is a DG-functor which is fully faithful on the cohomology level then there is a restriction homomorphism
$\Phi^*:HH^*(\Cscr)\to HH^*(\Ascr)$ (since in this case the diagonal bimodule for $\Cscr$ restricts to the diagonal bimodule
for $\Ascr$).
These two operations are compatible via the following identity:
\begin{equation}\label{proj-formula}
x\cap\Phi_*(y)=\Phi_*(\Phi^*(x)\cap y),
\end{equation}
where $x\in HH^*(\Cscr)$, $y\in HH_*(\Ascr)$ and $\cap$ denotes the natural action of Hochschild cohomology on Hochschild
homology. 

Another result we need is that the Hochschild homology is additive with respect to semiorthogonal decompositions.
Namely, if $\Cscr$ is a DG-enhancement of a triangulated category and a pair of DG-subcategories $\Ascr$, $\Bscr$ in $\Cscr$
induces a semiorthogonal decomposition $H^0\Cscr=\lan H^0\Ascr, H^0\Bscr\ran$ 
then the inclusion functors $\Ascr\to\Cscr$ and $\Bscr\to \Cscr$ induce a direct sum decomposition
\begin{equation}\label{HH-semiorth-eq}
HH_*(\Cscr)\simeq HH_*(\Ascr)\oplus HH_*(\Bscr)
\end{equation}
(see \cite{Kuznetsov}). Together with the compatibility \eqref{proj-formula} this leads to the following result.

\begin{lem}\label{HH-decomp-abs-lem} 
Let $\Cscr$ be a DG-category admitting a semiorthogonal decomposition into the subcategories
$\Cscr_1,\ldots,\Cscr_n$. Then there is a decomposition of $HH^*(\Cscr)$-modules
$$HH_*(\Cscr)\simeq \bigoplus_{i=1}^n HH_*(\Cscr_i),$$
where the module structure on $HH_*(\Cscr_i)$ is induced by the homomorphism $HH^*(\Cscr)\to HH^*(\Cscr_i)$.
\end{lem}

\subsection{Graded modules versus DG-modules}\label{graded-dg-sec}

For an abelian category $\Ascr$ let $\Gr(\Ascr)$ be the category of $\ZZ$-graded objects over
$\Ascr$ (with maps of degree $0$), and let $C(\Ascr)$ be the corresponding category of complexes. The shift functor
on $\Gr(\Ascr)$ is written as $(?)$ and on $C(\Ascr)$ as $[?]$. Let $\Ab$ be the category of abelian groups.
It is clear that the functor 
\begin{equation}
\label{abfunctor}
F:C(\Gr(\Ab))\r C(\Ab):M\mapsto \bigoplus_i M_i[-i]
\end{equation}
is compatible with quasi-isomorphisms, homological shifts and cones.
Here $M_i$ is the complex of abelian groups obtained by restricting $M$
to grading degree $i$.

\medskip

Let $A$ be a $\ZZ$-graded ring. We can view it as a DG-ring with zero differential. For a graded $A$-module $M$ let 
$M^{\dg}$ be $M$ viewed as a DG-module over $A$ with zero differential.

We denote by $A-\grmod$ and $A{-}\dgmod$ the categories of graded and DG-modules
over $A$, respectively, and by $D(A{-}\grmod)$ and $D(A{-}\dgmod)$ the corresponding derived categories. 

The functor $F$ from \eqref{abfunctor} specializes to a functor
\[
F:C(A{-}\grmod)\r A{-}\dgmod:M\mapsto \bigoplus_i M_i[-i]
\]
 still compatible with quasi-isomorphisms, shifts and cones since these properties
do not depend on the $A$-structure. So one obtains
an induced exact functor of triangulated categories
\[
F:D(A{-}\grmod)\r D(A{-}\dgmod).
\]
This functor has the following properties
\begin{itemize}
\item
$F(M)=M^{\dg}$ for $M\in A{-}\grmod$;
\item
$F(A(i)[j])=A^{\text{dg}}[i+j]$, thus $F$ preserves perfect complexes; and
\item
$F(M(i)[-i])=FM$ for $M\in D(A-\grmod)$.
\end{itemize}
The last property suggests that $D(A{-}\dgmod)$ 
is some kind of orbit category over $D(A{-}\grmod)$. This is indeed the case and follows
from \cite[\S9.3]{Keller6}. For the benefit of the reader we provide a proof in our special case.

If $M,N\in D(A{-}\grmod)$ then  the canonical morphisms in $C(\Ab)$
\[
F_n:R\Hom_{A-\grmod}(M,N(n))[-n]\r  R\Hom_{A-\dgmod}(FM,F(N(n)[-n]))=R\Hom_{A{-}\dgmod}(FM,FN)
\]
can be combined to a morphism
\[
\tilde{F}:\bigoplus_n R\Hom_{A-\grmod}(M,N(n))[-n]\r R\Hom_{A{-}\dgmod}(FM,FN)
\]
\begin{thm} \label{qi} Assume that $M$ is perfect. Then $\tilde{F}$ is a quasi-isomorphism.
\end{thm}
\begin{proof} By the usual argument involving the 5-lemma it is sufficient to prove this
for $M=A(i)[j]$. Let $N$ be a complex over $A{-}\gr$. Then we have
\begin{equation}
\label{index}
\bigoplus_n R\Hom_{A{-}\grmod}(A(i)[j],N(n)[-n])=\bigoplus_n N_{n-i}[-n-j]
\end{equation}
Similarly
\[
R\Hom_{A{-}\dgmod}(FM,FN)=R\Hom_{A{-}\dgmod}(A[i+j],FN)=(FN)[-i-j]
=\bigoplus_m N_m[-m-i-j]
\]
which is a reindexing of \eqref{index}.
\end{proof}

\begin{cor}\label{graded-generation-cor} 
Assume that $M$ is a perfect object in $D(A{-}\grmod)$.
 If $FM$ generates $D(A^{\dg})$ then $M(i)_{i\in\ZZ}$ generates $D(A{-}\grmod)$ and $M$ 
generates $D(A-\mod)$, the derived category of (ungraded) $A$-modules.
\end{cor}

\begin{proof} Assume we have
$N$ in $D(A{-}\grmod)$ such that 
\[
R\Hom_{A-\grmod}(M(i),N)=0
\]
for all $i$. 
By \eqref{dgext} we conclude $R\Hom_{A{-}\dgmod}(FM,FN)=0$ and hence $FN=0$. Then
clearly $N=0$.

Hence, $M(i)_{i\in\ZZ}$ generate $D(A{-}\grmod)$. The fact that $M$ generates $D(A-\mod)$ follows
from the fact that the essential image of $D(A-\grmod)$ generates $D(A-\mod)$. Indeed,
the latter obviously contains the generator $A$ of $D(A-\mod)$.
\end{proof}

\begin{thm}
\label{dgextth}
Let $M,N\in D(A{-}\grmod)$ and assume that $M$ is perfect. Then
\begin{equation}
\label{dgext}
\bigoplus_n \Ext^{i-n}_{A}(M,N)_n=\Ext^{i}_{A{-}\dgmod}(FM,FN)
\end{equation}
where the $\Ext$ on the left is for ordinary ungraded $A$-modules.
\end{thm}
\begin{proof}
This follows by considering $H^i(\tilde{F})$ and invoking Theorem \ref{qi}.
\end{proof}

Below we will also need the following result similar to Corollary \ref{graded-generation-cor}.

\begin{lem}\label{graded-generation-lem}
Let $M_1,\ldots,M_n$ be perfect objects in $D(A{-}\grmod)$ 
such that the corresponding ungraded
objects $\ov{M}_1,\ldots,\ov{M}_n$
in $D(A-\mod)$ satisfy the assumptions of Lemma \ref{intrinsic}. Assume also that the objects
$FM_1,\ldots,FM_n$ generate an admissible subcategory $\langle FM_1,\ldots,FM_n\rangle\sub D(A{-}\dgmod)$.
Assume now that for some $N\in D(A{-}\grmod)$ we have
$FN\in \langle FM_1,\ldots,FM_n\rangle$.
Then $N$, viewed as an object of $D(A-\mod)$, belongs to the subcategory
$\langle M_1,\ldots,M_n\rangle\sub D(A-\mod)$. 
\end{lem}

\Pf . First, repeating the steps of the proof of Lemma \ref{intrinsic} and observing that the corresponding
Hom-complexes will be graded we can get an exact triangle in $D(A{-}\grmod)$
$$N'\to N\to C\to \ldots$$
with $N'$ in the subcategory generated by $M_i(j)$ and $C$ such that $\Ext^*_A(M_i,C)=0$ for $i=1,\ldots,n$.
By Theorem \ref{dgextth}, this implies that $FC$ lies in the right orthogonal of
$\langle FM_1,\ldots,FM_n\rangle$ in $D(A{-}\dgmod)$. Applying $F$ to the above exact triangle and
using the fact that $FN\in \langle FM_1,\ldots,FM_n\rangle$, we derive that $FC=0$, hence $C=0$.
\ed

\subsection{Equivariant sheaves and equivariant cohomology}\label{equiv-sh-sec}

We refer to \cite{BerLun} for general facts on equivariant derived categories. We will refer to objects of the equivariant
derived categories of constructible sheaves defined in \cite{BerLun} simply as {\it equivariant sheaves}.

Let $\bG$ be a connected linear group over $\CC$, and let $X$ be a $\CC$-scheme. Below we always assume
that $X$ is of finite type over $\CC$. We denote the corresponding equivariatn derived category by $D_{\bG,c}(X)$.
For an $\bG$-equivariant sheaf $F$ on $X$ we denote
by $R\Ga_\bG(X,F)$ the equivariant push-forward of $F$ to the point, which is an object of
the derived category $D_{\bG,c}(pt)$ of $\bG$-equivariant sheaves on the point. The equivariant cohomology 
is obtained by passing to the cohomology of this object:
$$H^*_\bG(X,F)=H^*R\Ga_\bG(X,F).$$
Similarly, for $F,F'\in D_{\bG,c}(X)$ we denote
$$R\Hom_\bG(F,F'):=R\Ga_\bG(X,R\und{\Hom}(F,F')),$$
$$\Ext^*_\bG(F,F')=H^*R\Hom_\bG(F,F').$$

%Lemma on $G$-equivariant Ext on $G$-homogeneous space...

\begin{lem}\label{G/H-lem} 
%Assume characteristic is zero???
Let $\bH\sub \bG$ be a normal subgroup of finite index, and let $X$ be a $\bG$-scheme.

\noindent
(i) For $\bG$-equivariant sheaves $\Fscr$ and $\Gscr$ on $X$ one has a natural isomorphism
$$R\Hom_\bG(\Fscr,\Gscr)\simeq R\Hom_\bH(\Fscr,\Gscr)^{\bG/\bH}.$$

\noindent
(ii) For any $\bG/\bH$-representation $V$ we have a natural isomorphism
$$R\Ga_\bG(X, V^\vee\ot\und{\C})\simeq \Hom_{\bG/\bH}(V, R\Ga_\bH(X,\und{\C})).$$
This isomorphism is compatible with the action of $R\Ga_\bG(pt)=R\Ga_\bH(pt)^{\bG/\bH}$.
\end{lem}

\Pf . (i) Let us denote by $E\bG\to B\bG$ the simplicial version of the contractible $\bG$-torsor over the classifying
space of $\bG$ (see \cite[Sec.\ 6.1]{De-HIII}). By \cite[I.2, App.\ B]{BerLun}, we can think of $\bG$-equivariant (resp.,
$\bH$-equivariant) sheaves on $X$ as objects
of the suitable derived category of sheaves on $E\bG\times_\bG X$ (resp., $E\bG\times_\bH X$). 
Consider the $\bG/\bH$-torsor $\rho:E\bG\times_\bH X\to E\bG\times_\bG X$.
We have a natural isomorphism 
$$\rho_*R\und{\Hom}(\Fscr_\bH,\Gscr_\bH)^{\bG/\bH}\simeq R\und{\Hom}(\Fscr_\bG,\Gscr_\bG),$$
where $\Fscr_\bH$ (resp., $\Fscr_\bG$) is the sheaf on $E\bG\times_\bH X$ (resp., $E\bG\times_\bG X$)
obtained from $\Fscr$ by descent.
Applying the functor $R\Ga$ to both sides gives the required isomorphism.
%$$(\rho_*\Q\ot V^\vee)^{G/H}\sub \rho_*\Q\ot V^\vee$$
%with the sheaf on $EG\times_G X$ obtained from $\Q_{EG\times X}\ot V^\vee$ by descent.
%Passing to cohomology we get the result.

\noindent
(ii) This is a particular case of (i) for $\Fscr=\und{\C}_X$ and $\Gscr=V^\vee\ot\und{\C}_X$.
\ed

As in \cite{GKM} (where the case of torus actions is considered) 
we say that a $\bG$-scheme $X$
is {\it equivariantly formal}
if the Leray-Serre spectral sequence
$$E_2^{pq}=H^p(B\bG, H^q(X,\C)) \implies H_\bG^{p+q}(X)$$
associated with the fibration $X\times_\bG E\bG\to B\bG$ collapses.
In this case one has a (non-canonical) isomorphism of $H^*_\bG(pt)$-modules
$$H^*_\bG(X,\C)\simeq H^*_\bG(pt)\ot H^*(X,\C).$$
For example, if $H^i(X,\C)$ vanishes for all odd $i$ then $X$ is equivariantly formal.

\begin{lem}\label{equiv-formal-lem} 
Let $\bG$ be a connected algebraic group, and let $X$ be a $\bG$-scheme.
Assume that $H^i(X,\C)$ vanishes for all odd $i$.
Then one has a natural isomorphism
$$H^*_\bG(X,\C)\simeq H^*_\bT(X,\C)^W,$$
where $\bT\sub \bG$ is a maximal torus, 
$W=N(\bT)/\bT$ is the corresponding Weyl group.
\end{lem}

\Pf . The action of $N(\bT)$ on the pair $(\bT,X)$ induces an action of $W$
on $H^*_\bT(X,\C)$ and on the corresponding Leray-Serre spectral sequence. Hence, 
from $\bT$-equivariant formality we get an isomorphism
$$H^*_\bT(X,\C)^W\simeq (H^*_\bT(pt)\ot H^*(X,\C))^W\simeq H^*_\bT(pt)^W\ot H^*(X,\C),$$
where we used the fact that $N(\bT)$ acts trivially on $H^*(X,\C)$ (since the action of $N(\bT)$ extends to an action of $\bG$).
On the other hand, by $\bG$-equivariant formality we have an isomorphism
$$H^*_\bG(X,\C)\simeq H^*_\bG(pt)\ot H^*(X,\C).$$
Since the natural morphism $H^*_\bG(X,\C)\to H^*_\bT(X,\C)$ extends to a morphism of Leray-Serre spectral sequences,
the required isomorphism follows from the standard isomorphism
$$H^*_\bG(pt)\simeq H^*_\bT(pt)^W.$$
\ed

We also need to recall the equivariant version of Verdier duality following \cite[Sec.\ 3]{BerLun}.
For a linear algebraic group $\bG$ acting on a variety $X$ the Verdier duality is a contravariant functor
$D$ from the derived category $D^b_{\bG,c}(X)$ of equivariant constructible sheaves on $X$ to itself defined
by 
$$D(F)=R\und{\Hom}(F,D_X),$$
where $D_X=p^!\und{\C}$, where $p:X\to pt$ is the projection to the point.
It satisfies the usual properties (see \cite[Sec.\ 3.5, 3.6]{BerLun}). In particular, $D^2=\Id$, and for proper $X$
and for $F\in D^b_{\bG,c}(X)$ we have an isomorphism
$$R\Ga(X,D(F))\simeq D(R\Ga(X,F))$$
in the category $D_\bG(pt)$. 

Consider the graded algebra $A=H^*_\bG(pt)$. We view $A$ as a DG-algebra with zero differential.
Let $D^f(A-\dgmod)\sub D(A-\dgmod)$ be the full subcategory of bounded complexes of finitely generated
$A$-DG-modules. 
For connected $\bG$, one has an equivalence
$$D^f(A-\dgmod)\simeq D^b_{\bG,c}(pt)$$
such that the Verdier duality functor $D$ corresponds to the duality $M\mapsto R\Hom(M, A)$ on
DG-modules over $A$ (see \cite[Thm.\ 12.7.2]{BerLun}).

%\subsection{Mixed sheaves}

%Recall the definition of the derived category of mixed sheaves (use \cite{Ekedahl} or \cite{Behrend})...

%\section{DG-models and Frobenius}\label{dg-sec}

\section{Hochschild homology for some actions of finite groups}\label{Hochschild-sec}

In this section we discuss the canonical decomposition \eqref{mot-dec2-eq}
of the Hochschild homology of the category of $G$-equivariant
coherent sheaves, where $G$ is a finite group. In the case of a linear action on a vector space $V$ 
we give a sufficient condition for it to be a decomposition into indecomposable graded $\Oscr(V)^G$-modules.
We then study this condition for complex reflection groups.

\subsection{Canonical decomposition of the Hochschild homology}\label{Hoch-dec-sec}

In this section we work over a field $k$ of characteristic zero.

\begin{lem} Let $X$ be a smooth quasiprojective variety with an action of a finite group $G$. Then there is a decomposition
of $\Oscr(X)^G$-modules
\begin{equation}\label{HH-decomposition}
HH_*([X/G])\simeq \bigoplus_{g\in G/\sim} HH_*(X^g)^{C(g)},
\end{equation}
where $X^g\sub X$ is the fixed locus of $g\in G$,
the $\Oscr(X)^G$-module structure on the right is induced by the natural homomorphisms
$\Oscr(X)^G\to \Oscr(X^g)^{C(g)}$. Here $HH_*([X/G])$ is the Hochschild homology of the category of
$G$-equivariant perfect complexes on $X$. 
If in addition, $X$ has a $\G_m$-action commuting with the action of $G$,
such that $X$ admits a $\G_m$-equivariant ample line bundle,
then the relevant Hochschild homology groups get equipped with a natural additional grading, and
\eqref{HH-decomposition} is a decomposition of graded $\Oscr(X)^G$-modules.
\end{lem}

\Pf . The isomorphism \eqref{HH-decomposition} follows from \cite[Prop.\ 4]{Bar} which gives a quasi-isomorphism of the corresponding mixed complexes. More precisely, there is a quasi-isomorphism of mixed complexes of pairs
\begin{equation}\label{mixed-complexes-map}
C(\Cscr^b_{ac}(X)\rtimes G,\Cscr^b(X)\rtimes G)\to \left(\bigoplus_{g\in G} C(\Cscr^b_{ac}(X^g),\Cscr^b(X^g))\right)_G,
\end{equation}
where $\Cscr^b(X)$ denotes the exact category of bounded complexes of vector bundles, $\Cscr^b_{ac}(X)$ the subcategory
of  acyclic complexes, and for a category $\Cscr$ with a $G$-action the category $\Cscr\rtimes G$ is the full subcategory
in the corrresponding category of $G$-equivariant objects in $\Cscr$, 
consisting of objects $\bigoplus_{g\in G} g^*(O)$ with $O\in\Cscr$.
The mixed complex on the left of \eqref{mixed-complexes-map} also maps quasi-isomorphically to the one computing
Hochschild homology of $[X/G]$.
Assume now that $X$ has an additional $\G_m$-action, such that there exists
a $\G_m$-equivariant ample line bundle $L$ on $X$. Now let $\Cscr(L)\sub \Cscr^b(X)$ be the full subcategory 
consisting of complexes with
terms that are finite direct sums of line bundles $L^n$, $n\in\Z$, and let $\Cscr_{ac}(L)\sub\Cscr(L)$ be the subcategory
of acyclic complexes. Then the mixed complex of the pair $(\Cscr_{ac}(L),\Cscr(L))$ has a natural grading coming from
the $\G_m$-action, and the embedding functor $\Cscr(L)\to \Cscr^b(X)$ induces an equivalence of the corresponding
derived categories (see e.g., \cite[Thm.\ 4]{Orlov-dim}). By \cite[Thm.\ 2.4(b)]{kellerexact}, this implies that the morphism of mixed complexes
$$C(\Cscr_{ac}(L)\rtimes G,\Cscr(L)\rtimes G)\to C(\Cscr^b_{ac}(X)\rtimes G,\Cscr^b(X)\rtimes G)$$
is a quasi-isomorphism.
The similar assertion holds for non-equivariant categories associated with $X^g$ and the restriction $L|_{X^g}$.
It remains to observe that the map \eqref{mixed-complexes-map} induces a similar morphism
$$C(\Cscr_{ac}(L)\rtimes G,\Cscr(L)\rtimes G)\to \left(\bigoplus_{g\in G} C(\Cscr_{ac}(L|_{X^g}),\Cscr(L|_{X^g}))\right)_G,
$$
which is compatible with the gradings.
\ed

\begin{prop}\label{Brion-cor-prop}
Assume that $X$ is a smooth quasiprojective variety with an action of a finite group $G$, such that
the quotient $X/G$ is smooth. Then there is a natural isomorphism
$$HH_*(X/G)\simeq HH_*(X)^G.$$
\end{prop}

\Pf . Let $\und{HH}(X):=Rp_{1*}(\Oscr_\De\otimes^{\dL}\Oscr_\De)$ be the sheafified Hochschild homology, so that
$HH_*(X)=R\Ga(X,\und{HH}(X))$. Then we have the Kostant-Hochschild-Rosenberg isomorphism in $D(X)$,
$$\und{HH}(X)\simeq \bigoplus_i \Om_X^i[i]$$
(see \cite[Sec.\ 1.4]{Mark}), and a similar isomorphism for $Y:=X/G$.
Let $\pi:X\to Y$ be the projection map. We claim that for each $i$ the natural map induced by the pull-back,
$$\Om^i_Y\to (\pi_*\Om^i_X)^G,$$
is an isomorphism. Indeed, for affine $X$ this is Brion's theorem in \cite{Brion}. The general case follows since
we can cover $X$ by $G$-invariant affine open sets.
Thus, we get an isomorphism
$$\und{HH}(Y)\simeq\left(\pi_*\und{HH}(X)\right)^G,$$
and the result follows by applying the functor $R\Ga(Y,\cdot)$ to both sides.
\ed

\begin{rem} It is shown in \cite[Thm 1.24]{Tabuada} that Proposition \ref{Brion-cor-prop} holds in fact
for any so-called ``additive invariant''.
\end{rem}

%Note that by Solomon's theorem in \cite{Solomon} (combined with the Hochschild-Kostant-Rosenberg isomorphism) we
%have for every finite (complex) reflection group $G$ a natural isomorphism
%$$HH_*(V)^G\simeq HH_*(V/G).$$

Thus, if for every $g\in G$ the variety $X^g/C(g)$ is smooth
%, i.e., $C(g)$ acts on $V^g$ as a finite reflection group,
then the decomposition \eqref{HH-decomposition} can be rewritten as
$$HH_*(X/G)\simeq \bigoplus_{g\in G/\sim} HH_*(X^g/C(g)),$$
which looks like the decomposition associated with a semiorthogonal decomposition (see \eqref{HH-semiorth-eq}).
Below we will show that in some cases we can match this decomposition with the one obtained from a semiorthogonal
decompositon of $D^b_G(X)$ (see Proposition \ref{HH-semiorth-match-prop} below).

We are interested in the case when a finite group $G$ acts linearly on a vector space $V$.
For $g\in G$ we denote by $C(g)\sub G$ the centralizer of $g$, and by $V^g\sub V$ the subspace of $g$-invariant
vectors. Let us consider the following condition on such an action:

\noindent {\bf
($\star$) For every $g\in G$ the natural map $V^g/C(g)\to V/G$ is birational onto its image.}

Note that if $R$ is a finitely generated commutative
graded $k$-algebra then the category of finitely generated graded $R$-modules has finite-dimensional $\Hom$-spaces
and hence is {\it Krull-Schmidt},
i.e., every object has a direct sum decomposition into indecomposable objects, which are uniquely defined up to
permutation, and the endomorphism ring of every indecomposable object is local (see \cite{Atiyah}).
This in particular applies to the category of finitely generated graded $\Oscr(V)^G$-modules.

\begin{prop}\label{indecomp-prop} 
Assume that a finite group $G$ acts linearly on a vector space $V$, satisfying the condition ($\star$).
The decomposition of graded $\Oscr(V)^G$-modules
\begin{equation}\label{HH0-decomposition}
HH_0([V/G])\simeq \bigoplus_{g\in G/\sim} \Oscr(V^g))^{C(g)},
\end{equation}
obtained from \eqref{HH-decomposition}, is the (unique) decomposition
into indecomposable graded $\Oscr(V)^G$-modules.
%Any direct sum decomposition of $HH_0(\Oscr(V)[G])$ into nonzero indecomposable graded $\Oscr(V)^G$-modules has
%summand isomorphic (up to permutation) to those appearing in \eqref{HH0-decomposition}.
\end{prop}

\Pf . We have to check that $\Oscr(V^g)^{C(g)}$ is indecomposable as a graded $\Oscr(V)^G$-module.
By Lemma \ref{domain-emb-lem}(i) below (applied to $A$ being the image of the homomorphism
$\Oscr(V)^G\to\Oscr(V^g)^{C(g)}$ and $B=\Oscr(V^g)^{C(g)}$), we obtain
$$\End_{\Oscr(V)^G}(\Oscr(V^g)^{C(g)})\simeq\Oscr(V^g)^{C(g)}.$$
It follows that endomorphisms of $\Oscr(V^g)^{C(g)}$ as a {\it graded} $\Oscr(V)^G$-module reduce to $k$.
This shows that $\Oscr(V^g)^{C(g)}$ is indecomposable.
The uniqueness follows from the Krull-Schmidt property.
\ed

\begin{lem}\label{domain-emb-lem} (i) Let $A\to B$ be an embedding of commutative
domains with the same fraction field. Then 
the natural map $B\to\End_A(B)$ is an isomorphism.

\noindent (ii) For $A\to B_1$, $A\to B_2$ as in (i), if $B_1$ and $B_2$ are isomorphic as $A$-modules,
then they are isomorphic as $A$-algebras.
\end{lem}

\Pf . (i) Let $K$ be the common fraction field of $A$ and of $B$. Then $B\ot_A K=K$, so any endomorphism
$B\to B$ of $A$-modules is induced by some $K$-linear map $K\to K$. Such map is a multiplication by $x\in K$.
Since it sends $1$ to an element of $B$, the assertion follows.

\noindent (ii) This follows immediately from (i).
\ed

\begin{prop}\label{HH-semiorth-match-prop} 
Assume that a finite group $G$ acts linearly on a vector space $V$, satisfying the condition ($\star$).
Suppose the derived category $D^b_G(V)$ of equivariant coherent sheaves on $V$ has a semiorthogonal decomposition
with the pieces $\Cscr_i\simeq D^b(X_i)$, $i=1,\ldots,r$, where $X_i$ are smooth affine $\G_m$-varieties over $V/G$, and
$r\ge c$, where $c$ is the number of conjugacy classes in $G$ (we assume that the equivalences
$\Cscr_i\simeq D^b(X_i)$ are compatible with the $D^b(V/G)$-module structures). Then $r=c$, and 
there exists an ordering of the conjugacy classes in $G$,
such that if $g_1,\ldots,g_r$ are representatives, then for each $i$ we have an isomorphism
$X_i\simeq V^{g_i}/C(g_i)$ over $V/G$. In particular, in this case all the varieties $V^{g_i}/C(g_i)$ are smooth.
\end{prop}

\Pf . From such a semiorthogonal decomposition we get a decomposition of the Hochschild homology,
$$HH_0(V/G)=\bigoplus_{i=1}^r HH_0(X_i)=\bigoplus_{i=1}^r \Oscr(X_i),$$
compatible with the $\Oscr(V/G)$-action and with the grading (see Lemma \ref{HH-decomp-abs-lem}). 
Now the Krull-Schmidt property and
Proposition \ref{indecomp-prop} implies that this decomposition should match \eqref{HH0-decomposition}
up to permutation.
\ed

\subsection{The case of reflection groups}

In this section we discuss property ($\star$), as well as the property that all the quotients $V^g/C(g)$
are smooth, for some finite reflection groups, i.e., finite subgroups $G\sub\GL(V)$ generated by pseudoreflections.
It is natural to restrict to such groups since by Chevalley-Shephard-Todd Theorem these are precisely subgroups
for which $V/G$ is smooth (recall that we work over a field $k$ of characteristic zero).

Recall that in the case $k=\CC$ there is a classification of such groups.
The list of irreducible complex reflection groups consists of the infinite family $(G(m,k,n))$ and of $34$ exceptional cases. 
Here $G(m,k,n)$, for $k|m$, is the subgroup of $G(m,1,n):=(\mu_m)^n\rtimes S_n$ formed by elements
$(z_1,\ldots,z_n;\si)$, where $z_i$ are $m$th roots of unity in $\C^*$ such that $(z_1\ldots z_n)^{m/k}=1$.
By the {\it rank} of a reflection group we mean the dimension of $V$. Note that most of the exceptional complex reflection
groups have rank $2$.

Some of our results will be for finite {\it real reflection groups}, i.e., those generated by reflections in a Euclidean space
(abstractly, these are exactly finite Coxeter groups). The list of irreducible finite real reflection groups consists of the Weyl groups, the dihedral groups $G(n,n,2)$, and two more groups $H_3$ (the group of symmetries of the regular icosahedron)
and $H_4$ (the group of symmetries of the regular $120$-cell).
Note that the classical Weyl groups are $G(1,1,n)=S_n$ (of rank $n-1$), $G(2,1,n)=W_{B_n}=W_{C_n}$ and
$G(2,2,n)=W_{D_n}$.

The brief summary is that we can check property ($\star$) and smoothness of all the quotients of $V^g/C(g)$ for the
groups $G(m,1,n)$, for all reflection groups of rank $2$, and for all real reflection groups except for the Weyl groups of types
$D_n$ and $E_n$.

In the next Lemma we reformulate propery ($\star$) in a more convenient form.

\begin{lem}\label{birat-lem} 
For an element $g\in G$ let $H_g\sub G$ be the pointwise stabilizer of $V^g$.
Then the following conditions are equivalent:

\noindent
(i) the map $V^g/C(g)\to V/G$ is birational onto its image;

\noindent
(ii) $N_G(H_g)=C(g)H_g$;

\noindent
(iii) for any $x\in N_G(H_g)$, the elements $g$ and $xgx^{-1}$ are conjugate in $H_g$.
\end{lem}

\Pf . Let us consider the following Zariski open subset in $V^g$:
$$U:=V^g\setminus\cup_{g'\in G\setminus H_g} V^g\cap V^{g'}.$$
Note that for every $g'\not\in H_g$ we have $V^g\cap V^{g'}\neq V^g$, so $U$ is nonempty.
Note also that for every $v\in U$ the stabilizer subgroup of $v$ is exactly $H_g$.
Now suppose we have $v\in U$ and $x\in G$ such that $xv\in U$. Then the stabilizer of $xv$
is $xH_gx^{-1}$, hence, $xH_gx^{-1}=H_g$. In particular, this shows the inclusion $C(g)\sub N_G(H_g)$.

Now let us show the equivalence of (i) and (ii). Using the above observation, we see
that the map $V^g/C(g)\to V/G$ is birational onto its image if and only if
for generic $v\in V^g$ and $x\in N_G(H_g)$, the points $v$ and $xv$ lie in the same $C(g)$-orbit.
Since the stabilizer of such $v$ is $H_g$, this is equivalent to $x\in C(g)H_g$.

On the other hand, the equivalence of (ii) and (iii) is clear: for $x\in N_G(H_g)$ the elements $g$ and $xgx^{-1}$
are conjugate in $H_g$ if and only if $x\in H_gC(g)=C(g)H_g$.
\ed

The proof ot the following result was explained to us by Victor Ostrik.

\begin{prop}\label{star-reflection-prop} 
Let $W$ be a finite subgroup generated by reflections of a real vector space $V_{\R}$.
Then the property ($\star$) holds for the action of $W$ on $V_\C=V_\R\ot \C$.
\end{prop}

\Pf .  For every $g\in G$ the subgroup $H_g\sub W$ is parabolic (i.e., conjugate to a standard subgroup $W_J\sub W$, where
$J$ is a subset of a simple reflections), and $g$ is not
contained in any smaller parabolic subgroup. In other words, the conjugacy class of $g$ is cuspidal in $H_g$ (see 
\cite[Sec.\ 3.1]{GP}). Thus, the condition (iii) of Lemma \ref{birat-lem} follows from \cite[Thm.\ 3.2.11]{GP}.
\ed

\begin{lem}\label{hyper-lem} 
Assume that a finite group $G$ acts linearly on a complex vector space $V$ over a field $k$,
and for an element $g$ the invariant subspace $V^g\sub V$ has codimension $1$.
Then the equivalent conditions of Lemma \ref{birat-lem} hold for $g$.
\end{lem}

\Pf . Note that the determinant homomorphism $\det:G\to k^*$ induces an embedding $H_g\hra k^*$.
Indeed, an element of $H_g$ with trivial determinant is unipotent (since it fixes the hyperplane $V^g$) and of finite order,
hence, it is trivial. Using this we can check the condition (iii) of Lemma \ref{birat-lem}: since $\det(xgx^{-1})=\det(g)$,
we get that for $x\in N_G(H_g)$ one has $xgx^{-1}=g$.
\ed

Now we are ready to check the propery ($\star$) for some other complex reflection groups.

\begin{prop}\label{star-rank2-Gmpn-prop}
Property ($\star$) holds for 

\noindent
(i) finite complex reflection groups of rank $2$;

\noindent
(ii) the groups $G(m,k,n)$ such that either $m$ is a prime number or $k=1$.
\end{prop}

\Pf . (i) The assertion of ($\star$) is clear if $V^g=0$. 
In the case when $V^g$ is $1$-dimensional it holds by Lemma \ref{hyper-lem}.

\noindent
(ii)
%Next, let us consider the case of the group $G(m,p,n)$ acting naturally on $V=\C^n$.
For $(z;\si)=(z_1,\ldots,z_n;\si)\in (\mu_m)^n\rtimes S_n$ let $(C_1,\ldots,C_s)$ be the orbits of $\si$ on $\{1,\ldots,n\}$ 
(i.e., cycles in $\si$). Let us associate with every cycle $C_j$ the element 
\begin{equation}\label{z-C-eq}
z(C_j):=\prod_{i\in C_j} z_i\in \mu_m.
\end{equation}
It is easy to see that the map
$$(z;\si)\mapsto (|C_1|,z(C_1)),\ldots,(|C_s|,z(C_s))$$
gives a bijection between conjugacy classes in $(\mu_m)^n\rtimes S_n$ and colored partitions of $n$, i.e.,
partitions of $n$ with an additional assignment of a color for each part, where the color is an element of $\mu_m$.
Given an element $(z;\si)\in G(m,k,n)$, let us order the orbits of $\si$ so that 
$z(C_1)=\ldots=z(C_r)=1$ and $z(C_j)\neq 1$ for $j>r$. 
Let us set
$$\Sigma=\cup_{j=1}^r C_j.$$
%We claim that replacing $(z;\si)$ by a conjugate element we can further assume that $z_j=0$ for any $j\in\Sigma$. 
%Indeed, ???
For notational convenience let us rename $\si$ to $\si^{-1}$, so that our element is $(z;\si^{-1})$.
Then the fixed subspace of $(z;\si^{-1})$ is
$$V^{(z;\si^{-1})}=\{(x_i)\in V \ |\ z_ix_{\si(i)}=x_i \ \text{ for } i\in C_j, j\le r; \ x_i=0 \ \text{ for } i\not\in\Si\}.$$
Let $U\sub V^{(z;\si^{-1})}$ be the open subset of points $x$ such that $x_i\neq 0$ for $i\in\Sigma$ and
$x_i^m\neq x_{i'}^m$ for $i\in C_j$, $i'\in C_{j'}$, with $j,j'\le r$, $j\neq j'$. It is enough to check that if 
$(t;\tau^{-1})\cdot x=y$, where $x,y\in U$ and $(t;\tau^{-1})\in G(m,k,n)$ then $x$ and $y$ belong
to the same orbit of the centralizer of $(z;\si^{-1})$. First, we claim that
$\tau$ induces a permutation $\ov{\tau}$ of $\{1,\ldots,r\}$
such that $\tau(C_j)=C_{\ov{\tau}(j)}$ for $j=1,\ldots,r$. Indeed, let $i\in C_j$, where $j\le r$. Then
using the equations $(t;\tau^{-1})x=y$ and $(z,\si^{-1})y=y$ we get
\begin{equation}\label{xzt-si-tau-eq}
t_ix_{\tau(i)}=y_i=z_iy_{\si(i)}=z_it_{\si(i)}x_{\tau\si(i)},
\end{equation}
hence $x_{\tau(i)}\neq 0$ and
$x_{\tau(i)}^m=x_{\tau\si(i)}^m$. Since $x\in U$ this implies that $\tau(i)$ and $\tau\si(i)$
belong to the same subset $\Sigma_k$ for some $k\le r$. Thus, $\tau$ permutes the subsets $C_1,\ldots,C_r$,
as we claimed. It follows that there exists a permutation $\tau_1$, commuting with $\si$, acting trivially on the complement
to $\Si$ and such that $\tau_1(C_j)=\tau(C_j)$ for every $j\le r$.
Now let us distinguish two cases.

\noindent
{\bf Case 1. $\Si=\{1,\ldots,n\}$.} Then let us set
$$p_i=\frac{t_i x_{\tau(i)}}{x_{\tau_1}(i)}, \ \ p=(p_1,\ldots,p_n)\in(\C^*)^n.$$
We claim that the element $(p;\tau_1^{-1})\in(\C^*)^n\rtimes S_n$ belongs to $G(m,k,n)$, commutes with 
$(z;\si^{-1})$ and sends $x$ to $y$.  Indeed, since $x_{\tau(i)}^m=x_{\tau_1(i)}^m$ by the choice
of $\tau_1$, we have $p_i^m=1$. Next $\prod_i p_i=\prod_i t_i$, so $(p;\tau_1^{-1})$ is in $G(m,k,n)$.
The equation $(p;\tau_1^{-1})x=y=(t,\tau^{-1})x$ is equivalent to
$$p_ix_{\tau_1(i)}=t_ix_{\tau(i)}.$$
which holds by our choice of $p_i$.
Finally, the fact that $(p;\tau_1^{-1})$ commutes with $(z;\si^{-1})$ is equivalent to the equations
$$\frac{p_{\si(i)}}{p_i}=\frac{z_{\tau_1(i)}}{z_i}.$$
By definition of $p_i$ the left-hand side is
$$\frac{t_{\si(i)}x_{\tau\si(i)}x_{\tau_1(i)}}{t_i x_{\tau_1\si(i)} x_{\tau(i)}}=\frac{x_{\tau_1(i)}}{x_{\tau_1\si(i)}}=
\frac{x_{\tau_1(i)}}{z_ix_{\si\tau_1(i)}}=\frac{z_{\tau_1(i)}}{z_i},$$
where in the first equality we used \eqref{xzt-si-tau-eq}, and the last equality follows from $(z;\si^{-1})$-invariance of $x$.

\noindent
{\bf Case 2. $\Si\neq\{1,\ldots,n\}$.} In this case replacing $(z;\si^{-1})$ by a conjugate element in $G(m,k,n)$
we can assume that $z_i=1$ for $i\in\Si$ (indeed, it is enough to use the conjugation by elements of $(\mu_m)^n\cap G(m,m,n)$). Then $\tau_1$, viewed as an element of $G(m,k,n)$, commutes with $(z;\si^{-1})$. Thus, replacing
$x$ by $\tau_1(x)$, we can assume that $\tau(C_j)=C_j$ for $j\le r$. Furthermore, we have $x_i=x_{i'}$, $y_i=y_{i'}$ for
$i,i'\in C_j$ with $j\le r$. Hence, $y_i=t_ix_i$ for $i\in \Si$. In the case $k=1$ consider the element $t'\in (\mu_m)^n$ such
that $t'_i=t_i$ for $i\in\Si$ and $t'_i=1$ for $i\not\in\Si$. Then $t'$ commutes with $(z,\si^{-1})$ and $y=t'x$, which proves
the assertion in this case. Now assume that $m$ is prime. Let $\si^{-1}=\prod_j c_j$ be the cycle decomposition of $\si^{-1}$,
so that the support of $c_j$ is $C_j$. Let $\zeta=\prod_{i\in\Si}t_i\in\mu_m$. Since $z(C_{r+1})\neq 1$,
there exists an integer $a$ such that $\zeta=z(C_{r+1})^a$. Let us define an element $p\in(\mu_m)^n$ by
$$p_i=\begin{cases} t_i, & i\in\Si,\\ z_i^{-a}, &i\in C_{r+1},\\ 1 \text{ otherwise}.\end{cases}$$ 
Then $\prod_i p_i=1$, so the element $(p;c_{r+1})$ is in $G(m,k,n)$. Also, $(p;c_{r+1})$ commutes with $(z,\si^{-1})$ and
sends $x$ to $y$.
\ed

\begin{ex} The property ($\star$) often fails for groups $G(m,k,n)$ with non-prime $m$.
Here is an example that works for both $G(4,4,5)$ and $G(4,2,5)$.
Let $\si=(23)(45)$, $z=(1,1,-1,1,-1)$. Then the invariant subspace of $(z;\si)$ is the coordinate line spanned by $e_1$.
Now the element $(\zeta_4,\zeta_4^{-1},1,1,1)\in (\mu_4)^5\cap G(4,4,5)$ sends $e_1$ to $\zeta_4 e_1$. 
We claim that an element $(t;\tau)\in G(4,2,5)$, commuting with $(z;\si)$, 
can only map $e_1$ to $\pm e_1$. Indeed, the condition that $(t;\tau)$ is in the centralizer
implies that $\tau(1)=1$ and $t_3/t_2=\pm 1$, $t_5/t_4=\pm 1$. Now the condition $\prod t_i=\pm 1$ gives
$t_1t_2^2t_4^2=\pm 1$, hence $t_1=\pm 1$.
\end{ex}

Now we turn to the property that all the quotients $V^g/C(g)$ are smooth.
It is easy to see that this is true for any group $G(m,1,n)$, in particular, for $S_n=G(1,1,n)$ (see Section \ref{type-A-sec}).
We also have the following result for some groups of small rank.

\begin{prop}\label{H3-H4-prop} 
Let $G$ be the either a complex reflection group of rank $2$, or a 
real reflection group of rank $\le 3$, or the group of type $H_4$. 
Then for every $g\in G$ the geometric quotient $V^g/C(g)$ is
smooth.
\end{prop}

\Pf . If $\dim V^g\le 1$ or $g=1$ then this is clear. 
Assume now that $\dim V=3$ and $\dim V^g=2$, so that $g$ is a reflection.
Let $v$ be a unit vector orthogonal to $V^g$. Then every element of $C(g)$ sends $v$ to $\pm v$. Let
$G_v\sub G$ be the stabilizer of $v$. Then $G_v\sub C(g)$ and $C(g)=G_v\times \lan g\ran$. Hence,
$V^g/C(g)=V^g/G_v$. But $G_v$ is generated by reflections, hence $V^g/G_v$ is smooth.

In the case when $G$ is of type $H_4$ we use the information about the normalizers of the parabolic subgroups 
from the work \cite{Howlett}. Recall that by Proposition \ref{star-reflection-prop}, property ($\star$) holds for $G$, hence
by Lemma \ref{birat-lem}, $V^g/C(g)=V^g/N_G(H)$, where $H=H_g\sub G$ is the pointwise stabilizer of $V^g=V^H$.
Thus, it is enough to check that $N_G(H)$ acts on $V^H$ as a reflection group. But this follows from
Theorem \cite[Thm.\ 6]{Howlett} and the explicit description of $N_G(H)$ for the type $H_4$ (see \cite[p.\ 79]{Howlett}).
\ed

\begin{ex} For groups $G(m,k,n)$ with $k\neq 1$ there often exist elements $g$ with singular $V^g/C(g)$.
For example, for $G(3,3,3)$ we can take $g=(12)$, so that $V^g/C(g)$ is the quotient of the plane by the action of
$\Z/3$ generated by $(x,y)\mapsto (\zeta_3 x,\zeta_3 y)$.
In the case when $G=G(2,2,n)$, the Weyl group of type $D_n$, where $n\ge 4$, there are always elements $g$ with
singular $V^g/C(g)$. In the simplest example $n=4$ we can take $g=(123)$, so that $V^g/C(g)$ is the quotient of the plane
by the action of the involution $(x,y)\mapsto (-x,-y)$.
\end{ex}

\section{Construction of modules via Springer correspondence}\label{Springer-sec}

In this section we present our main construction of a semiorthogonal decomposition of the derived category of $W$-equivariant
coherent sheaves, where $W$ is a Weyl group, on the Lie algebra of a maximal torus in a reductive algebraic group.
The main idea is to use the Springer correspondence to relate this category to the category of constructible sheaves on
the nilpotent cone, where the semiorthogonal decomposition can be constructed geometrically.

\subsection{Springer correspondence}\label{Springer-prelim-sec}

Starting from this section we always assume that $k=\CC$. 
Let $\bG$ be a connected reductive algebraic group over $\C$, $\fg$ its Lie algebra. We denote by $X$ the flag variety associated with $\bG$,
i.e., the variety of Borel subgroups in $\bG$. 

We start by recalling the construction of the Springer correspondence via perverse sheaves
(see \cite{Lusztig-GP}, \cite[ch.\ 13]{Jantzen}).
Let
$$\pi:\wt{\Nscr}\to\Nscr$$
be the Springer resolution of the nilpotent cone $\Nscr\sub\fg$.
Recall that this resolution fits into the diagram with cartesian squares
\begin{diagram}
\wt{\Nscr}&\rTo{}&\wt{\fg}&\lTo{j}&\wt{\fg}_{rs}\\
\dTo{\pi}&&\dTo{\pi_{\fg}}&&\dTo{\pi_{\fg_{rs}}}\\
\Nscr&\rTo{}&\fg&\lTo{j}&\fg_{rs}
\end{diagram}
where $\wt{\fg}\sub \fg\times X$ is the set of pairs $(x,\bB)$ such that $x\in \Lie(\bB)$,
$\fg_{rs}\sub\fg$ is the set of regular semisimple elements.
The construction of the Springer correspondence is based on the study of the complex of constructible sheaves on $\Nscr$,
called the {\it Springer sheaf},
$$\Ascr:=R\pi_*\und{\C}_{\wt{\Nscr}}[\dim\Nscr]\simeq R\pi_{\fg,*}\und{\C}_{\wt{\fg}}[\dim\Nscr]|_{\Nscr}.$$
We view $\Ascr$ as an object of the equivariant derived category $D^b_\bG(\Nscr)$.

The map $\pi_{\fg}$ is small, so we have an isomorphism of $\bG$-equivariant perverse sheaves on $\fg$,
$$R\pi_{\fg,*}\und{\C}_{\wt{\fg}}[\dim\fg]=j_{!*}\pi_{\fg_{rs},*}\und{\C}_{\wt{\fg}_{rs}}[\dim\fg].$$
Now the natural $W$-action on fibers of $\pi_{\fg_{rs}}$ induces an action
of $W$ on $\pi_{\fg,*}\und{\C}_{\wt{\fg}}$ and hence on $\Ascr$ (obtained by restriction to $\Nscr$).
Restricting the sheaf $\Ascr$ to a point $a\in\Nscr$ we obtain the {\it Springer $W$-action} on
$H^*(X_a,\C)$, where 
$$X_a=\pi^{-1}(a)\sub \{a\}\times X=X$$
is the {\it Springer fiber}. Similarly, for a subgroup $K\sub \bG_a$, where $\bG_a\sub\bG$ is the stabilizer of $a$ in $\bG$, we obtain a $W$-action on the equivariant cohomology $H^*_K(X_a,\C)$.

The map $\pi$ is semismall, so $\Ascr$ is a ($\bG$-equivariant) perverse sheaf.
Let $\Irr W$ denote the set of isomorphism classes of irreducible representations of $W$.
Using the decomposition theorem (see \cite[Thm.\ 6.2.5]{BBD}) one shows that
$\Ascr$ decomposes into a direct sum 
\begin{equation}\label{A-decomposition}
\Ascr=\bigoplus_{\chi\in\Irr W}\chi \ot S_{a,\xi},
\end{equation}
where $S_{a,\xi}$ is an irreducible $\bG$-equivariant perverse sheaf on $\Nscr$ of the form
$$S_{a,\xi}=j_{a,!*}\Lscr_\xi[\dim O_a],$$
where $a\in\Nscr$, $O_a=\bG\cdot a$, $j_a:O_a\to \Nscr$ is the embedding,
$\xi$ is an irreducible representation of the component group $\bG_a/\bG_a^0$ (where $\bG_a^0$ is the connected component of $1$ in $\bG_a$), and $\Lscr_{\xi}$ is the local system on $O_a$ associated with $\xi$.
This gives a bijection between $\Irr W$ and a certain set $\Spr$ of pairs $(O_a,\xi)$
(it is known that for every nilpotent orbit $O_a$ there exists $\xi$ such that $(O_a,\xi)\in\Spr$).
The irreducible representation $\chi(a,\xi)$ of $W$ corresponding to the pair $(O_a,\xi)$
appears as the $\xi$-isotypic component in the top degree cohomology of the Springer fiber $X_a$.
Indeed, restricting the decomposition of $\Ascr$ to a point $a\in\Nscr$ we get
$$H^*(X_a,\C)[\dim\Nscr-\dim O_a]\simeq\Ascr|_a[-\dim O_a]\simeq \bigoplus_\xi \chi(a,\xi)\ot\xi\oplus\ldots,$$
where the sum is over $\xi\in \Irr \bG_a/\bG_a^0$ such that $S_{a,\xi}$ occurs in the decomposition
\eqref{A-decomposition},
the remaining summands correspond to Goresky-Macpherson extensions from orbits containing $O_a$
in its closure, and hence, live in negative degrees.
Thus, using the equality
\begin{equation}\label{dim-Xa-eq}
2\dim X_a=\dim \Nscr-\dim O_a
\end{equation}
(see \cite[Thm.\ 10.11]{Jantzen}) we deduce the isomorphism of $W\times \bG_a/\bG_a^0$-representations
$$H^{2\dim X_a}(X_a,\C)\simeq \bigoplus_{\xi} \chi(a,\xi)\ot \xi.$$

Note that choosing a Borel subgroup $\bB\sub \bG$  we obtain an isomorphism
\begin{equation}\label{Spr-res-coh-eq}
H^*_\bG(\wt{\Nscr})\simeq H^*_\bG(\bG\times_\bB \fn)\simeq H^*_\bG(\bG/\bB)\simeq H^*_\bB(pt)=H^*_\bT(pt),
\end{equation}
where $\fn$ is the nilpotent radical in $\Lie(\bB)$, $\bT=\bB/[\bB,\bB]$ is the
maximal torus. Interpreting this cohomology as $\Ext_\bG^*(\und{\C}_{\wt{\Nscr}},\und{\C}_{\wt{\Nscr}})$ and combining
it with the $W$-action on $\Ascr$ we get a homomorphism
\begin{equation}\label{AW-eq}
A_W:=W\ltimes H^*_\bT(pt) \to\Ext_\bG^\bullet(\Ascr,\Ascr)
\end{equation}
of graded algebras, where the grading on $A_W$ is induced by the cohomological grading on $H^*_\bT(pt)$. 
The crucial fact for us is that this map is an isomorphism.
This is explicitly stated in \cite[Thm.\ 3.1]{Kato}, as a consequence of \cite[Thm.\ 8.11]{Lusztig-HAII},
where the $\bG\times\G_m$-equivariant $\Ext$-algebra is identified with the graded Hecke algebra.

Thus, we have a graded action of $A_W$ on $\Ascr$, i.e., a homogeneous 
element $P\in A_W$ of degree $d$ acts as a map $P:\Ascr\to \Ascr[d]$ in the $\bG$-equivariant category. Now, given
a point $a\in\Nscr$, we get by restriction a graded action of $A_W$ on 
$$\Ascr|_{O_a}\simeq R\pi_{O_a,*}\und{\C}_{\pi^{-1}(O_a)}[\dim\Nscr]$$
by $\bG$-equivariant morphisms, where $\pi_{O_a}:\pi^{-1}(O_a)\to O_a$ is the natural projection. 
Using the identifications $O_a\simeq \bG/\bG_a$  and
\begin{equation}\label{O-a-isom}
\pi^{-1}(O_a)\simeq \bG\times_{\bG_a} X_a,
\end{equation}
we can identify $\bG$-equivariant sheaves on $\pi^{-1}(O_a)$ (resp., $O_a$) with $\bG_a$-equivariant sheaves
on $X_a$ (resp., $pt$), so that the $\bG$-equivariant push-forward with respect to the
projection $\pi_{O_a}$ corresponds to taking the $\bG_a$-equivariant cohomology.
Thus, the above construction gives a structure of a graded $A_W$-module on the
$\bG_a$-equivariant cohomology of Springer fibers, $H^*_{\bG_a}(X_a,\C)$. More generally,
if $\xi$ is an irreducible representation of the component group $\bG_a/\bG_a^0$, then twisting
$\Ascr|_{O_a}$ by the corresponding local system $\Lscr_\xi$, we get after the above identification
a structure of a graded $A_W$-module on $(\xi\ot H^*_{\bG_a^0}(X_a,\C))^{\bG_a/\bG_a^0}$
(here we used Lemma \ref{G/H-lem}(ii)).
We need some facts about this module structure.

\begin{lem}\label{two-actions-lem} 
The action of $H^*_\bG(pt)$ on $(\xi\ot H^*_{\bG_a^0}(X_a,\C))^{\bG_a/\bG_a^0}$,
%the $\bG_a$-equivariant cohomology of Springer fibers 
induced by the homomorphism $H^*_\bG(pt)\to H^*_{\bG_a}(pt)\simeq H^*_{\bG_a^0}(pt)^{\bG_a/\bG_a^0}$ 
and by the natural action of $H^*_{\bG_a^0}(pt)$ on the $\bG_a^0$-equivariant cohomology,
coincides with the restriction of the action
of $H^*_\bT(pt)\sub A_W$ via the natural embedding
$H^*_\bG(pt)\to H^*_\bT(pt)$.
\end{lem}

\Pf . Recall that the action of $H^*_\bT(pt)\sub A_W$ on $(\xi\ot H^*_{\bG_a^0}(X_a,\C))^{\bG_a/\bG_a^0}$ is given by the
homomorphism $H^*_\bT(pt)\simeq H^*_\bG(X)\to H^*_{\bG_a}(X_a)$ induced by the embedding $X_a\to X$.
The commutative square of morphisms of quotient stacks
\begin{diagram}
[X_a/\bG_a] &\rTo{}& [X/\bG] \\
\dTo{}&&\dTo{}\\
[pt/\bG_a] &\rTo{}& [pt/\bG]
\end{diagram}
leads to a commutative square of homomorphisms
\begin{diagram}
H^*_{\bG_a}(X_a) &\lTo{}& H^*_\bG(X)\\
\uTo{}&&\uTo{}\\
H^*_{\bG_a}(pt)&\lTo{}& H^*_\bG(pt)
\end{diagram}
This reduces us to checking that the natural $H^*_\bG(pt)$-action on $H^*_\bG(X)$ is given by the
homomorphism $H^*_\bG(pt)\to H^*_\bT(pt)$ via the isomorphism \eqref{Spr-res-coh-eq}.
To this end we note that $H^*_\bG(X)=H^*_{\bG\times \bB}(\bG)$ and we have to calculate the
$H^*_{\bG\times \bB}(pt)$-action on the latter space. Finally, we observe that $\bG=(\bG\times \bB)/\De(\bB)$,
where $\De(\bB)$ is the diagonal copy of $\bB$. Hence, there is an isomorphism
$$H^*_{\bG\times \bB}(\bG)\wt{\to} H^*_{\De(\bB)}(pt)$$
compatible with the restriction homomoprhism $H^*_{\bG\times \bB}(pt)\to H^*_{\De(\bB)}(pt)$.
Thus, the $H^*_{\bG\times \bB}(pt)$-action on $H^*_{\bG\times \bB}(\bG)$ factors through the latter homomorphism.
\ed

\begin{lem}\label{Spr-classical-W-lem}
(i) For a subgroup $K\sub \bG_a$ the restriction homomoprhism $H^*_K(X,\C)\to H^*_K(X_a,\C)$ is
compatible with the Springer $W$-actions.

\noindent
(ii) The natural isomorphism
$$H^*_\bT(\bG/\bB,\C)\rTo{\sim} H^*_\bT(\bG/\bT,\C)$$
is compatible with $W$-actions, where $W$ acts on $H^*_\bT(\bG/\bB,\C)$ by Springer construction and
it acts on $H^*_\bT(\bG/\bT,\C)$ via the geometric action on $\bG/\bT$.
\end{lem}

\Pf . (i) We have a natural isomorphism 
\begin{equation}\label{Spr-global-sections-eq}
H^*_K(X,\C)\rTo{\sim} H^*_K(\wt{\fg},\C)\simeq H^*_K(\fg,R\pi_{\fg,*}\C),
\end{equation}
where the first arrow is the pull-back with respect to the projection $\wt{\fg}\to X$.
It is easy to see that the restriction homomorphism $H^*_K(X,\C)\to H^*_K(X_a,\C)$
is equal to the composition of \eqref{Spr-global-sections-eq} with the natural map
\begin{equation}\label{Spr-global-res-map-eq}
H^*_K(\fg,R\pi_{\fg,*}\C)\to i_a^*R\pi_*\C\simeq H^*_K(X_a,\C),
\end{equation}
where $i_a:\{a\}\hookrightarrow \Nscr$ is the inclusion map.
In the case $a=0$ this shows that the map \eqref{Spr-global-sections-eq} is $W$-equivariant
(cf.\ \cite[Lem.\ 13.6]{Jantzen} for a similar argument).
Since the map \eqref{Spr-global-res-map-eq} is also $W$-equivariant, the assertion follows.

\noindent
(ii) 
The same statement about non-equivariant cohomology is proved in \cite[Prop.\ 13.7]{Jantzen}.
Recall that the proof is based on the commutative diagram
\begin{diagram}
\bG/\bT &\rTo{\iota}& \wt{\Obscr}&\rTo{}&\fg'_{rs}&\rTo{}&\wt{\fg}&\rTo{f}&\bG/\bB\\
&\rdTo{}&\dTo{}&&\dTo{\pi'_{rs}}&&\dTo{\pi_{\fg}}\\
&&\Obscr&\rTo{}&\fg_{rs}&\rTo{}&\fg
\end{diagram}
where $\Obscr\sub\fg_{rs}$ is the $\bG$-orbit of some semisimple element in $\ft$,
$\fg'_{rs}\sub \bG/\bT\times \fg_{rs}$ is the subvariety of $(gT,x)$ such that $x\in \Ad(g)(\ft)$,
$\wt{\Obscr}\sub\fg'_{rs}$ is the preimage of $\Obscr$ under the natural projection $\pi'_{rs}$.
The map $\iota$ is induced by the action of $\bG$ on $\fg'_{rs}$.
The only additional point that we have to make is that all the varieties in the above diagram are equipped with the action of $\bG$ (and hence
of $\bT\sub \bG$) and all the maps are $\bG$-equivariant. Now the argument of \cite[Prop.\ 13.7]{Jantzen}
can be repeated literally using the $\bT$-equivariant cohomology. 
\ed

\begin{rem} The assertion of Lemma \ref{Spr-classical-W-lem} for $K=\{1\}$ is proved also in
\cite[Thm.\ 1.1]{HS} using the original definition of the $W$-action due to Springer. We use Lusztig's definition
of the $W$-action which differs from the original one by the twist with the sign character of $W$
(see \cite{Hotta}).
\end{rem}

%\subsection{Cohomology of Springer fibers}

\subsection{Construction of modules using Springer correspondence}\label{modules-constr-sec}

Let 
$$A^{dg}_W:=R\Hom_\bG^\bullet(\Ascr,\Ascr)$$
be the DG-algebra of endomorphisms of $\Ascr$, where we use the DG-enhancement constructed in Appendix \ref{dg-sec}.
Furthermore, we can assume that $A^{dg}_W$ is equipped with a DG-endomorphism, compatible with the action of
the Frobenius endomorphism on the $\Ext^*$-algebra computed over an algebraic extension of a finite field
(see Theorem \ref{anr}).

\begin{thm}\label{A-W-purity-thm} 
The DG-algebra $A^{dg}_W$ is formal, so it is quasi-isomorphic
to $A_W$.
\end{thm}

\Pf . By Theorem \ref{ref-1.1-1} from Appendix \ref{formality-sec}, 
it is enough to verify the purity of the action of Frobenius on $A_W$.
But the elements of $W$ correspond to the endomorphisms of $\Ascr$ defined over a finite field,
so they commute with the Frobenius, while the algebra $H^*_\bT(pt)\simeq H^*_\bG(\bG/\bB)$ is generated
by the 1st Chern classes of line bundles defined over a finite field, which are pure classes in $H^2$.
\ed
%the action of $S(\ft^*)\simeq H^*_\bG(X)$ comes from \eqref{Spr-res-coh-eq}. 
%(see \cite{Rider} or Section \ref{formality-sec}), 

Recall that we denote by $D(A_W-\dgmod)$ the category of DG-modules over $A_W$, viewed as a DG-algebra
with zero differential. Let $\Perf(A_W-\dgmod)\sub D(A_W-\dgmod)$ denote the subcategory of 
perfect DG-modules. Consider the duality functor
\begin{equation}\label{main-duality-functor-eq}
\Du:D(A_W-\dgmod)^{\circ}\to D(A_W-\dgmod): M\mapsto R\Hom_{H^*_\bG(pt)}(M, H^*_\bG(pt)),
\end{equation}
where we view $H^*_\bG(pt)=H^*_\bT(pt)^W$ as a central subalgebra of $A_W$. The derived 
functor $R\Hom$ here can be defined using the bar-resolution as in \cite[Sec.\ 10.12]{BerLun}.
Note that for an $A_W$-module $M$ the space $\Hom_{H^*_\bG(pt)}(M, H^*_\bG(pt))$ is equipped with a
right $A_W$-module structure. We convert it into a left $A_W$-module structure using the
isomorphism $A_W^{\circ}\to A_W$ sending $w$ to $w^{-1}$ and identical on $H^*_\bT(pt)$.
It is easy to see that
the restriction of the functor \eqref{main-duality-functor-eq} to $\Perf(A_W-\dgmod)$ is an involution.

We have a natural contravariant functor
\begin{equation}\label{Springer-functor}
R\Hom_\bG(?,\Ascr): D_\bG(\Nscr)^{\circ}\to D(A^{dg}_W-\dgmod).
\end{equation}
Let $D_{\bG,\Spr}(\Nscr)\sub D_{\bG,c}(\Nscr)$ be the thick subcategory generated by 
$\Ascr$. Then the functor \eqref{Springer-functor} induces an equivalence of
$D_{\bG,\Spr}(\Nscr)^{\circ}$ with the subcategory of $D(A^{dg}_W-\dgmod)$ generated by $A^{dg}_W$.
But by Theorem \ref{A-W-purity-thm}, the DG-algebra is quasi-isomorphic to $A_W$. 
Thus, \eqref{Springer-functor} induces an equivalence
\begin{equation}\label{Springer-equivalence}
\Phi:D_{\bG,\Spr}(\Nscr)^{\circ}\rTo{\sim} \Perf(A_W-\dgmod): F\mapsto R\Hom_\bG(F,\Ascr).
\end{equation}
Combining it with the duality \eqref{main-duality-functor-eq} we also get an equivalence
$$\Du\Phi: D_{\bG,\Spr}(\Nscr)\rTo{\sim} \Perf(A_W-\dgmod).$$
Note that a similar equivalence is constructed in \cite{Rider} in a different way.

\begin{prop}\label{appear-prop} 
(i) For any $(O_a,\xi)$ appearing in the Springer correspondence,
one has $j_{a,!}\Lscr_{\xi}\in D_{\bG,\Spr}(\Nscr)$. 

\noindent
(ii) Let $S$ be the set of nilpotent orbits forming a stratification of a closed $\bG$-invariant subset in $\Nscr$.
Let $\Spr_S$ be the set of all $(a,\xi)$ appearing in the Springer correspondence 
such that $O_a\in S$, and let $D_{\bG,\Spr_S}(\Nscr)$ be the triangulated subcategory of $D_\bG(\Nscr)$
generated by the simple perverse sheaves associated with $\Spr_S$.
Then the objects $(j_{a,!}\Lscr_\xi)_{(a,\xi)\in\Spr_S}$ generate $D_{\bG,\Spr_S}(\Nscr)$.
\end{prop}

\Pf . (i) This is deduced in \cite{Kato} (see Claim B on p.19 and Theorem 3.1.8) from the orthogonality
relation \cite[24.8c]{Lusztig-char-shV} for perverse sheaves arising in the generalized Springer correspondence
of \cite{Lusztig-int-coh} from different cuspidal data.

\noindent
(ii) If $S$ consists of a single closed nilpotent orbit $O_a$ then we have 
$$j_{a,!}\Lscr_\xi[\dim O_a]=j_{a,!*}\Lscr_\xi[\dim O_a]=S_{a,\xi},$$
so the assertion holds by the definition of $D_{\bG,\Spr_S}(\Nscr)$.
For general $S$ we can assume by induction that the assertion holds for some $S'=S\setminus O_a$,
such that the orbits corresponding to $S'$ still form a stratification of a closed subset.
By part (i), $S_{a,\xi}$ lies in the subcategory generated by $j_{a,!}\Lscr_\xi[\dim O_a]$ and
$D_{\bG,\Spr_{S'}}(\Nscr)$. Hence, the assertion for $S$ follows from the assertion for $S'$.
\ed

\begin{defi}\label{modules-Springer-defi}
With each $\chi\in \Irr W$ we associate DG-modules over $A^{dg}_W$ by setting
$$P^{dg}_\chi=P^{dg}_{a,\xi}:=R\Hom_\bG(S_{a,\xi},\Ascr),$$
%$$Q^{dg}_\chi=Q^{dg}_{a,\xi}:=\Phi(S_{a,\xi})=\Du (P^{dg}_{a,\xi}),$$
$$M^{dg}_\chi=M^{dg}_{a,\xi}:=R\Hom_\bG(j_{a,!}\Lscr_\xi[\dim O_a],\Ascr),$$
$$N^{dg}_\chi=N^{dg}_{a,\xi}:=\Du(M^{dg}_{a,\xi})[-\dim \Nscr],$$
where $(O_a,\xi)$ is the pair associated with $\chi$ by the Springer correspondence,
$S_{a,\xi}=j_{a,!*}\Lscr_\xi[\dim O_a]$ is the corresponding irreducible perverse sheaf. 
We also consider the corresponding graded $A_W$-modules 
$$P_\chi=P_{a,\xi}:=H^*P^{dg}_{a,\xi},
\ \ M_\chi=M_{a,\xi}:=H^*M^{dg}_{a,\xi}, \ \ N_\chi=N_{a,\xi}:=H^*N^{dg}_{a,\xi}.$$
\end{defi}

From the decomposition of $\Ascr$ we easily get the formula
$$P_\chi=\chi\ot H^*_\bT(pt)$$ 
which is the indecomposable projective $A_W$-module associated with $\chi$.

Recall that $\bG^0_a\sub \bG_a$ is the connected component of $1$ in the stabilizer subgroup of $a\in\Nscr$. 
%For an irreducible representation $\xi$ of $G_a/\bG^0_a$ set
%$$H_{a,\xi}:=\left(\xi^\vee\ot H^*_{\bG^0_a}(X_a,D_{X_a})\right)^{G_a/\bG^0_a}[?],$$
%where $D_Y$ denotes the Verdier dualizing complex on $Y$. 

\begin{prop}\label{modules-Springer-prop}  
(i) The $H^*_\bG(pt)$-action on $M_{a,\xi}$ and $N_{a,\xi}$ (induced by the embedding $H^*_\bG(pt)\sub A_W$) factors through
the natural homomorphism 
\begin{equation}\label{H-G-Ga-hom-eq}
H^*_\bG(pt)\to H^*_{\bG_a}(pt)=H^*_{\bG_a^0}(pt)^{\bG_a/\bG^0_a}, 
\end{equation}
and one has natural isomorphisms of $H^*_{\bG_a}(pt)$-modules
\begin{equation}\label{RHom-A-j*-eq}
\begin{array}{l}
M_{a,\xi}\simeq \left(\xi^\vee\ot H^*_{\bG^0_a}(X_a,D_{X_a}[-2\dim X_a])\right)^{\bG_a/\bG^0_a},\\
N_{a,\xi}\simeq \left(\xi\ot H^*_{\bG^0_a}(X_a,\C)\right)^{\bG_a/\bG^0_a}.
\end{array}
\end{equation}
The second isomorphism is compatible with the $A_W$-action,
where the action of $A_W$ on the right is induced by the Springer $W$-action and by the restriction homomorphism
$$H^*_\bT(pt)\simeq H^*_\bG(X)\to H^*_{\bG_a}(X_a,\C).$$
%The induced $H^*_G(pt)$-module structures on the right are induced by the homomorphism
%$$H^*_G(pt)\to H^*_{G_a}(pt)=H^*_{G_a^0}(pt)^{G_a/\bG^0_a}.$$ 
%and by the natural action of $H^*_{G_a}(pt)$ on the equivariant cohomology.

\noindent
(ii) The spaces $H^*_{\bG^0_a}(X_a,\C)$ and $H^*_{\bG^0_a}(X_a,D_{X_a})$ are free $H^*_{\bG^0_a}(pt)$-modules,
and the Frobenius action on $H^*_{\bG^0_a}(X_a,\C)$ and $H^*_{\bG^0_a}(X_a,D_{X_a})$ is pure.

\noindent 
(iii) The DG-modules $P^{dg}_{a,\xi}$, $M^{dg}_{a,\xi}$ and $N^{dg}_{a,\xi}$ are formal.

\noindent
(iv) $M_{a,\xi}$ and $N_{a,\xi}$ are maximal Cohen-Macauley $H^*_{\bG_a}(pt)$-modules. 
In particular, their support as $H^*_{\bG}(pt)$-modules is the image of the finite morphism 
$\Spec(H^*_{\bG_a}(pt))\to \Spec(H^*_{\bG}(pt))$.
If in addition, the action of $\bG_a/\bG_a^0$ on the affine space
$\Spec H^*_{\bG_a^0}(pt)$ is generated by (pseudo)reflections, then $M_{a,\xi}$ and $N_{a,\xi}$ are projective modules over 
$H^*_{\bG_a}(pt)$.
\end{prop}

\Pf . (i) First, we note that by Lemma \ref{two-actions-lem}, the $H^*_\bG(pt)$-module structure on $M_{a,\xi}$ and 
$N_{a,\xi}$ corresponds via the isomorphisms \eqref{RHom-A-j*-eq} to the action
induced by the homomorphism \eqref{H-G-Ga-hom-eq}
and by the natural action of $H^*_{\bG_a}(pt)$ on the equivariant cohomology. This implies the first assertion.

Let us write $j_a=j$ for brevity. We have 
$$R\Hom_\bG(j_!\Lscr_\xi,\Ascr)= R\Hom_\bG(j_!\Lscr_\xi,R\pi_*\und{\C}_{\wt{\Nscr}}[\dim \Nscr])\simeq
R\Hom_\bG(\Lscr_\xi,j^!\pi_*D_{\wt{\Nscr}}[-\dim \Nscr]).$$
Using the base change, we can rewrite this as
$$R\Hom_\bG(\Lscr_\xi,\pi_{O_a,*}D_{\pi^{-1}(O_a)}[-\dim \Nscr])\simeq 
R\Ga_\bG(\pi^{-1}(O_a),\pi_{O_a}^*\Lscr_\xi^\vee\ot D_{\pi^{-1}(O_a)}[-\dim \Nscr]),$$
where $\pi_{O_a}:\pi^{-1}(O_a)\to O_a$ is the restriction of $\pi$.
Now using the isomorphism \eqref{O-a-isom}
we get an isomorphism
\begin{align*}
&M_{a,\xi}=\Ext^*_\bG(j_!\Lscr_\xi,\Ascr)[-\dim O_a]\simeq 
H^*_\bG(\pi^{-1}(O_a),\pi_{O_a}^*\Lscr_\xi^\vee\ot D_{\pi^{-1}(O_a)}[-\dim \Nscr-\dim O_a])\simeq\\
&H^*_{\bG_a}(X_a, \xi^\vee\ot D_{X_a}[-\dim \Nscr+\dim O_a])\simeq 
(\xi^\vee \ot H^*_{\bG_a^0}(X_a, D_{X_a}[-2\dim X_a]))^{\bG_a/\bG_a^0},
\end{align*}
where for the last isomorphism we used  Lemma \ref{G/H-lem}(ii) and the equality \eqref{dim-Xa-eq}.

Now let us calculate $N_{a,\xi}$. First, we observe that the object $\Ascr=R\pi_*\und{\C}_{\ti{\Nscr}}[\dim\Nscr]$
is self-dual, i.e., we have an isomorphism $\Du(\Ascr)\simeq\Ascr$. 
Next, we recall that for any constructible sheaf $F$ we have an isomorphism
%the equivariant Verdier duality gives an isomorphism
$$\Du(F\ot \Ascr)\simeq R\und{\Hom}(F,\Du(\Ascr)).$$
Combining this with the self-duality of $\Ascr$ we get for $F=j_!\Lscr_\xi$,
$$\Du(j_!\Lscr_\xi\ot\Ascr)\simeq R\und{\Hom}(j_!\Lscr_\xi,\Ascr).$$
Thus, we get
$$\Du R\Hom(j_!\Lscr_\xi,\Ascr)\simeq \Du R\Ga_\bG (\Nscr, R\und{\Hom}(j_!\Lscr_\xi,\Ascr))\simeq
R\Ga_{c,\bG}(\Nscr, j_!\Lscr_\xi\ot\Ascr).$$
Using the projection formula we can rewrite this as
$$R\Ga_{c,\bG}(\Nscr, j_!(\Lscr_\xi\ot j^*\Ascr))\simeq R\Ga_{c,\bG}(O_a, \Lscr_\xi\ot \pi_{O_a,*}\und{\C}[\dim\Nscr])\simeq
R\Ga_{c,\bG}(\pi^{-1}(O_a),\pi_{O_a}^*\Lscr_\xi[\dim\Nscr]),$$
hence, we get
$$N_{a,\xi}\simeq H^*_{c,\bG}(\pi^{-1}(O_a),\pi_{O_a}^*\Lscr_\xi)\simeq 
(\xi\ot H^*_{\bG^0_a}(X_a,\C))^{\bG_a/\bG_a^0}.$$

The compatibility of the constructed isomorphism with the
action of $A_W=\Ext^*_\bG(\Ascr,\Ascr)$ follows from the functoriality of the isomorphisms
we used together with the fact that the involution
\begin{equation}\label{duality-involution-eq}
\Ext^*_\bG(\Ascr,\Ascr)\to\Ext^*_\bG(\Du(\Ascr),\Du(\Ascr))\simeq \Ext^*_\bG(\Ascr,\Ascr),
\end{equation}
coming from the self-duality of $\Ascr$,
sends $w\in W$ to $w^{-1}$ and is the identity on $H^*_\bT(pt)\simeq H^*_\bG(X)$. Indeed,
the fact about the $W$-action follows from the similar compatibility on the regular semisimple locus, where it can be
easily checked.

%The last assertion follows from Lemma \ref{two-actions-lem}.

\noindent
(ii) The cohomology of $X_a$ is concentrated in even degrees, hence 
$X_a$ is $\bG_a^0$-equivariantly formal (see \S\ref{equiv-sh-sec}), 
%by the degenerate Leray-Serre spectral sequence
so we have
$$H^*_{\bG_a^0}(X_a,\C)\simeq B_a\ot H^*(X_a,\C),$$
where $B_a=H^*_{\bG_a^0}(pt)$.
In particular, $H^*_{\bG_a^0}(X_a,\C)$ is a free $B_a$-module,
and the action of the Frobenius on it is pure, by the result of \cite{Springer}.
By Theorem \ref{ref-1.1-1}, this implies the formality of $R\Ga_{\bG_a^0}(X_a,\C)$ as an object of
$D(dg-B_a-\mod)$. By the equivariant Verdier duality on $X_a$ (see Section \ref{equiv-sh-sec}), we deduce
$$R\Ga_{\bG_a^0}(X_a,D_{X_a})\simeq R\Hom_{B_a}(H^*_{\bG_a^0}(X_a,\C),B_a)\simeq
B_a\ot H^*(X_a,\C)^*,$$
which implies our assertions about $H^*_{\bG_a^0}(X_a,D_{X_a})$.

\noindent
(iii) This follows from Theorem \ref{ref-1.1-1}. The required purity for $P_{a,\xi}$ follows from the proof of Theorem 
\ref{A-W-purity-thm}, and for $M_{a,\xi}$ and $N_{a,\xi}$---from part (ii). 

\noindent
(iv) These statements follow from the fact that
$H^*_{\bG_a^0}(X_a,\C)$ and $H^*_{\bG_a^0}(X_a,D_{X_a})$ are free $H^*_{\bG_a^0}(pt)$-modules (by part (ii)).
Namely, to derive that $M_{a,\xi}$ and $N_{a,\xi}$ are Cohen-Macauley we use the fact that $H^*_{\bG_a^0}(pt)$ is
a polynomial ring together with the fact that the modules of covariants for finite groups are Cohen-Macauley
(see \cite[Thm.\ 4.1]{VdB}, \cite[Sec.\ 5.1]{Popov}). The assumption that $\bG_a/\bG_a^0$ acts on 
$\Spec H^*_{\bG_a^0}(pt)$ as a complex reflection group implies that $H^*_{\bG_a^0}(pt)$ is a projective 
module over $H^*_{\bG_a}(pt)=H^*_{\bG_a^0}(pt)^{\bG_a/\bG_a^0}$, hence, the same is true about the modules of covariants.
\ed

%We have
%$$H^*_{G_a^0}(X_a, D_{X_a})\simeq H^*D_{G_a^0} R\Ga_{G_a^0}(X_a, \Q),$$
%where $D_{G_a^0}$ is the
%Verdier duality functor on $D_{G_a^0}(pt)$. According to \cite[Thm.\ 12.7.2]{BerLun}, the category
%$D^b_{G_a^0}(pt)$ is equivalent to the derived category of dg-modules over the
%dg-algebra $A_{\bG^0_a}=H^*_{\bG^0_a}(pt)$ (with zero differential), and the Verdier duality functor
%corresponds to the usual duality $M\mapsto R\Hom(M,A_{\bG^0_a})$.

%We have an identification of $H^*_\bG(pt)$ with the central subalgebra of $A_W$ via
%the embeddings $H^*_\bG(pt)\sub H^*_\bT(pt)\sub A_W$. 
Let us consider the duality functor
$$D(A_W-\mod)^{\circ}\to D(A_W-\mod): M\mapsto R\Hom_{H^*_\bG(pt)}(M,H^*_\bG(pt)),$$
which is defined in the same way as \eqref{main-duality-functor-eq}, but on the derived category of ungraded $A_W$-modules.
%where we use the natural right $A_W$-module structure on $R\Hom_{H^*_\bG(pt)}(?,H^*_\bG(pt))$
%and turn it into a left $A_W$-module structure using the anti-involution of $A_W$ sending $w$ to $w^{-1}$ 
%and trivival on $H^*_\bT(pt)$. 

Note that if in addition $M$ is a graded module then
$R\Hom_{H^*_\bG(pt)}(M,H^*_\bG(pt))$ can be also defined as an object of the derived category of graded
$A_W$-modules.

\begin{cor}\label{CM-duality-cor} 
The $A_W$-modules $M_{a,\xi}$ and $N_{a,\xi}$, viewed as $H^*_\bG(pt)$-modules, are Cohen-Macauley.
One has
$$R\Hom_{H^*_\bG(pt)}(M_{a,\xi},H^*_\bG(pt))\simeq N_{a,\xi}(-i)[i],$$
where $i=\rk \bG-\rk \bG^0_a$ (where $\rk \bG^0_a$ is the rank of the reductive part of $\bG^0_a$).
\end{cor}

\Pf . First, we claim that $H^*_{\bG_a}(pt)$ is finitely generated as $H^*_\bG(pt)$-module.
Indeed, let $\bT_a$ be a maximal torus in $\bG^0_a$, and let $\bT$ be a maximal torus in $\bG$, such that $\bT\supset\bT_a$.
Note that $H^*_{\bT_a}(pt)$ is a finitely generated $H^*_\bG(pt)$-module, as a quotient of $H^*_\bT(pt)$. 
Since $H^*_\bG(pt)$ is Noetherian, our claim follows from the embedding 
$$H^*_{\bG_a}(pt)\sub H^*_{\bG_a^0}(pt)\sub H^*_{\bT_a}(pt).$$
By Proposition \ref{modules-Springer-prop}(iv), this implies that $M_{a,\xi}$ and $N_{a,\xi}$ are Cohen-Macauley
$H^*_\bG(pt)$-modules.

Next, consider the functor $F:D(A_W-\grmod)\to D(A_W-\dgmod)$ (see \S\ref{graded-dg-sec}).
For any perfect complex $M\in D(A_W-\grmod)$ we have a natural isomorphism
$$F\left(R\Hom_{H^*_\bG(pt)}(M,H^*_\bG(pt))\right)\simeq R\Hom_{A_W}(FM,H^*_\bG(pt))\simeq \Du(FM),$$
where $\Du$ is the duality \eqref{main-duality-functor-eq} on the derived category of DG-modules over $A_W$.
It remains to apply this to $M=M_{a,\xi}$ and to use the fact that it is Cohen-Macauley over $H^*_\bG(pt)$ with
support of dimension $\dim \bT_a$, so that $R\Hom_{H^*_\bG(pt)}(M_{a,\xi},H^*_\bG(pt))$ is concentrated in 
degree $i=\rk \bG-\dim \bG^0_a$.
\ed

\begin{ex} For $a=0$ we have $\bG_a=\bG^0_a=\bG$, so the component group $\bG_a/\bG^0_a$ is trivial.
The representation of $W$ associated with $a=0$ is given by the Springer action on the top degree cohomology
of the flag variety $X$, which is the sign representation $\de$. Hence, we have
$$M_0=P_0\simeq \de\ot H^*_\bG(X,\C)\simeq \de\ot H^*_\bT(pt).$$
On the other hand,
$$N_0=H^*_\bG(X,\C)\simeq H^*_\bT(pt).$$
Thus, the duality between $M_0$ and $N_0$ corresponds to an isomorphism of $A_W$-modules
$$R\Hom_{H^*_\bG(pt)}(H^*_\bT(pt),H^*_\bG(pt))\simeq \de\ot H^*_\bT(pt).$$
\end{ex}

%Concrete identification in terms of Levi...

%In nice cases the strata for the action of $G_a/\bG^0_a$ on $\Spec H^*_{\bG^0_a}(pt)^{G_a/\bG^0_a}$
%should go to the strata in $\Spec H^*_G(pt)$. Below we will consider what happens for classical types...

\subsection{Semiorthogonal decomposition}\label{Springer-semiorth-sec}

By Proposition \ref{modules-Springer-prop}, the DG-modules $M^{dg}_{a,\xi}$ 
over $A^{dg}_W\simeq A_W$ are formal.
The Springer correspondence gives us information about morphisms between these modules in
the category $D(A_W-\dgmod)$. Using Theorem \ref{dgextth} we can convert it into the information on morphisms
in $D(A_W-\mod)$.

\begin{prop}\label{Ext-calculation-prop} 
(i) Assume that $O_{a'}$ is not contained in the closure of $O_a$. Then for any
 $\Lscr_\xi$ on $O_a$ (resp., $\Lscr_{\xi'}$ on $O_{a'}$) one has 
 $$\Ext^*_{A_W}(M_{a,\xi},M_{a',\xi'})=\Ext^*_{A_W}(P_{a,\xi},M_{a',\xi'})=0.$$
 
\noindent
(ii) For fixed $a$ and for $\xi\in \Irr \bG_a/\bG^0_a$ there is a natural morphism
$P_{a,\xi}\to M_{a,\xi}$ such that for any $\xi'\in \Irr \bG_a/\bG^0_a$
the induced map
$$\Ext^*_{A_W}(M_{a,\xi},M_{a,\xi'})\to\Ext^*_{A_W}(P_{a,\xi},M_{a,\xi'})=\Hom_{A_W}(P_{a,\xi},M_{a,\xi'})$$
is an isomorphism.
In particular, $\Ext^{>0}_{A_W}(M_{a,\xi},M_{a,\xi'})=0$.
Furthermore, we have a natural isomorphism of algebras
$$\bigoplus_{\xi,\xi'}\Hom_{A_W}(M_{a,\xi},M_{a,\xi'})\simeq
\bigoplus_{\xi,\xi'}\Hom_{\bG_a/\bG^0_a\ltimes B_a}(\xi'\ot B_a, \xi\ot B_a),$$
where $\xi,\xi'$ run through irreducible representations of $\bG_a/\bG^0_a$ appearing in
the Springer correspondence, and $B_a=H^*_{\bG^0_a}(pt)$.
\end{prop}

\Pf . (i) Since $O_{a'}$ is not contained 
in the closure of $O_a$, we have 
$$\Ext^*_\bG(j_{a',!}\Lscr_{\xi'},j_{a,!}\Lscr_\xi)=\Ext^*_\bG(\Lscr_{\xi'},j^!_{a'}j_{a,!}\Lscr_\xi)=0.$$
By the equivalence \eqref{Springer-equivalence}, 
this implies the vanishing
$$\Ext^*_{A_W-\dgmod}(M^{dg}_{a,\xi},M^{dg}_{a',\xi})=0.$$
By Theorem \ref{dgextth}, this implies the same semiorthogonality for $A_W$-modules
$M_{a,\xi}$ and $M_{a',\xi'}$. The proof of the second vanishing is analogous using the vanishing
$j^!_{a'}j_{a,!*}=0$.

\noindent
(ii) We have a morphism $P_{a,\xi}\to M_{a,\xi}$ induced by the natural morphism
of sheaves $j_{a,!}\Lscr_\xi\to j_{a,!*}\Lscr_\xi$. Furthermore, we have the following commutative
diagram in which the vertical arrows are induced by this morphism of sheaves and the
horizontal arrows are adjunction isomorphisms: 
\begin{diagram}
\Ext^*_\bG(j_{a,!}\Lscr_{\xi'},j_{a,!}\Lscr_\xi)&\rTo{\sim}&\Ext^*_\bG(\Lscr_{\xi'},j^!_aj_{a,!}\Lscr_\xi)\\
\dTo{}&&\dTo{\sim}\\
\Ext^*_\bG(j_{a,!}\Lscr_{\xi'},j_{a,!*}\Lscr_{\xi})&\rTo{\sim}&\Ext^*_\bG(\Lscr_{\xi'},j^!_aj_{a,!*}\Lscr_\xi)
\end{diagram}
Since $j^!_aj_{a,!*}\Lscr_\xi\simeq j^!_aj_{a,!}\Lscr_\xi$, 
by the equivalence \eqref{Springer-equivalence}, we deduce that the map
$$\Ext^*_{A_W-\dgmod}(M^{dg}_{a,\xi},M^{dg}_{a',\xi})\to \Ext^*_{A_W-\dgmod}(P^{dg}_{a,\xi},M^{dg}_{a',\xi}),$$
induced by the morphism $P^{dg}_{a,\xi}\to M^{dg}_{a,\xi}$, is an isomorphism. By the formality of these
modules (see Proposition \ref{modules-Springer-prop}) and by Theorem \ref{dgextth}, 
we deduce that the similar map
of $\Ext^*$-spaces in the category $D(A_W-\mod)$ is an isomorphism. In particular, we deduce the vanishing
of $\Ext^{>0}_{A_W}(M_{a,\xi},M_{a,\xi'})$.
Now the last assertion follows from the isomorphism
\begin{align*}
&\Ext^*_\bG(j_{a,!}\Lscr_{\xi'},j_{a,!}\Lscr_{\xi})=\Ext^*_\bG(\Lscr_{\xi'},j^!_{a}j_{a,!}\Lscr_\xi)=\\
&\Ext^*_\bG(\Lscr_{\xi'},\Lscr_{\xi})=H^*_{\bG_a}(pt, (\xi')^\vee\ot \xi)=
\Hom_{\bG_a/\bG_a^0}(\xi', \xi\ot H^*_{\bG_a^0}(pt)),
\end{align*}
where in the last equality we used Lemma \ref{G/H-lem}(ii).
%Reduce to studying the equivariant modules for the action of $G_a/\bG^0_a$ on $\Spec H^*_{\bG^0_a}(pt)$...
\ed

To proceed with constructing a semiorthogonal decomposition we need the following result based
on some information about irreducible representations of
the component groups $\bG_a/\bG_a^0$ appearing in the Springer representation.
We also use the classification of nilpotent orbits for exceptional groups and the calculation of their
centralizers (see \cite[Ch.\ 22]{LS}). We specify the type of a nilpotent orbit by the class of the corresponding distinguished
nilpotent in a Levi subgroup, as in \cite[Sec.\ 3.3.6]{LS}.

\begin{lem}\label{E-a-lem} 
For every nilpotent $a\in\fg$ the algebra
\begin{equation}\label{Ea-eq}
E_a:=\End_{\bG_a/\bG_a^0\ltimes B_a}(\bigoplus_{\xi: (a,\xi)\in\Spr}\xi\ot B_a),
\end{equation}
where $B_a:=H^*_{\bG^0_a}(pt)$, has finite global dimension.
\end{lem}

\Pf . Let us write $E_a=E_a(\bG)$ to stress the dependence on $\bG$.
Note that if $\bG$ is the quotient of $\wt{\bG}$ by a finite central subgroup 
then we have an isomorphism of algebras $E_a(\bG)\simeq E_a(\wt{\bG})$. Indeed, in this case
the natural map $H^*_{\wt{\bG}^0_a}(pt)\to H^*_{\bG^0_a}(pt)$ is an isomorphism, 
the corresponding homomorphism of the component groups $\rho:\wt{\bG}_a/\wt{\bG}^0_a\to \bG_a/\bG^0_a$ is surjective, and
the characters of $\wt{\bG}_a/\wt{\bG}^0_a$ appearing in the Springer correspondence for $\wt{\bG}$ are obtained by
composing the similar characters of $\bG_a/\bG^0_a$ with $\rho$.
On the other hand, if $\bG$ is the quotient of $\wt{\bG}$ by a connected subgroup of the center then
$E_a(\wt{\bG})\simeq E_a(\bG)[t_1,\ldots,t_k]$. Thus, it is enough to prove the assertion in the case when $\bG$ is almost simple.
Furthermore, we can replace $\bG$ by any isogenous group.

If $\bG_a$ is connected then the algebra $E_a$ is isomorphic to $B_a$ which is an algebra of polynomials,
so it has finite global dimension. In particular, this is the case for all nilpotent orbits in type $A$ (for the adjoint group). 

On the other hand, if for some nilpotent orbit $O_a$ all irreducible representations
of $\bG_a/\bG^0_a$ appear in the Springer correspondence then 
$$E_a=\End_{\bG_a/\bG^0_a\ltimes B_a}(P_a), \ \text{ where }\ P_a=\bigoplus _{\xi\in\Irr(\bG_a/\bG^0_a)} \xi\ot B_a.$$
Since $P_a$ is a projective generator of the category of $\bG_a/\bG^0_a\ltimes B_a$ and since
$\bG_a/\bG^0_a\ltimes B_a$ has finite global global dimension, it follows that $E_a$ has finite global dimension.

Now, assume $\bG$ is a symplectic or a special orthogonal group. In this case for any nilpotent $a$ one has 
$\bG_a/\bG_a^0\simeq (\Z/2)^k$.
Using the explicit description of the Springer correspondence in this case (see \cite{Lusztig-int-coh}) 
we see that there exists a surjection
$\pi:\bG_a/\bG_a^0\to F$ such that the action on $B_a$ factors through $F$ and the set characters of $\bG_a/\bG_a^0$
appearing in the Springer representation is a union of cosets for the subgroup 
$\hat{\pi}(\widehat{F})\sub\widehat{\bG_a/\bG_a^0}$
(see \S\ref{BD-sec} for details in the orthogonal case, the symplectic case is similar).
Note that if the characters $\xi$ and $\xi'$ of $\bG_a/\bG_a^0$ belong to different cosets then
the projective modules $\xi\otimes B_a$ and $\xi'\otimes B_a$ 
are mutually orthogonal in $D(\bG_a/\bG_a^0\ltimes B_a)$. Indeed,
this immediately follows from the fact that $\xi$ and $\xi'$ have different restrictions to the subgroup
$\ker(\pi)\sub \bG_a/\bG_a^0$ that acts trivially on $B_a$. Next, for a given coset 
$C=\xi_0+\widehat{\pi}(\widehat{F})$ we have
\begin{equation}\label{typeB-coset-End}
\End\left(\bigoplus_{\xi\in C} \xi\otimes B_a\right)\simeq 
\End\left(\bigoplus_{\xi\in \widehat{\pi}(\widehat{F})} \xi\otimes B_a\right)\simeq
F\ltimes B_a,
\end{equation}
where the first isomorphism is induced by the functor of tensoring with $\xi_0$. The algebra $F\ltimes B_a$
has finite global dimension, so we are done in this case.

To deal with exceptional groups we use the explicit calculations of the Springer correspondence for the corresponding
adjoint groups (see \cite[Sec.\ 13.3]{Carter}).
For types $E_6$ and $E_7$ all irreducible representations of the component groups $\bG_a/\bG_a^0$
appear in the Springer correspondence, so the assertion follows in these cases.
For the remaining types $G_2$, $F_4$ and $E_8$ the only nilpotent orbits for which not all representations
of the component group are present, are the orbits
$G_2(a_1)$, $F_4(a_3)$ and $E_8(a_7)$.
However, in all these cases the group $\bG_a^0$ has trivial reductive part, so that $B_a=k$ and $E_a$
is a semisimple algebra, hence of finite global dimension.
\ed

For a set $S$ of nilpotent orbits forming a stratification of a closed subset in $\Nscr$, let
$\Spr_S$, $D_{\bG,\Spr_S}$ be as in Proposition \ref{appear-prop}(ii).
Let also $\Cscr_S\sub D(A_W-\mod)$ be the thick subcategory generated by the projective modules $P_{a,\xi}$
associated with the elements of $\Spr_S$.

\begin{thm}\label{Springer-thm} 
(i) For each nilpotent orbit $O_a$ let $\Cscr_a\sub D(A_W-\mod)$ be the thick subcategory generated by all the modules
$M_{a,\xi}$, where $(a,\xi)\in\Spr$. Then $\Cscr_a$ is admissible, 
$\Cscr_a\simeq D^f(E_a^{op})$, where $E_a$ is given by \eqref{Ea-eq}, and we have
$$\Ext^*_{A_W}(\Cscr_a,\Cscr_{a'})=0$$
whenever $O_{a'}$ is not in the closure of $O_a$.

\noindent
(ii) Let $S$ be a set of nilpotent orbits forming a stratification of a closed subset in $\Nscr$.
For any ordering $S=\{a_1,\ldots,a_n\}$ such that
$O_{a_i}$ is in the closure of $O_{a_j}$ only if $i<j$, we have a semiorthogonal decomposition 
into admisible subcategories
$$\Cscr_S=\lan \Cscr_{a_n},\ldots, \Cscr_{a_1}\ran.$$
Also, for any $(a,\xi)\in\Spr_S$ we have an exact triangle
\begin{equation}\label{orthogonalizing-triangle-eq} 
K\to P_{a,\xi}\to M_{a,\xi}\to\ldots
\end{equation}
in $D(A_W-\mod)$ with $K\in \Cscr_{S'}$, where $S'$ is the set of orbits in $\ov{O_a}\setminus O_a$.

\noindent
(iii) Let $\{a_1,\ldots,a_r\}$ be an ordering of all nilpotent orbits such that $O_{a_i}$ is in the closure of $O_{a_j}$
only if $i<j$. Then we have a semiorthogonal decomposition into admissible subcategories  
$$D^f(A_W)=\lan  \Dscr_{a_1},\ldots,\Dscr_{a_r}\ran,$$
where $\Dscr_a$ is generated by all the modules $N_{a,\xi}$ with $(a,\xi)\in\Spr$. For each $a$ we have
an equivalence $\Dscr_a\simeq D^f(E_a^{op})$.
\end{thm}

\Pf . (i) Using Lemma \ref{intrinsic}, Proposition \ref{Ext-calculation-prop}(ii) and Lemma
\ref{E-a-lem}, we see that $\Cscr_a$ is admissible and is equivalent to
$D^f(E_a^{op})$. The required semiorthogonalities follow from
Proposition \ref{Ext-calculation-prop}(i).

\noindent (ii)
Using part (i) we see that the subcategory
$$\wt{\Cscr}_S:=\lan \Cscr_{a_n},\ldots, \Cscr_{a_1}\ran\sub D(A_W-\mod)$$ 
is admissible. We have to check the equality of subcategories $\wt{\Cscr}_S=\Cscr_S$. 
Replacing $S$ by the closure of a given orbit $O_a\in S$ we can assume that $S$ is the set of
all orbits in the closure of $O_a$ (so $a=a_n$). By induction, we can assume that the assertion holds
for $S'=S\setminus \{O_a\}$. First, let us check that $P_{a,\xi}$ belongs to $\wt{\Cscr}_S$ for all $\xi$.
Indeed, by Lemma \ref{graded-generation-lem}, it is enough to check the same statement
in the derived category of DG-modules over $A_W$. Using an equivalence \eqref{Springer-equivalence}
this reduces to showing that the object $S_{a,\xi}\in D_\bG(\Nscr)$ belongs to the subcategory generated
by $(j_{a,!}\Lscr_\xi)_{(a,\xi)\in\Spr_S}$. But this follows from Proposition \ref{appear-prop}(ii).

Next, the semiorthogonal decomposition
$$\wt{\Cscr}_S=\lan \Cscr_a, \Cscr_{S'} \ran$$
shows that $\Cscr_{S'}={}^\perp\Cscr_a$, the left orthogonal to $\Cscr_a$ in $\wt{\Cscr}_S$.
We already know that $P_{a,\xi}\in \wt{\Cscr}_S$. Hence, the object $K$ from the exact triangle
\eqref{orthogonalizing-triangle-eq} is also in $\wt{\Cscr}_S$.
Now by Proposition \ref{Ext-calculation-prop}(ii), we have
$$\Hom(K,M_{a,\xi'})=0 \text{ for } (a,\xi')\in\Spr.$$
This implies that $K$ is in ${}^\perp\Cscr_a=\Cscr_{S'}$.

Finally, using the triangles \eqref{orthogonalizing-triangle-eq} we see that the category $\wt{\Cscr}_S$, generated by $\Cscr_{S'}$ and the modules $M_{a,\xi}$,
is equal to $\Cscr_S$.

\noindent (iii) This is obtained from (ii) using the duality $R\Hom_{H^*_\bG(pt)}(?,H^*_\bG(pt))$ 
(see Corollary \ref{CM-duality-cor}).
\ed

\begin{rems} 
1. Kato proved that the natural morphism $P_\chi\to M_\chi$ is surjective and that
$$M_\chi=P_\chi/(\sum_{\chi'<\chi, f:P_{\chi'}\to P_\chi} \Im f)$$
 (see the proof of \cite[Cor.\ 3.6]{Kato}).

\noindent 2. Parts (i) and (ii) of Theorem \ref{Springer-thm} can be deduced without using the formality 
of $A^{dg}_W$ and $M^{dg}_{a,\xi}$, but instead using some results from \cite{JMW} (see \cite[Sec.\ 4]{Achar}).
\end{rems}

%Using the duality \eqref{main-duality-functor-eq}???, we obtain also semiorthogonal decomposition of
%$D^b(A_W)$ with pieces generated by the modules $N_{a,\xi}$ with fixed $a$. We will use these decompositions
%in the case of type $A$, since in this case the latter modules have an explicit description (see Section \ref{KP-sec} below).

\begin{prop}\label{exceptional-prop}
Let $\bG$ be the adjoint simple group of type $G_2$ or $F_4$. Then the semiorthogonal decomposition of Theorem
\ref{Springer-thm} can be further refined to match the motivic decomposition \eqref{mot-dec2-eq}.
\end{prop} 

\Pf . When $\bG_a$ is not connected the subcategory $\Cscr_a$ has a further semiorthogonal decomposition.
In the case when $\bG_a/\bG_a^0$ acts trivially on $B_a$ we just have an orthogonal decomposition numbered
by irreducible representations of $\bG_a/\bG_a^0$. Let us go over other cases.

For type $G_2$ the only orbit that has a nontrivial component group $\bG_a/\bG_a^0$ is the
orbit $G_2(a_1)$ for which $B_a=k$, so that we have further orthogonal decomposition of the corresponding
subcategory $\Cscr_a$ into the pieces numbered by the irreducible representations of $\bG_a/\bG_a^0=S_3$.

For type $F_4$ the only orbits with nontrivial $\bG_a/\bG_a^0$ and nontrivial reductive part $R_a$ of $\bG_a^0$ are
the orbits $\wt{A}_1$, $A_2$, $B_2$ and $C_3(a_1)$. In all of these cases $\bG_a/\bG_a^0=S_2$.
We claim that in all four cases the nontrivial generator
$\bG_a/\bG_a^0$ acts on $\Spec H^*_{\bG_a^0}(pt)=\Spec H^*_{R_a}(pt)$ as a reflection. 
For the orbit $C_3(a_1)$ the group $R_a$ has type $A_1$, so this is automatic.
For the orbits $\wt{A}_1$ and $A_2$ the group $R_a$ has types $A_3$ and $A_2$, respectively.
Note that $\bG_a/\bG_a^0$ acts on $H^*_{\bG_a^0}(pt)$ via the natural homomorphism from $\bG_a/\bG_a^0$
to the group of outer automorphisms of $\bG_a^0$. Hence, it is enough to check that 
the outer automorphism of $R_a$ acts on $\Spec H^*_{R_a}(pt)$ by a reflection. But this
action is induced by the automorphism $x\mapsto -x$ of the Cartan subalgebra, hence it preserves polynomial
invariants of degree $2$ and $4$ and acts by $-1$ on the invariant of degree $3$, and
the assertion follows in this case. 
Finally, for the orbit $B_2$ the group $R_a$ has type $A_1+A_1$, so the
algebra of invariant polynomials is generated by two quadratic polynomials on the Cartan subalgebra $\ft$ of
the Lie algebra of $R_a$.
Since the action of any involution on $\ft$ has at most one $-1$ eigenvalue on $S^2\ft$, our claim follows.

The fact that the pieces of the semiorthogonal decomposition match with those of 
the motivic decomposition \eqref{mot-dec2-eq} follows from Propositions
\ref{HH-semiorth-match-prop} and \ref{star-reflection-prop}.
\ed

\begin{rem}
For $E_6$ we get a semiorthogonal decomposition with one noncommutative block corresponding
to the nilpotent orbit of type $A_2$. Indeed, one can see from the list of nilpotent orbits (see \cite[Ch.\ 22]{LS}) that the only other orbit with nontrivial $\bG_a/\bG_a^0$ and nontrivial reductive part $R_a$ of $\bG_a^0$ is the orbit $D_4(a_1)$, for which
the reductive part of the centralizer is the $2$-dimensional torus, on which $\bG_a/\bG_a^0\simeq S_3$ acts in a standard way.
So we can refine the corresponding piece of the semiorthogonal decomposition of Theorem \ref{Springer-thm} using the
same Theorem for the $S_3$-action. The appearance of the noncommutative block for the nilpotent orbit of type $A_2$
is due to the fact that in this case
the component group $\bG_a/\bG_a^0=\Z/2$ of the centralizer 
swaps the two copies of $A_2$ in the reductive part $A_2+A_2$ of $\bG^0_a$.
To check this one can use the realization of $\bG=E_6$ as the centralizer of a subgroup $S$ of type $A_2$ in $E_8$
(see \cite[Ch.\ 16]{LS}).
Then the reductive part of the centralizer is obtained as $C_R(S)$, where $R$ is the reductive part of the centralizer of our nilpotent element in $E_8$. As proved in \cite[Ch.\ 15]{LS}, $R$ is the semidirect product of $E_6$ with $\Z/2$, acting on 
$E_6$ by the graph automorphism. Now we have to look at the subgroup $A_2$ in $E_6$, invariant under the graph automorphism, take its centralizer and see how
the graph automorphism acts on it. This $A_2$ is generated by $\alpha_2$ and by the maximal root, so the centralizer will
have $A_2$'s generated by $\alpha_1,\alpha_3$ and by $\alpha_5,\alpha_6$. It remains to observe that the graph automorphism swaps these two $A_2$'s.
\end{rem}

\subsection{Equivariant cohomology of the Springer fibers}\label{KP-sec}

In this section we discuss a description of the equivariant cohomology of the Springer fibers
as coordinate rings of some linear arrangements, due to Goresky-MacPherson (see \cite[Thm.\ 3.1]{GM}) and 
Kumar-Procesi (see \cite{KP}). Note that \cite{GM} and \cite{KP} assume $\bG$ to be semisimple,
while we only assume it to be connected reductive, however, one can
check that the constructions and arguments of \cite{GM} and \cite{KP} still go through in this case.
In fact, we are mainly interested in the case $\bG=\GL_n$ (due to the assumption $(\dagger)$ below).

%Kumar and Procesi computed in \cite{KP} the equivariant cohomology 
%$H^*_{\bT_a}(X_a,\C)$ of some Springer fibers, where $\bT_a$ is a maximal torus in the centralizer of $a$.
%Their result also follows from a more general Theorem of Goresky-MacPherson \cite[Thm.\ 3.1]{GM}
We start with a nilpotent $a\in \fg$ which is a principal nilpotent in the
Lie algebra of a Levi subgroup $\bL\sub \bP$ of some parabolic subgroup $\bP\sub \bG$.
Define the torus $\bT_a\sub \bL$ as the connected component of $1$ in the center of $\bL$. Note that $a$ is a distinguished nilpotent
in $\bL$, hence, $\bT_a$ is a maximal torus in $\bG_a^0$, the connected component of $1$ in the centralizer of $a$ in $\bG$
(see \cite[Sec.\ 4.7]{Jantzen}).
Let $\bT\sub \bL$ be a maximal torus, and let $\ft$ (resp., $\ft_a$) be the Lie algebra of
$\bT$ (resp., $\bT_a$).

We define the {\it linear arrangement associated with $(a, \bL)$} as above, to be the following union of linear subspaces
in $\ft_a\times\ft$
\begin{equation}\label{KP-variety-eq}
Z_a=\cup_{w\in W} (\id\times w)\De(\ft_a)\sub \ft_a\times \ft,
\end{equation}
where $\ft_a\sub \ft$ is the Lie algebra of $\bT_a\sub \bT$, $\De:\ft_a\to \ft_a\times \ft$ is the diagonal embedding,
and $W=N_{\bG}(\bT)/\bT$ is the Weyl group of $\bG$ acting on $\ft$. 
We view $Z_a$ as a subscheme in the affine space $\ft_a\times\ft$, equipped with the reduced scheme structure,
and denote by 
$$\C[Z_a]=S(\ft_a^*)\ot S(\ft^*)/I(Z_a)$$ 
the corresponding coordinate ring (where $I(Z_a)$ is the ideal of polynomials vanishing on $Z_a$).
Note that the natural projection $Z_a\to \ft_a$ is a finite surjective morphism.
Also, $Z_a$ is a cone, so the algebra $\C[Z_a]$ is naturally graded.

Let us make the following assumption:

\vspace{2mm}

\noindent
{\bf $(\dagger)$ the restriction homomorphism $H^*(X,\C)\to H^*(X_a,\C)$ is surjective.}

\vspace{2mm}

For example, this assumption is satisfied in the case of type $A$, due to the work of Spaltenstein \cite{Spalt}.
Under the assumption $(\dagger)$ there is an isomorphism of graded algebras (see \cite[Thm.\ 3.1]{GM}, \cite{KP})
\begin{equation}\label{KP-isom}
\C[Z_a]\rTo{\sim}H^*_{\bT_a}(X_a,\C),
\end{equation}
where the standard grading on the ring of functions on $Z_a$ is doubled. This isomorphism fits into
the commutative diagram (which uniquely determines it)
\begin{equation}\label{KP-diagram}
\begin{diagram}
S(\ft_a^*)\ot S(\ft^*)&\rTo{\chi}& H^*_{\bT_a}(X,\C)\\
\dTo{}&&\dTo{}\\
\C[Z_a] &\rTo{\sim}& H^*_{\bT_a}(X_a,\C),
\end{diagram}
\end{equation}
where the vertical arrows are induced by the embeddings $Z_a\hra \ft_a\times \ft$ and $X_a\hra X$,
and $\chi$ is induced by the natural map 
$$S(\ft^*)=H^*_\bT(pt)\rTo{\b} H^*_\bG(X)\to H^*_{\bT_a}(X)$$
and by the standard action of $S(\ft_a^*)=H^*_{\bT_a}(pt)$ on $H^*_{\bT_a}(X)$.
Here the map $\b$ is defined in \cite{KP} by $\b(\la)=c_1(L_\la)$, where $L_\la$ is the $\bG$-equivariant line bundle on $X$ associated with a character $\la:\bT\to\G_m$, and $c_1$ denotes the $\bG$-equivariant first Chern class.
It is easy to see that $\b$ coincides with the natural isomorphism obtained
using the identification $X=\bG/\bB$ (see \eqref{Spr-res-coh-eq}).
%Compatibilities:
%the action of $H^*_{\bT_a}(pt)$ corresponds to the projection on the first factor $Z_a\to \ft_a$,
%the action of $W$ is induced by the action on the second component.

Recall that $X_a$ is equivariantly formal (since its cohomology is concentrated in even degrees). 
Hence, by Lemma \ref{equiv-formal-lem}, we have 
\begin{equation}\label{KP-isom-aux}
H^*_{\bG^0_a}(X_a)\simeq H^*_{\bT_a}(X_a)^{W_a},
\end{equation}
where $W_a$ is the Weyl group attached to the maximal torus $\bT_a\sub \bG^0_a$.

\begin{lem}\label{KP-W-action-lem} 
Assume condition $(\dagger)$ holds.

\noindent
(i) The Springer action of $W$ on $H^*_{\bT_a}(X_a,\C)$ corresponds under \eqref{KP-isom} to the action
of $W$ on $Z_a$ given by the action on the second component.

\noindent
(ii) The natural action of $W_a$ on $H^*_{\bT_a}(X_a,\C)$ corresponds under \eqref{KP-isom} to
the action of $W_a$ on $Z_a$ given by the action on the first component. 
\end{lem}

\Pf . (i) Since the restriction homomorphism $H^*_{\bT_a}(X,\C)\to H^*_{\bT_a}(X_a,\C)$ is compatible with the
Springer $W$-action (see Lemma \ref{Spr-classical-W-lem}(i)), it is enough to check that the homomorphism $\chi$ from diagram \eqref{KP-diagram} 
is compatible with the $W$-actions. In other words, it is enough to check that the natural isomomorphism
$$H^*_\bT(pt)\simeq H^*_\bG(\bG/\bT,\C)\simeq H^*_\bG(\bG/\bB,\C)$$
is compatible with $W$-actions. But this follows immediately from Lemma \ref{Spr-classical-W-lem}(ii).

\noindent (ii)
It is enough to check that the surjective map
$$S(\ft_a^*)\ot S(\ft^*)\to \C[Z_a]\rTo{\sim} H^*_{\bT_a}(X_a)$$
is $W_a$-equivariant, where $W_a$ acts naturally on $S(\ft_a^*)$ and acts trivially on
$S(\ft^*)$. But this follows from the fact that the restriction of $\chi$ to $S(\ft^*)$ is the homomorphism
$$S(\ft^*)\simeq H^*_\bG(X)\to H^*_{\bG^0_a}(X)\to H^*_{\bG^0_a}(X_a)\simeq H^*_{\bT_a}(X_a)^{W_a}\hra
H^*_{\bT_a}(X_a).$$
\ed

\begin{rem} Kumar and Procesi in \cite{KP} sketch a proof of the compatibility with $W$-actions, similar to
the one in Lemma \ref{KP-W-action-lem}(i), for the induced identification of the non-equivariant cohomology of
the Springer fiber $X_a$.
%\noindent 
%2. About the definition of the homomorphism $\chi$ via Chern classes and the localization at the fixed points of $T$???
\end{rem}

%The following Proposition summarizes the above discussion.
Let us consider the geometric quotient
\begin{equation}\label{KP-variety-2-eq}
\ov{Z}_a=Z_a/W_a
\end{equation}
of $Z_a$ with respect to the $W_a$-action
induced by the action on the first factor of the product $\ft_a\times \ft$.

\begin{prop}\label{KP-prop}
Assume that $(\dagger)$ holds.
Then we have an isomorphism of $S(\ft^*_a)\ot S(\ft^*)$-modules (resp.
$S(\ft^*_a)^{W_a}\ot S(\ft^*)$-modules) compatible with the algebra structures and the $W$-action,
$$H^*_{\bT_a}(X_a,\C)\simeq \C[Z_a] \ \text{ (resp., }
H^*_{\bG^0_a}(X_a)\simeq \C[\ov{Z}_a]).$$
Furthermore, the schemes $Z_a$ and $\ov{Z}_a$ are Cohen-Macaulay, and the natural projections
$Z_a\to \ft_a$ and $\ov{Z}_a\to \ft_a/W_a$ are flat.
\end{prop}

\Pf .
The first assertion follows the isomorphisms \eqref{KP-isom} and \eqref{KP-isom-aux} and Lemma 
\ref{KP-W-action-lem}.
To prove the second assertion, we note that by the equivariant formality of $X_a$, the equivariant cohomology 
$H^*_{\bT_a}(X_a,\C)$ (resp., $H^*_{\bG^0_a}(X_a)$)
is free of finite rank over $H^*_{\bT_a}(pt)\simeq S(\ft_a)$ (resp., $H^*_{\bG^0_a}(pt)\simeq S(\ft_a^*)^{W_a}$). Hence,
we deduce that the projection $Z_a\to \ft_a$ (resp., $\ov{Z}_a\to \ft_a/W_a$)
Applying the Hironaka's criterion (see \cite[Cor.\ 18.17]{Eis})
we deduce that $Z_a$ (resp., $\ov{Z}_a$) is Cohen-Macaulay.
\ed

%\begin{equation}\label{KP-quot-isom}
%\C[Z_a]^{W_a}\simeq 
%\end{equation}
%compatible with $W$-actions,

%We have a natural map induced by the projection $\ft_a\times\ft\to \ft$, 
%\begin{equation}\label{pi-a-map-eq}
%\pi_a:\ov{Z}_a\to \bS_a:=W\cdot \ft_a\sub\ft.
%\end{equation}
%Note that $\bS_a$ is the support of the module $N_{a,1}$ associated with the nilpotent $a$ (and the trivial local system).
%In some cases the map $\pi_a$ is a partial normalization of $\bS_a$. In any case the singularities of $\ov{Z}_a$ are better
%behaved that those of $\bS_a$, as the following result shows.

%\begin{lem}\label{Za-CM-lem} 
%\end{lem}

\begin{rems} 1. There exist other descriptions of the cohomology of Springer fibers, still under the assumption $(\dagger)$,
see \cite{DeCP}, \cite{Carrell}, \cite{KT}.

\noindent
2. By Lemma \ref{G/H-lem}(ii), we have
$$H^*_{\bG_a}(X_a,\C)\simeq H^*_{\bG_a^0}(X_a,\C)^{\bG_a/\bG_a^0}.$$
Hence, from Proposition \ref{KP-prop} we get in the above situation
$$H^*_{\bG_a}(X_a,\C)\simeq \C[\ov{Z}_a]^{\bG_a/\bG_a^0}$$
provided $(\dagger)$ holds. It is easy to see using \cite[Thm.\ 1.3]{GM}
that the latter isomorphism holds under the weaker assumption that the homomorphism
\begin{equation}\label{weaker-surjectivity}
H^*_{\bG_a}(X,\C)\to H^*_{\bG_a}(X_a,\C)
\end{equation}
is surjective. For example, this holds if $\bG_a/\bG_a^0$ acts trivially on $H^*_{\bG_a^0}(pt)$ and the image of the restriction 
homomorphism $H^*(X,\C)\to H^*(X_a,\C)$ coincides with the subspace of $\bG_a/\bG_a^0$-invariants
in $H^*(X_a,\C)$.
\end{rems}

\subsection{The case of type $A$}\label{local-A-sec}

In this section we specialize to the case $\bG=\GL_n$, so that $W=S_n$.
In this case nilpotent orbits are parametrized by partitions of $n$.
For a partition $\la$ of $n$, the reductive part 
of the centralizer of the corresponding nilpotent element $a_\la$
is $\bG_\la=\prod_i (\GL_{r_i})^i$, where $r_i$ are multiplicities of parts of $\la$.
In particular, $\bG_\la$ is connected. 
The Springer correspondence associates with a nilpotent orbit $O_\la$ the irreducible
representation of $S_n$ in the top degree cohomology of the corresponding Springer fiber
$X_\la:=X_{a_\la}$, which is
known to be the usual representation
$L_\la$ of $S_n$ corresponding to $\la$ (e.g., $L_{(n)}=\mathbf{1}$ is the trivial representation,
$L_{(1)^n}=\de$, is the sign representation, where $(1)^n$ is $1$ repeated $n$ times).
We use the {\it dominance partial ordering} $\le$ on partitions of $n$ defined by $\la\ge\mu$ if 
$\la_1+\ldots+\la_i\ge \mu_1+\ldots+\mu_i$ for all $i\ge 1$.
Note that $(n)$ is the biggest partition and $(1)^n$ is the smallest. 
It is well known (see \cite{Gerst})
that this order corresponds to the adjacency order of nilpotent orbits,
i.e., $O_\mu\sub\ov{O_\la}$ if and only if $\mu\le \la$. We will use the following property of this ordering.

\begin{lem}\label{part-merge-lem} 
Let $\la^{(1)},\ldots,\la^{(s)}$ and $\mu^{(1)},\ldots,\mu^{(s)}$
be partitions with 
$$|\la^{(1)}|+\ldots+|\la^{(s)}|=|\mu^{(1)}|+\ldots+|\mu^{(s)}|=n,$$ 
and let $\la$ (resp., $\mu$) be the partition of $n$ obtaining by taking all the parts of the partitions $\la^{(1)},\ldots,\la^{(s)}$
(resp., $\mu^{(1)},\ldots,\mu^{(s)}$). Assume that $\la^{(i)}\ge\mu^{(i)}$ for $i=1,\ldots,s$. Then
$\la\ge \mu$. If in addition, $\la^{(i)}>\mu^{(i)}$ for at least one $i$ then $\la>\mu$.
\end{lem}

\Pf . This becomes apparent from the following reformulation of the condition $\la\ge \mu$: for every $k\ge 1$ and
for every collection of $k$ parts of $\mu$, $\mu_{i_1},\ldots,\mu_{i_k}$ (where $i_1,\ldots,i_k$ are distinct),
there exists a collection of $k$ parts of $\la$, $\la_{j_1},\ldots,\la_{j_k}$ with
$$\la_{j_1}+\ldots+\la_{j_k}\ge\mu_{i_1}+\ldots+\mu_{i_k}.$$
\ed

We can identify the Lie algebra of a maximal torus $\ft\sub\fg$ with the standard $n$-dimensional representation of $S_n$, so that $H^*_\bT(pt)=S(\ft^*)\simeq\C[x_1,\ldots,x_n]$ and 
$$A:=A_{S_n}=S_n\ltimes \C[x_1,\ldots,x_n],$$
where $\deg(x_1)=\ldots=\deg(x_n)=2$, while the elements of $S_n$ have degree $0$.

For each partition $\lambda$, $|\la|=n$, 
the maximal torus $\bT_\la$ in $\bG_\la$ consists of diagonal matrices $\diag(t_1,\ldots,t_n)\in\GL_n$ such that
$t_1=\ldots=t_{\la_1}$, $t_{\la_1+1}=\ldots=t_{\la_1+\la_2}$, etc.
Its Lie algebra $\ft_\la\sub \ft$ is described by the same equations inside $\ft$.
%Let $x_1,\ldots,x_n$ be standard coordinates on $\ft$, subject to $\sum_i x_i=0$.
%we have a linear subspace  (see Section \ref{KP-sec}).
The subgroup 
$$S_\lambda=\prod_k S_{\la_k}\sub S_n$$ 
is the pointwise stabilizer of $\ft_\lambda$. 
Let $H_\lambda\sub S_n$ be the normalizer of $S_\lambda$: 
$$H_\la=S_\lambda\rtimes W_\la,$$
where 
$$W_\la=\prod_i S_{r_i}$$ 
($r_i$ is the multiplicity of $i$ in $\la$).
Note that $H_\la$ is exactly the set of $w\in S_n$ such that $w(\ft_\lambda)=\ft_\la$. 
%We denote by $\lambda'$ the conjugate partition of $\lambda$.

On the other hand, $\ft_\la$ is the maximal torus in the reductive part of the centralizer $\bG_\la$, and
$W_\la$ is the corresponding Weyl group.
Hence, the algebra $H^*_{\bG_\la}(pt)$ can be identified with
$$R_\la:=S(\ft_\la^*)^{W_{\la}}=\Oscr(\ft_{\la}/W_\la)$$
(note that $R_\la$ is isomorphic to a polynomial algebra).
Thus, the pieces of the semiorthogonal decomposition of $D^b_{S_n}(\A^n)$ obtained from Theorem \ref{Springer-thm}
have form $D^b(\ft_\la/W_\la)$.

For a partition $\la$, $|\la|=n$, let 
$$P_\lambda=\Oscr(\ft)\otimes L_\lambda.$$ 
be the standard $A$-projective modules. Also, let
$M_\la$ and $N_\la$ be the $A$-modules obtained from the Springer correspondence
(see Definition \ref{modules-Springer-defi}). Let us restate some of the results of Theorem \ref{Springer-thm}
and of Propositions \ref{Ext-calculation-prop} and \ref{modules-Springer-prop}(iv) in the type $A$ case.

\begin{thm}\label{type-A-thm} (i) For each partition $\la$, $|\la|=n$, we have isomoprhisms
\begin{equation}
\label{springer}
R\Hom_A(M_\la,M_\la)\simeq R\Hom_A(N_\la,N_\la)\simeq R_\la,
\end{equation}
compatible with the $\Oscr(\ft)^{S_n}$-action.
Let $\lan M_\la\ran$ (resp., $\lan N_\la\ran$) denote 
the thick subcategory of $D^f(A)$ generated by $M_\la$ (resp., $\lan N_\la\ran$). Then these subcategories
are admissible subcategory and we have equivalences
$$\lan M_\la\ran\simeq \lan N_\la\ran\simeq D^f(R_\la)=\Perf(R_\la).$$
%given by the functor $R\Hom(M_\la,?)$. 
Furthermore, $M_\la$ and $N_\la$ are projective as $R_\la$-modules.

\noindent
(ii) One has
\[
R\Hom_A(P_\mu,M_\la)=R\Hom_A(M_\mu,M_\la)=R\Hom_A(N_\la,N_\mu)=0\qquad \text{for $\mu\not\ge\la$}.
\]
%We have 
%$$R\Hom(N_\la,N_\mu)=0 \ \ \text{ for } \la\not\le\mu,$$

\noindent
(iii) For any set $S$ of partitions of $n$, closed under passing to smaller elements with respect to the
dominance order, we have
$$\Cscr_S:=\lan P_\la \ |\ \la\in S\ran=\lan M_\la \ |\ \la\in S\ran.$$
For any ordering $S=\{\la_1<\ldots<\la_k\}$ refining the dominance order, the
subcategories $\lan M_{\la_k}\ran,\ldots,\lan M_{\la_1}\ran$ form a semiorthogonal decomposition of $\Cscr_S$.
Furthermore, for every $\la$ we have an exact triangle
\begin{equation}\label{Pla-Mla-eq}
K\to P_\la\to M_\la\to\ldots
\end{equation}
in $D^f(A)$ such that $K\in \lan P_\mu \ |\ \mu<\la\ran$.

\noindent
(iv) Let $\la_1<\ldots<\la_{p(n)}$ be an ordering refining the dominance order. 
Then we have a semiorthogonal decomposition 
$$D^f(A)=\lan \lan N_{\la_1}\ran, \ldots , \lan N_{\la_{p(n)}}\ran\ran.$$
\end{thm}

\begin{rem} By Proposition \ref{HH-semiorth-match-prop}, 
the pieces of our semiorthogonal decompositions in Theorem \ref{type-A-thm}
automatically match those predicted by Conjecture A for the action of $S_n$ on $\A^n$. 
It is not hard to match them explicitly. First, we observe that $\ft_\la=\ft^{w_\la}$, where
$w_\la\in S_n$ is the permutation with the cycle type $\la$.
Next, the centralizer $C_\la$ of $w_\la$ in $S_n$ is the semidirect product
$$C_\la=C^0_\la\rtimes W_\la,$$
where
$$C^0_\la=\prod_k \Z/\la_k.$$
The action of $C_\la$ on $\ft_\la$ factors through $W_\la$, so we get
$$\ft_\la/W_\la=\ft^{w_\la}/C(\la).$$
\end{rem}

%while  $W_\la$ can be identified with the Weyl group of $\bG_\la$.

%By Proposition \ref{modules-Springer-prop}(iv), the modules $M_\la$ and $N_\la$ obtained from the Springer correspondence %are supported precisely on 
%$S_n\cdot\ft^{w_\la}\sub \ft$, where 
%$w_\la\in S_n$ is the permutation with the cycle type $\la$.

Note that our two semiorthogonal decompositions of $D^f(A)$, in terms of the modules $(M_\la)$ and $(N_\la)$,   
are both natural but from different point views. Theorem \ref{type-A-thm}(iii) shows that
the modules $(M_\la)$ appear simply from the semi-orthogonalization procedure applied to the projective modules $(P_\la)$
and the dominance order (in Section \ref{induced-sec} we will also show that they appear from the semi-orthogonalization
procedure applied to certain induced $A$-modules).
On the other hand, the isomorphism of Proposition \ref{KP-prop} 
leads to a geometric interpretation of the modules $N_\la=N_{a_\la}$.
It is this identification that will allow us to globalize this picture in Section \ref{global-A-n-sec}. 
Namely, the nilpotent $a_\la$ is the principal nilpotent in the Levi subgroup $\bL=\prod_k \GL_{\la_k}$,
so we can associate the schemes $Z_{\la}$ and $\ov{Z}_\la$ to it (see \eqref{KP-variety-eq} and 
\eqref{KP-variety-2-eq}).

\begin{prop}\label{modules-KP-prop} 
For each partition $\la$ (with $|\la|=n$) we have an isomorphism of graded $A$-modules
$$N_\la\simeq \C[\ov{Z}_\la],$$
where the grading on the right is doubled.
\end{prop}

\Pf . By Proposition \ref{modules-Springer-prop},
$$N_\la\simeq H^*_{\bG_\la}(X_\la,\C).$$
Now the required isomorphism follows from Proposition \ref{KP-prop}. Note that the assumption $(\dagger)$
is satisfied in this case (see \cite{Spalt}).
\ed

%Remark on the case of $S_{n_1}\times\ldots\times S_{n_k}$???

Note that the linear spaces constituting $Z_{\la}$ 
are in bijection with $S_n/S_\lambda$.
Thus, the components of $\ov{Z}_{\la}=Z_{\la}/W_\la$ are numbered by $S_n/H_\lambda$.
%The latter group acts on $\ft_\la$ by permuting
%blocks of coordinates corresponding to equal parts of the partition $\la$.
%We observe that $S_\lambda\rtimes W_\la$ is exactly 
%the subgroup of $w\in S_n$ such that $w(\ft_\la)=\ft_\la$.
Hence, the components of $\ov{Z}_\la$ and of $S_n\cdot\ft_\la$ are in bijection.
Furthermore, it is easy to see that the projection
\begin{equation}\label{pi-la-map-eq}
\pi_\la:\ov{Z}_\la\to \bS_\la:=S_n\cdot \ft_\la
\end{equation}
is finite, surjective and is an isomorphism over the general point of each component
(where we equip $\bS_\la$ with the reduced scheme structure).

We are going to show now how the isomorphism of Proposition \ref{KP-prop}
can be used to deduce that the stratum $\bS_\la$
is Cohen-Macauley in some cases.

\begin{prop}\label{equal-parts-prop}
%Assume the characteristic is zero.
If all parts in $\la$ are equal then the map $\pi_\la$ (see \eqref{pi-la-map-eq})
is an isomorphism. Thus, in this case we have an isomorphism
$$H^*_{\bG_\la}(X_\la)\simeq \C[\bS_\la]$$
of graded algebras (where the grading on $\C[\bS_\la]$ is doubled).
\end{prop}

\Pf . Note that $\ov{Z}_\la$ is a reduced closed subscheme of 
$\ft_{\la}/W_\la\times \bS_\la\sub \ft_\la/W_\la\times\ft$.
Thus, to show that $\pi_\la$ is an isomorphism,
it is enough to construct a morphism $\bS_\la\to \ft_\la/W_\la$, such that the composition
$\ov{Z}_\la\to \bS_\la\to \ft_\la/W_\la$ is the natural projection.
Let $\la=(m)^p$, where $n=mp$. 
Then $\ft_\la$ is 
%the subspace in $\ft=\A^n$
%\{(x_1,\ldots,x_n)\in\A^n \ |\ x_1+\ldots+x_n=0\}$
%given by the equations
%$$x_1=\ldots=x_m, x_{m+1}=\ldots=x_{2m},\ldots, x_{n-m+1}=\ldots=x_n.$$
%In other words, $\ft_\la$ is 
the image of the embedding 
$$i_{m}: \A^p\to\A^n: (y_1,\ldots,y_p)\mapsto (y_1,\ldots,y_1,y_2,\ldots,y_2,\ldots,y_p,\ldots,y_p)$$
where each $y_i$ is repeated $m$ times. 
%(and $y_1+\ldots +y_p=0$).
Note that $W_\la=S_p$ so $\C[\ft_a/W_\la]$ is just the ring of symmetric functions in $(y_1,\ldots,y_p)$.
%modulo the relation $y_1+\ldots+y_p=0$.
Now we observe that
if $e_i(\cdot)$ are elementary symmetric functions then $e_1(y),\ldots,e_p(y)$
can be expressed as some universal polynomials in $e_1(i_m(y)),\ldots,e_p(i_m(y))$.
%(this uses the assumption that the characteristic is zero).
These universal polynomials give us the required map $\ft\to\ft_\la/W_\la$ which we then restrict to $\bS_\la$.

Now the isomorphism with the equivariant cohomology of $X_\la$ follows from Proposition \ref{KP-prop}.
\ed

\begin{cor}\label{CM-cor} 
If all parts in $\la$ are equal then $\bS_\la$ is Cohen-Macaulay.
\end{cor}

\Pf . Combine Proposition \ref{equal-parts-prop} with Proposition \ref{KP-prop}.
\ed

\begin{cor}
For $\la=(m)^p$ there is an isomorphism of graded spaces
\begin{equation}\label{graded-spaces-cor} 
\C[y_1,\ldots,y_p]^{S_p}\ot H^{*/2}(X_\la,\C)\simeq \C[\bS_\la].
\end{equation}
\end{cor}

%In fact, the proof also provided a flat map from $\SS_\la$ to the affine space $\ft_\la/W_\la$...

\begin{rem}
It was proved in \cite{EGL} that $\bS_\la$ is Cohen-Macaulay for all $\la=((m)^s,(1)^r)$, where $r<m$.
%(note that by Lemma \ref{norm-bijection-lem} the map $\pi_\la$ is a bijection for such $\la$).
The above Corollary \ref{CM-cor} gives a more direct proof in the case $r=0$.
The isomorphism \eqref{graded-spaces-cor} allows to compute the character of $\C[\bS_\la]$, as a graded $S_n$-representation, still in the case $r=0$. Namely, \eqref{graded-spaces-cor} reduces this to computing the character of
$H^*(X_\la,\C)$, which can be realized as the stalk of the Springer sheaf $\Ascr$. Then its character can be computed using the decomposition \eqref{A-decomposition} together with \cite[Thm.\ 2]{Lusztig-GP}.
The formula for the character of $\C[\bS_\la]$ was also deduced in \cite{EGL} by different methods.
\end{rem}

\subsection{Induced modules in type $A$}\label{induced-sec}

For each $\lambda$, a partition of $n$, let us consider the induced representations of $S_n$,
\[
I_\lambda=\Ind^{S_n}_{H_\lambda}\Oscr(\ft_\lambda).
\]
We view $I_\lambda$ as an $A$-module using the restriction homomorphism 
$\Oscr(\ft)\to \Oscr(\ft_\la)$,
which is compatible with $H_\la$-action.

%$L_\lambda$ as an $A$-module.

\begin{lem} 
\label{lem1} One has as $\Oscr(\ft)^{S_n}$-modules:
\[
\Hom_A(P_\lambda,I_\mu)
=
\begin{cases}
0&\lambda\not\ge \mu\\
\Oscr(\ft_\mu)^{W_\mu}&\lambda=\mu.
\end{cases}
\]
\end{lem}
\begin{proof}
We have
\[
\Hom_A(P_\lambda,I_{\mu})=\Hom_{S_n}(L_\lambda,I_{\mu})
\]
By Lemma \ref{lem2} below, this is equal to
\[
\Hom_{S_n}(L_\lambda,\Ind^{S_n}_{S_\mu}\one)\otimes \Oscr(\ft_\mu)^{W_\mu}.
\]
It follows from Pieri's formula that $\Ind^{S_n}_{S_\mu}\one$ is a direct sum of 
$L_{\gamma}$ with $\gamma\ge \mu$,
such that $L_{\mu}$ occurs with multiplicity one. This immediately implies
our assertion.
\end{proof}

We have used the following lemma.
\begin{lem}
\label{lem2}
We have 
\[
I_\mu\cong \Ind^{S_n}_{S_\mu}\one \otimes \Oscr(\ft_\mu)^{W_\mu}
\]
as  $S_n\otimes \Oscr(\ft)^{S_n}$-modules (with $S_n$ acting trivially on $\Oscr(\ft_\mu)^{W_\mu}$).
\end{lem}
\begin{proof}
We observe that $W_\lambda=\prod_i S_{r_i}$ acts as a reflection group on $\ft_\lambda$.
The harmonic decomposition (which is a part of the Shephard-Todd's characterization of complex reflection groups;
see e.g., \cite[Thm.\ 4.1]{Broue}) yields
\[
\Oscr(\ft_\mu)= k[W_\mu]\otimes \Oscr(\ft_\mu)^{W_\mu}
\]
as $W_\mu$-modules and hence as $H_\mu$-modules. Applying $\Ind^{S_n}_{H_\mu}$ yields the result.
\end{proof}

From now on we fix a total ordering $\prec$ on partitions of $n$ which
refines the dominance ordering.  

Let $\Bscr_{\la}$ be the thick subcategory of $D^f(A)$ generated by $P_{\epsilon}$, $\epsilon\preceq \la$.
In particular, $\Bscr_{[1\cdots 1]}=D^f(A)$. 
By Theorem \ref{type-A-thm}, we know that $\Bscr_{\la}$ is admissible and that
$$\Bscr_\la=\lan M_\epsilon \ |\ \epsilon\preceq\la\ran.$$
Also, the subcategories $\lan M_\mu\ran$, ordered in decreasing order
of $\mu$ for $\prec$, form a semiorthogonal decomposition of $D^f(A)$.
In particular, any subsegment of $(M_\mu)$
generates an admissible subcategory.

%\Bscr_\mu$, where $\mu$ is the predecessor of $\la$ for $\prec$.
%In other words, $M_\la$ be obtained  by right orthogonalizing $P_\la$
%with respect to $\Bscr_{\mu}$.

%The $(M_\mu)_\mu$, satisfy the hypotheses of Lemma \ref{intrinsic}
%with $R=\Oscr(\ft)^{S_n}$. 

Let $\Ascr_\la$ be the thick
subcategory of $D^f(A)$ generated by $I_\mu$, $\mu\succeq\la$.

\begin{prop} One has 
$$\Ascr_{\la}=\lan M_\mu \ |\ \mu\succeq \la\ran.$$
In particular, $\Ascr_{\la}$ is admissible and we have a semiorthogonal decomposition
$$D^f(A)=\langle \Ascr_{\la},\Bscr_{\gamma}\rangle,$$ 
where $\gamma$ is the predecessor to $\la$ for $\prec$.
Also, there exists an exact triangle
\begin{equation}\label{Mla-Ila-eq}
M_\la\to I_\la\to C\to\ldots
\end{equation}
with $C\in \Ascr_\tau$, where $\tau$ is the successor of $\la$. In other words, $M_\la$ is obtained from $I_\la$ by
left orthogonalization with respect to the subcategory $\Ascr_\tau$ generated by $I_\mu$ with $\mu\succ\la$.
\end{prop}

\begin{proof} By
  induction we may assume that $\Ascr_{\tau}$ is generated by
  $M_\mu$ for $\mu\succeq \tau$. In particular, $\Ascr_{\tau}$ is admissible. 

Let $M$ be the left orthogonalization of $I_\la$ with respect to $\Ascr_{\tau}$.
Since $M$ is right orthogonal to $\Bscr_{\gamma}$ (by Lemma \ref{lem1}) and left orthogonal to $M_\mu$,
$\mu\succeq\tau$, we obtain that $M$ is in the triangulated subcategory $\lan M_\la\ran\sub D^f(A)$ 
generated by $M_{\la}$. 

Recall that by Theorem \ref{type-A-thm}, 
\begin{equation}
\label{equivalence}
 \lan M_\la\ran\cong\Perf(R_\la)
\end{equation}
where $R_\la=\End_A(M_{\la})=\Oscr(\ft_{\la})^{W_{\la}}$, 
and the equivalence is given by $R\Hom_A(M_{\la},-)$.
%We have as $\Oscr(\ft)^{S_n}$-modules
%\begin{align*}
From the exact triangle \eqref{Pla-Mla-eq} and from the vanishing of Theorem \ref{type-A-thm}(ii),
since $M\in \lan M_\la\ran$, we deduce an isomorphism
$$R\Hom_A(M_{\la},M)=R\Hom_A(P_\la,M).$$
Next, from Lemma \ref{lem1} we deduce
$$R\Hom_A(P_\la,M)=R\Hom_A(P_{\la},I_{\la})
= \Oscr(\ft_{\la})^{W_{\la}}.$$
%\end{align*}
A priori this isomorphism is only as $\Oscr(\ft)^{S_n}$-modules. However, it is sufficient to conclude that
$R\Hom_A(M_{\la},M)$ is Cohen-Macaulay and hence a projective $R_\la$-module.
Since its rank is equal to $1$, we get that 
$$R\Hom_A(M_{\la},M)\cong R_\la$$ 
as $R_\la$-modules.
Since $M\in\lan M_\la\ran$, using the equivalence \eqref{equivalence} we deduce an isomorphism $M\cong M_{\la}$.
We conclude that $M_{\la}$ is in  $\Ascr_\la$, finishing the proof.
\end{proof}

\begin{cor} The subcategory $\Ascr_\la$ consists of $C\in D^f(A)$ such that $H^iC$, as $S_n$-modules,
have only $L_\mu$ with $\mu\succeq\la$ as summands.
\end{cor}

\begin{cor}\label{Mla-Ila-cor} 
The morphism of modules $M_\la\to I_\la$ from \eqref{Mla-Ila-eq} is injective
and its localization over a dense open subset of $\Spec(R_\la)$ is an isomorphism.
\end{cor}

\begin{proof} It is easy to check that if $\ft_\la\sub w \ft_\mu$ for some $w\in S_n$ then
$\la\ge\mu$. This implies that all the modules $I_\mu$ with $\mu>\la$ are zero over
the generic point of $\Spec(R_\la)$, so the triangle \eqref{Mla-Ila-eq} implies that
the morphism $M_\la\to I_\la$ is an isomorphism over the generic point of $\Spec(R_\la)$.
Since $M_\la$ is projective as $R_\la$-module, the injectivity follows.
\end{proof}

\begin{rems} 1. One can check that in fact $C$ from the exact triangle \eqref{Mla-Ila-eq} lies in the
subcategory $\lan M_\mu \ |\ \mu>\la\ran$.

\noindent
2. Kato's result that the natural morphism $P_\la\to M_\la$ is surjective (see \cite[Cor.\ 3.6]{Kato}) together with Corollary
\ref{Mla-Ila-cor} imply that
the module $M_\mu$ is exactly the image of the morphism of $A$-modules
$$P_\mu\to I_\mu$$
corresponding to $1\in \Oscr(\ft_\mu)^{W_\mu}=\Hom_A(P_\mu,I_\mu)$.
\end{rems}

\section{Actions on powers of a curve}\label{type-A-sec}

In this section we construct semiorthogonal decompositions of the derived category of equivariant coherent sheaves on 
$C^n$, where $C$ is a smooth curve, with respect to various natural groups. We start with the case of
$S_n$ acting by permutations of components, and end with the case of $G^n\rtimes S_n$, where
$G$ is any finite group acting on $C$. 

\subsection{Semiorthogonal decomposition for the action of $S_n$ on $C^n$}\label{global-A-n-sec}

Let $C$ be a smooth curve over $\C$. In this section we will prove Theorem B, 
stating that the functors $\Nscr_\la$ defined via diagram \eqref{N-lambda-diagram} give a semiorthogonal
decomposition of the category $D^b_{S_n}(C^n)$. 

The main observation is that for $C=\A^1$ this statement reduces to Theorem \ref{type-A-thm}.
Indeed, we can identify $C^n=\A^n$ with the Lie algebra of a maximal torus $\ft$ in $\GL_n$, $C[\la]\sub C^n$
with $\ft_\la$, $C^{(\la)}$ with $\ft_\la/W_\la$,
and $\ov{Z}_\la(\A^1)$ with the variety \eqref{KP-variety-2-eq}. Hence, by Proposition 
\ref{modules-KP-prop}, we get
$$\Nscr_\la(\Oscr_{\A^1[\la]/W_\la})\simeq N_\la,$$
where $N_\la$ are the $S_n\ltimes \Oscr(\ft)$-modules given by the equivariant cohomology of the Springer
fibers. More generally, in the notation of Section \ref{local-A-sec} we have the following commutative diagram of functors
\begin{diagram}
D^b(\A^1[\la]/W_{\la})&\rTo{\Nscr_\la}& D^b_{S_n}(\A^n)\\
\dTo{\simeq}&&\dTo{\simeq}\\
D^f(R_\la)&\rTo{?\ot N_\la}& D^f(A)
\end{diagram} 
Therefore, Theorem \ref{type-A-thm} gives the required semiorthogonal decomposition for $C=\A^1$.
%implies that for $C=\A^1$ the functors $\Nscr_\la$ are fully faithful and their images
%satisfy the semiorthogonality with respect to the dominance order of partitions
%$$\Hom(\Nscr_\mu(\cdot),\Nscr_\la(\cdot))=0 \ \ \text{ for } \mu>\la,$$
%and form a semiorthogonal decomposition of the derived category of $S_n$-equivariant sheaves on $\A^n$
%(by refining the dominance order to some total order).

The rest of the proof consists of reducing the statement for arbitrary $C$ to the local case $C=\A^1$. For this
we analyze the geometry of the diagram \eqref{N-lambda-diagram} over a formal neighborhood of
a point $x\in C^n$. Thus, we have to analyze the natural $W_\la\times S_n$-equivariant map
$$f_\la:Z_\la(C)\to C^n,$$
%inducing the map $\ov{f}_\la$ 
given by the projection to the second component, over a neighborhood of $x$ (here $W_\la$ acts trivially on $C^n$). 
Note that the image of $f_\la$ is $S_n\cdot C[\la]$, hence, it is enough to consider the case when $x=(x_1,\ldots,x_n)\in C[\la]$.
However, $x$ can be a non-generic point of $C[\la]$, i.e., in addition to the equalities \eqref{C-lambda-eq} there
may be some extra equalities $x_i=x_j$, so that the stabilizer subgroup $\St_x$ of $x$ in $S_n$ may be bigger than $S_\la$.

The following point of view on varieties $Z_\la(C)$ will be convenient. We replace the partition $\la$ by any
decomposition $\Pscr=(L_p)_{p\in P}$ of the set $\{1,\ldots,n\}$ into (unordered) non-empty disjoint subsets:
$\{1,\ldots,n\}=\sqcup_{p\in P}L_p$. 
The partition $\la$ is recovered by considering the cardinalities $|L_p|$. 
The group $S_n$ acts on such decompositions by $w\Pscr=(wL_p)_{p\in P}$\footnote{The marking of the parts is not
part of the data. For example, the decomposition $\Pscr=(\{1,2\},\{3,4\})$ is preserved by the permutation $(13)(24)$.}.
We denote by $H_\Pscr\sub S_n$ the subgroup of $w$ such that $w\Pscr=\Pscr$, and by $W_\Pscr$ the quotient of
$H_\Pscr$ by the subgroup of permutations preserving each subset $L_p$.

Let $C[\Pscr]\sub C^n$ be the subvariety consisting of $(c_s)$ such that
$c_s=c_{s'}$ whenever $s,s'\in L_p$ for some $p$ (one has $C[\Pscr]\simeq C^P$).
Then we can define a subscheme $Z_{\Pscr}(C)\sub C[\Pscr]\times C^n$ as the union  (with reduced
scheme structure) of the graphs of
the composed maps $C[\Pscr]\rTo{\iota}C^n\rTo{w}C^n$, where $\iota$ is the natural embedding and
$w$ runs through $S_n$. It is equipped with a natural action of the group $W_\Pscr\times S_n$.
Let $f_\Pscr:Z_{\Pscr}(C)\to C^n$ denote the morphism induced by the projection to the second factor.
For any permutation $w\in S_n$ we have a commutative diagram
\begin{diagram}
Z_{\Pscr}(C)&\rTo{w\times\id}& Z_{w\Pscr}(C)\\
\dTo{f_\Pscr}&&\dTo{f_{w\Pscr}}\\
C^n&\rTo{\id}&C^n
\end{diagram}
where we use the isomorphism
$C[\Pscr]\to C[w\Pscr]$ induced by $w$. 
This gives a natural isomorphism
\begin{equation}\label{Z-functoriality-eq}
f_{\Pscr,*}\Oscr\simeq f_{w\Pscr,*}\Oscr
\end{equation}
compatible with the $S_n$-action and with the action of $W_\Pscr\simeq W_{w\Pscr}$.
Note that the above picture works similarly with $\{1,\ldots,n\}$ replaced by any finite set.

For each $\lambda$, a partition of $n$, we can consider the decomposition $\Pscr_\la$ into the subsets
$\{1,\ldots,\la_1\}$, $\{\la_1+1,\ldots,\la_1+\la_2\}$, etc. Then we recover the previously defined schemes
$C[\la]=C[\Pscr_\la]$, $Z_\la(C)=Z_{\Pscr_\la}(C)$, and we have $f_\la=f_{\Pscr_\la}$, $H_\la=W_{\Pscr_\la}$,
$W_\la=W_{\Pscr_\la}$.
Thus, for any decomposition $\Pscr$ we have an $S_n$-isomorphism 
$f_{\Pscr,*}\Oscr\simeq f_{\la,*}\Oscr$, where $\la$ is the 
partition of $n$ given by the sizes of the subsets forming $\Pscr$.

We can view a decomposition $\Pscr=(L_p)_{p\in P}$ as a surjective map
$$\psi:\{1,\ldots,n\}\to P,$$
defined by $L_p=\psi^{-1}(p)$ for each $p\in P$.
We would like to investigate the map $f_\Pscr$ near the fiber of a point $x\in C[\Pscr]$.
Extra equalities between the coordinates of a point $x$ correspond to an equivalence relation on $P$,
i.e., a surjective map $\phi:P\to Q$, where $Q$ is a finite set. Note that the stabilizer subgroup of $x$, $\St_x\sub S_n$,
is precisely the subgroup of permutations preserving the fibers of the composition $\phi\circ\psi$.

Now for each $q\in Q$ we have a decomposition $\Pscr(q)$ of the set $S(q):=\psi^{-1}\phi^{-1}(q)$
corresponding to the surjective map
$$\psi|_{S(q)}:S(q)\to \phi^{-1}(q)=P(q),$$
so that the parts of $\Pscr(q)$ are the fibers of $\psi$ over $P(q)$. In other words, $\Pscr(q)$ consists of all the parts of
$\Pscr$ corresponding to a given value of the coordinates of $x$.
We will refer to the collection $(\Pscr(q))_{q\in Q}$ as {\it decompositions associated with $\Pscr$ and a point $x\in C[\Pscr]$}.

The decomposition $P=\sqcup_{q\in Q} P(q)$
gives rise to the identification
\begin{equation}\label{C-lambda-dec-eq}
C[\Pscr]=C^{P}=\prod_{q\in Q}C^{P(q)}=\prod_{q\in Q}C[\Pscr(q)]
\end{equation}
so that the point $x$ corresponds to a collection $(x(q))_{q\in Q}$, where
each $x(q)$ belongs to the small diagonal $C\sub C[\Pscr(q)]$.
Note also that we have a natural decomposition 
\begin{equation}\label{C-n-lambda-dec-eq}
C^n=\prod_{q\in Q}C^{S(q)},
\end{equation}
where $S(q)\sub\{1,\ldots,n\}$ are exactly the sets of indices with the same value of the
corresponding coordinate of $x$.

\begin{lem}\label{W-lambda-orbit-lem}
For $w_1,w_2\in S_n$ such that $w_1x\in C[\Pscr]$ and $w_2x\in C[\Pscr]$ the following conditions are equivalent:

\noindent
(i) $w_1x$ and $w_2x$ belong to the same $W_\Pscr$-orbit;

\noindent
(ii) $H_\Pscr w_1\St_x=H_\Pscr w_2\St_x$;

\noindent
(iii) for each $q\in Q$ the decompositions of $S(q)$ induced by $w_1^{-1}\Pscr$ and $w_2^{-1}\Pscr$ differ by an automorphism
of $S(q)$.
\end{lem}

\Pf . The equivalence of (i) and (ii) is clear since $W_\Pscr w_i x=H_\Pscr w_i x$ for $i=1,2$.
Condition (iii) means that $gw_2^{-1}\Pscr=w_1^{-1}\Pscr$ for some $g\in\St_x$. This is equivalent
to $w_1gw_2^{-1}\Pscr=\Pscr$, i.e., $w_1gw_2^{-1}\in H_\Pscr$. Thus, we get an equivalence of (iii) and (ii).
\ed

Let us denote by $\hat{Z}_{\Pscr,(x,x)}$ the formal completion of $Z_\Pscr=Z_{\Pscr}(C)$ near the point 
$(x,x)\in Z_{\Pscr}\sub C[\Pscr]\times C^n$.

\begin{lem}\label{Z-lambda-prod-lem} 
One has an isomorphism 
$$\hat{Z}_{\Pscr,(x,x)}\simeq\hat{\prod}_{q\in Q}\hat{Z}_{\Pscr(q),(x(q),x(q))},$$
where on the right we take the completion of the product,
compatible with the decompositions \eqref{C-lambda-dec-eq} and \eqref{C-n-lambda-dec-eq}.
In particular, it is compatible with the action of $W_\Pscr\cap\St_x=\prod_{q\in Q} W_{\Pscr(q)}$.
\end{lem}

\Pf . By definition, the components of $Z_{\Pscr}$ passing through the point $(x,x)$ are numbered by $w\in \St_x$. Now
the assertion follows from the decompositions \eqref{C-n-lambda-dec-eq} and
\eqref{C-lambda-dec-eq}, compatible via the embedding $C[\Pscr]\to C^n$, together with the decompositions
$\St_x=\prod_q S_{|S(q)|}$, $W_\Pscr\cap\St_x=\prod_q W_{\Pscr(q)}$.
\ed

We now return to the setting of Theorem B.
Let $\pi:C^n\to C^n/S_n$ be the natural morphism. Also, for a partition $\la$, $|\la|=n$, let
$$g_\la:C^{(\la)}=C[\la]/W_\la\to C^n/S_n$$
be the morphism induced by the embedding of $C[\la]$ into $C^n$.

\begin{lem}\label{local-semiorth-lem}
(i) For a pair of partitions of $n$, $\la$ and $\mu$, one has
$$\pi_*R\und{\Hom}(\Nscr_\la(\Oscr),\Nscr_\mu(\Oscr))^{S_n}=0 \ \ \text{ for } \la\not\le\mu,$$
where we view $\Nscr_\la(\Oscr)$ and $\Nscr_\mu(\Oscr)$ as sheaves on $C^n$ equipped with $S_n$-action.
 
  \noindent
(ii) For each $\la$ the natural map
$$g_{\la,*}\Oscr\to \pi_*R\und{\Hom}(\Nscr_\la(\Oscr),\Nscr_\la(\Oscr))^{S_n}$$
is an isomorphism.

\noindent
(iii) If $F\in D^b_{S_n}(C^n)$ satisfies
$$\left(\pi_*R\und{\Hom}(\Nscr_\la(\Oscr),F)\right)^{S_n}=0$$
for all $\la$, $|\la|=n$, then $F=0$.
\end{lem}

\Pf . Let us set $Z_\la=Z_\la(C)$, $Z_\Pscr=Z_{\Pscr}(C)$ for brevity.

\noindent
(i) Let $x\in C^n$. Since $\pi^{-1}(\pi(x))\simeq S_n/\St_x$, we can compute the completion of the
sheaf in question at $\pi(x)$ as follows:
$$\left(\pi_*R\und{\Hom}(\Nscr_\la(\Oscr),\Nscr_\mu(\Oscr))^{S_n}\right)\sphat_{\pi(x)}\simeq
R\und{\Hom}\bigl(\Nscr_\la(\Oscr)\sphat_x,\Nscr_\mu(\Oscr)\sphat_x\bigr)^{\St_x}.$$
Thus, it is enough to study the picture in the formal neighborhood of $x$ in $C^n$ (keeping track of the $\St_x$-action).
Note that 
$$\Nscr_\la(\Oscr)=f_{\la,*}(\Oscr_{Z_\la})^{W_\la}.$$
The fiber $f_{\la}^{-1}(x)$ consists of the points $(y,x)$ such that $y\in S_n x\cap C[\la]$.
Since the morphism $f_\la$ is finite, we obtain
\begin{equation}\label{N-O-completion-decomposition}
\Nscr_\la(\Oscr)\sphat_x=\left(\bigoplus_{y\in S_n x\cap C[\la]}\hat{\Oscr}_{Z_{\la},(y,x)}\right)^{W_\la}=
\bigoplus_{i=1}^N\left(\hat{\Oscr}_{Z_\la,(y_i,x)}\right)^{W_\la\cap\St_{y_i}},
\end{equation}
where $y_1,\ldots,y_N$ are representatives of $W_\la$-orbits on $S_n x\cap C[\la]$.
For each $y_i$ let us fix $w_i\in S_n$ such that $y_i=w_ix$. Then 
by \eqref{Z-functoriality-eq}, setting $\Pscr=\Pscr^\la$, we get an isomorphism
\begin{equation}\label{Z-y-i-x-eq}
\hat{\Oscr}_{Z_\la,(y_i,x)}\simeq \hat{\Oscr}_{Z_{w_i^{-1}\Pscr},(x,x)}
\end{equation}
of $\hat{\Oscr}_{C^n_x}[\St_x]$-modules, compatible with the commuting action of
$W_\la\cap\St_{y_i}\simeq W_{w_i^{-1}\Pscr}\cap\St_x$.

Let 
\begin{equation}\label{x-S-q-decomp-eq}
\{1,\ldots,n\}=\sqcup_{q\in Q}S(q)
\end{equation} 
be the decomposition
corresponding to the subsets of equal coordinates of $x$. We have $x=(x(q))_{q\in Q}$ with $x(q)\in C\sub C^{S(q)}$.
For each $i=1,\ldots,N$, let $(\Pscr_i(q))_{q\in Q}$ be the decompositions associated with
$w_i^{-1}\Pscr$ and the point $x\in C[w_i^{-1}\Pscr]$, where $\Pscr_i(q)$ is a decomposition of $S(q)$.
The subgroup $\St_x\sub S_n$ consists of all permutations preserving each subset $S(q)$.
Applying Lemma \ref{Z-lambda-prod-lem} we get
\begin{equation}\label{Z-w-i-x-prod-eq}
\hat{\Oscr}_{Z_{w_i^{-1}\Pscr},(x,x)}\simeq \hat{\bigotimes}_{q\in Q}\hat{\Oscr}_{Z_{\Pscr_i(q)},(x(q),x(q))}
\end{equation}
compatible with $\hat{\Oscr}_{C^n,x}=\hat{\bigotimes}_{q\in Q}\hat{\Oscr}_{C^{S(q)},x(q)}$-module
structure and with $\St_x=\prod S_{|S(q)|}$-action. The action of $W_{w_i^{-1}\Pscr}\cap\St_x$ on the left-hand side
of \eqref{Z-w-i-x-prod-eq}
corresponds to the action of $\prod_{q\in Q}W_{\Pscr_i(q)}$ on the right-hand side. 
Also, by \eqref{Z-functoriality-eq}, we can replace
$Z_{\Pscr_i(q)}$ with $Z_{\la(i,q)}$ in the right-hand side of \eqref{Z-w-i-x-prod-eq}, where
$\la(i,q)$ is the partition given by the sizes of the subsets forming $\Pscr_i(q)$. Hence, 
combining isomorphisms \eqref{Z-y-i-x-eq} and \eqref{Z-w-i-x-prod-eq}
we get
$$\left(\hat{\Oscr}_{Z_\la,(y_i,x)}\right)^{W_\la\cap\St_x}\simeq 
\hat{\bigotimes}_{q\in Q} \Nscr_{\la(i,q)}(\Oscr)\sphat_{x(q)}.$$
Thus, we can rewrite \eqref{N-O-completion-decomposition} as
\begin{equation}\label{N-O-decomposition-2}
\Nscr_\la(\Oscr)\sphat_x\simeq\bigoplus_{i=1}^N \hat{\bigotimes}_{q\in Q} \Nscr_{\la(i,q)}(\Oscr)\sphat_{x(q)}.
\end{equation}
We have a similar decomposition for $\mu$:
$$\Nscr_\mu(\Oscr)\sphat_x\simeq\bigoplus_{j=1}^M \hat{\bigotimes}_{q\in Q} \Nscr_{\mu(j,q)}(\Oscr)\sphat_{x(q)},$$
where $z_1,\ldots,z_M$ are representatives of $W_\mu$-orbits on $S_nx\cap C[\mu]$.
Linearizing the situation in the formal neighborhood of each $x(q)\in C\sub C^{I_q}$ and applying our local results
(see Theorem \ref{type-A-thm} and Proposition \ref{modules-KP-prop}),
we see that $\Ext^*(\Nscr_\la(\Oscr)\sphat_x,\Nscr_\mu(\Oscr)\sphat_x)^{\St_x}$ can be nonzero only if for
some $i$ and some $j$ one has $\la(y_i,q)\le\mu(z_j,q)$ for all $q\in Q$.
By Lemma \ref{part-merge-lem}, this implies that $\la\le\mu$.

\noindent
(ii) The commutative diagram
\begin{equation}\label{g-lambda-com-diag}
\begin{diagram}
[\ov{Z}_\la/S_n]&\rTo{\ov{f}_\la}& [C^n/S_n]\\
\dTo{q_\la}&&\dTo{\pi}\\
C[\la]/W_\la&\rTo{g_\la}&C^n/S_n
\end{diagram}
\end{equation}
gives rise to a homomorphism of sheaves of algebras 
$$\pi^*g_{\la,*}\Oscr\to \ov{f}_{\la,*}\Oscr.$$
which corresponds by adjunction to the natural homomorphism
$$g_{\la,*}\Oscr\to \pi_*\ov{f}_{\la,*}\Oscr\simeq g_{\la,*}q_{\la,*}\Oscr.$$
Hence, we have the induced homomorphism 
$$\pi^*g_{\la,*}\Oscr\to R\und{\Hom}\bigl(\Nscr_\la(\Oscr),\Nscr_\la(\Oscr)\bigr),$$
or by adjunction, 
$$g_{\la,*}\Oscr\to \pi_*R\und{\Hom}\bigl(\Nscr_\la(\Oscr),\Nscr_\la(\Oscr)\bigr)^{S_n}.$$
To show that this is an isomorphism, we can pass to completions of the
stalks at a point $\pi(x)\in C^n/S_n$ for some $x\in C[\la]$.

First, note that the fiber $g_\la^{-1}(\pi(x))$ consists of $W_\la$-orbits on $S_nx\cap C[\la]$, so that
$$\left(g_{\la,*}\Oscr\right)\sphat_{\pi(x)}\simeq\bigoplus_{i=1}^N \hat{\Oscr}_{C[\la],y_i}^{W_\la}
\simeq \bigoplus_{i=1}^N \hat{\Oscr}_{C[w_i^{-1}\Pscr],x}^{W_{w_i^{-1}\Pscr}}.$$
Next, we claim that for each $i\neq i'$ there exists $q\in Q$ such that
$$\Ext^*\bigl(\Nscr_{\la(i,q)}(\Oscr)\sphat_{x(q)},\Nscr_{\la(i',q)}(\Oscr)\sphat_{x(q)}\bigr)^{S_{|S(q)|}}=0.$$
Indeed, otherwise by Theorem \ref{type-A-thm}, 
we would get $\la(i,q)\le\la(i',q)$ for each $q$. Since merging the partitions $(\la(i,q))_{q\in Q}$
(resp., $(\la(i',q))_{q\in Q}$) gives the same partition $\la$,
this can happen only if $\la(i,q)=\la(i',q)$ for all $q\in Q$ (using the last assertion of Lemma \ref{part-merge-lem}).
But this is impossible by Lemma \ref{W-lambda-orbit-lem}, since $y_i$ and $y_{i'}$ are not in the same $W_\la$-orbit.
Thus, using \eqref{N-O-decomposition-2}, we get an isomorphism
$$\bigl(\pi_*R\und{\Hom}(\Nscr_\la(\Oscr),\Nscr_\la(\Oscr))\sphat_x\bigr)^{S_n}=
\bigoplus_{i=1}^N R\End\bigl(\hat{\bigotimes}_{q\in Q} \Nscr_{\la(i,q)}(\Oscr)\sphat_{x(q)}\bigr)^{\St_x},$$
and the assertion follows from the isomorphisms 
$$\hat{\Oscr}_{C[\la(i,q)],x(q)}^{W_{\la(i,q)}}\to R\End\bigl(\Nscr_{\la(i,q)}(\Oscr)\sphat_{x(q)}\bigr)^{S_{|S(q)|}},$$
which follow from Theorem \ref{type-A-thm}.

\noindent
(iii) It is enough to prove that for such $F$ one has $\hat{F}_x=F_x\ot \hat{\Oscr}_{C^n,x}=0$ for every $x\in C^n$.
Let us consider the decomposition \eqref{x-S-q-decomp-eq} into subsets of equal coordinates of $x$, so that
$x=(x(q))_{q\in Q}$, with $x(q)\in C\sub C^{S(q)}$. Now let us choose any collection $(\Pscr(q))_{q\in Q}$ of decompositions
of $S(q)$. Let $\Pscr$ be the decomposition of $\{1,\ldots,n\}$ obtained by merging all $\Pscr(q)$.
Then $\Pscr=w\Pscr^\la$ for some $w\in S_n$ and some partition $\la$, $|\la|=n$. Hence, by assumption,
\begin{equation}\label{N-lambda-F-x-van}
\bigl(\pi_*R\und{\Hom}(\Nscr_\la(\Oscr),F)\sphat_x\bigr)^{S_n}=
R\und{\Hom}\bigl(\Nscr_\la(\Oscr)\sphat_x,\hat{F}_x\bigr)^{\St_x}=0.
\end{equation}
By \eqref{Z-functoriality-eq}, we have
$$\Nscr_\la(\Oscr)\sphat_x=\left(f_{\Pscr,*}\Oscr_{Z_{\Pscr}}\right)\sphat_x.$$
Arguing as in part (i) we get a decomposition of $\hat{\Oscr}_{C^n,x}[\St_x]$-modules
$$\left(f_{\Pscr,*}\Oscr_{Z_{\Pscr}}\right)\sphat_x=\bigoplus_{i=1}^N\left(\hat{\Oscr}_{Z_{\Pscr},(y_i,x)}\right)^{W_\Pscr\cap\St_{y_i}},$$
where $y_i$ are representatives of $W_\Pscr$-orbits on $S_nx\cap C[\Pscr]$, and we can assume that $y_1=x$.
Thus, the vanishing \eqref{N-lambda-F-x-van} implies that
$$\Hom_{\hat{\Oscr}_{S_n,x}}\bigl(\bigl(\hat{\Oscr}_{Z_{\Pscr},(x,x)}\bigr)^{W_\Pscr\cap\St_x},\hat{F}_x\bigr)^{\St_x}=0.$$
But, as we have shown in part (i),
$$\bigl(\hat{\Oscr}_{Z_{\Pscr},(x,x)}\bigr)^{W_\Pscr\cap\St_x}\simeq\hat{\bigotimes}_{q\in Q}
\Nscr_{\la(i,q)}(\Oscr)\sphat_{x(q)}.$$
By construction, $(\la(i,q))_{q\in Q}$ independently run over all partitions of $|S(q)|$. Hence, the assertion follows from
the local Theorem \ref{type-A-thm}.
\ed

\noindent
{\it Proof of Theorem B}.
Let $L$ be an ample line bundle on $C^n/S_n$. Then for every $\la$,
$g_\la^*L$ is an ample line bundle on $C[\la]/W_\la$ (since $g_\la$ is a finite morphism).
Therefore, the line bundles $(g_\la^*L^m)_{m\in\Z}$ strongly generate the derived category
$D^b(C[\la]/W_\la)$. Therefore, to check the semiorthogonality of images of $\Nscr_\la$ and $\Nscr_\mu$,
it is enough to check that
$$R\Hom_{D^b_{S_n}(C^n)}(\Nscr_\la(g_\la^*L^m),\Nscr_\mu(g_\mu^*L^k))=0 \ \ \text{ for } \la\not\le\mu$$
and for any $m,k\in\Z$.
Next, we observe that due to commutativity of diagram \eqref{g-lambda-com-diag}, one has
$$\Nscr_\la(g_\la^*L^m)=\ov{f}_{\la,*}q_\la^*g_\la^*L^m\simeq\ov{f}_{\la,*}(\ov{f}_\la^*\pi^*L^m)\simeq
\ov{f}_{\la,*}(\Oscr)\ot\pi^*L^m=\Nscr_\la(\Oscr)\ot\pi^*L^m.$$
Thus,
\begin{align*}
&R\Hom_{D^b_{S_n}(C^n)}(\Nscr_\la(g_\la^*L^m),\Nscr_\mu(g_\mu^*L^k))\simeq
R\Hom_{D^b_{S_n}(C^n)}(\Nscr_\la(\Oscr),\Nscr_\mu(\Oscr)\ot \pi^*L^{k-m}))\simeq\\
&R\Ga(C^n,R\und{\Hom}(\Nscr_\la(\Oscr),\Nscr_\mu(\Oscr))\ot\pi^*L^{k-m})^{S_n}\simeq\\
&R\Ga(C^n/S_n,\left(\pi_*R\und{\Hom}(\Nscr_\la(\Oscr),\Nscr_\mu(\Oscr))\right)^{S_n}\ot L^{k-m}),
\end{align*}
and the required vanishing follows from Lemma \ref{local-semiorth-lem}(i).

The fact that $\Nscr_\la$ is fully faithful follows from Lemma \ref{local-semiorth-lem}(ii) in a similar fashion.
Finally, once we know this, we deduce that the images of $\Nscr_\la$ for all $\la$ generate an admissible
subcategory of $D^b_{S_n}(C^n)$, so it is enough to check that its right orthogonal is zero. Suppose $F\in D^b_{S_n}(C^n)$
is such that $R\Hom_{D^b_{S_n}(C_n)}(\Nscr_\la(L^m),F)=0$ for all $m\in\Z$ and all partitions $\la$. 
Then, as before, we can rewrite this
as the vanishing
$$R\Ga(C^n/S_n,\left(\pi_*R\und{\Hom}(\Nscr_\la(\Oscr),F)\right)^{S_n}\ot L^m)=0$$
for all $m$ and $\la$. Since $L$ is ample, this implies the vanishing of
$\left(\pi_*R\und{\Hom}(\Nscr_\la(\Oscr),F)\right)^{S_n}$, and the
assertion follows from Lemma \ref{local-semiorth-lem}(iii).
\ed

\subsection{Semiorthogonal decompositions for actions of $(\Z/2)^n$ and $(\Z/2)^n\rtimes S_n$}\label{A1-sec}

Let $X$ be a smooth variety equipped with an involution $i:X\to X$ such that
the fixed locus of $i$ is a smooth divisor $X^i\sub X$, or equivalently, the corresponding
geometric quotient $Y=X/(\Z/2)$ is smooth. In this case the derived category $D^b_{\Z/2}(X)$ of $\Z/2$-equivariant
coherent sheaves on $X$ has a natural semiorthogonal decomposition (see \cite[Thm.\ 5.1]{CP})
\begin{equation}\label{main-Z-2-decomp-eq}
D^b_{\Z/2}(X)=\lan \pi^*D^b(Y), i_*D^b(X^i) \ran,
\end{equation}
where $\pi:X\to Y$ is the projection, the image of the functor $\pi^*$ (resp. $i_*$) is equipped with
the natural (resp., trivial) $\Z/2$-action.

More generally, if we have a collection $(X_a,i_a)_{1\le a\le n}$ of varieties with involutions as above,
then $X_1\times\ldots\times X_n$ has a natural $(\Z/2)^n$-action and the corresponding derived
category of equivariant sheaves has a semiorthogonal decomposition with pieces of the form
$$D^b(\prod_{a\in S} Y_a\times \prod_{a\not\in S} X_a^{i_a}),$$
where $Y_a=X_a/(\Z/2)$ and $S$ ranges over all subsets of $\{1,\ldots,n\}$.

In the case when $(X_a,i_a)=(X,i)$ for all $a=1,\ldots,n$ we get an action of the bigger group
$$W=W_{B_n}=(\Z/2)^n\rtimes S_n$$ 
on $X^n$, where $(\Z/2)^n$ acts as above and $S_n$ acts by permuting the factors.
Now we can lump together some pieces of the above semiorhogonal
decomposition to get a semiorthogonal decomposition of the category $D^b_W(C^n)$.

Set $R=X^i$, and for each $j$, $0\le j\le n$, consider the natural diagram
\begin{diagram}
&&[(R^j\times X^{n-j})/W_j]&&\\
&\ldTo{\id\times\pi^{n-j}}&&\rdTo{q_j}\\
[(R^j\times Y^{n-j})/(S_j\times S_{n-j})]&&&&[X^n/W]
\end{diagram}
where $W_j=(\Z/2)^n\rtimes (S_j\times S_{n-j})\sub W$, and $q_j$ is induced by the natural
closed embedding $R^j\times X^{n-j}\hra X^n$.
Let 
$$\Phi_j=(q_j)_*(\id\times\pi^{n-j})^*:D^b_{S_j\times S_{n-j}}(R^j\times Y^{n-j})\to D^b_W(X^n)$$
be the corresponding functor between the derived categories.

\begin{prop}\label{global-B-n-prop}
The functors $\Phi_j$ are fully faithful, and we have a semiorthogonal decomposition
$$D^b_{W}(X^n)=\lan \Phi_0(D^b_{S_n}(Y^n)), \ldots, 
\Phi_j(D^b_{S_j\times S_{n-j}}(R^j\times Y^{n-j})), \ldots, \Phi_n(D^b_{S_n}(R^n))\ran.$$
\end{prop}

Let us denote the normal subgroup $(\Z/2)^n\sub W$ by $H$. 
The regular representation of $H$ decomposes as a $W$-module into the direct sum
of $W$-representations as follows:
$$k[H]=\bigoplus_{j=0}^n V_j,$$
where $V_j$ is the direct sum of all characters of $H$ that are nontrivial on exactly $j$ basis elements
in $(\Z/2)^n$ (so $V_0$ is the trivial representation of $W$).
Hence, the multiplication map
$$k[H]\ot k[S_n]\rTo{\sim} k[W]$$
induces an isomorphism of $W$-modules
\begin{equation}\label{Bn-regular-representation}
\bigoplus_{j=0}^n V_j\ot k[S_n] \rTo{\sim} k[W].
\end{equation}

\begin{proof}[Proof of Proposition \ref{global-B-n-prop}] 
For an $S_j\times S_{n-j}$-equivariant coherent sheaf $\Fscr$ on 
$R^j\times Y^{n-j}$ one has
$$\Phi_j(\Fscr)=\bigoplus_{J, |J|=j} (i_J)_*(\id\times\pi^{n-j})^*\Fscr,$$
where $J$ ranges over subsets of $\{1,\ldots,n\}$ of cardinality $j$,
$$i_J:R^j\times X^{n-j}\to X^n$$
is the embedding sending the factors $R^j$ to the components numbered by the subset $J$
and preserving the order of components in $R^j$ and $X^{n-j}$.

Let $\Gscr$ be a $S_l\times S_{n-l}$-equivariant coherent sheaf on 
$R^l\times Y^{n-l}$.
We have to check that $\Ext^*_W(\Phi_j(\Fscr),\Phi_l(\Gscr))=0$ for $l<j$ and
that for $l=j$ the natural map
$$\Ext^*_{S_j\times S_{n-j}}(\Fscr,\Gscr)\to \Ext^*_W(\Phi_j(\Fscr),\Phi_l(\Gscr))=0$$
is an isomorphism.
The idea is to use repeatedly the semiorthogonality of the subcategories in
\eqref{main-Z-2-decomp-eq} for appropriate involutions.
More precisely, we have
$$\Ext^*_W(\Phi_j(\Fscr),\Phi_l(\Gscr))=
\Ext^*_{W_l}\bigl(\bigoplus_{J,|J|=j}(i_J)_*(\id\times\pi^{n-j})^*\Fscr,(i_L)_*(\id\times\pi^{n-l})^*\Gscr\bigr),$$
where $L=\{1,\ldots,l\}$.
We claim that if $l<j$ then all such $\Ext$'s vanish even after replacing $W_l$ with
$(\Z/2)^n$. Indeed, it suffices to check the vanishing
\begin{equation}\label{Ext-Z2-n-vanishing-eq}
\Ext^*_{(\Z/2)^n}\bigl((i_J)_*(\id\times\pi^{n-j})^*\Fscr,(i_L)_*(\id\times\pi^{n-l})^*\Gscr\bigr)=0
\end{equation}
for any $J$ such that $|J|=j$. 
To this end we observe that there exists an element $s\in J\setminus L$.
Then the above vanishing  immediately follows 
by considering the semiorthogonality \eqref{main-Z-2-decomp-eq}
for the $\Z/2$-action on the $s$-th component of the product $X^n$.

A similar argument shows that the vanishing \eqref{Ext-Z2-n-vanishing-eq} holds for $l=k$ and
$J\neq L$. Thus, in the case $l=j$ we obtain
\begin{align*}
&\Ext^*_{(\Z/2)^n}
\bigl(\bigoplus_{J,|J|=j}(i_J)_*(\id\times\pi^{n-j})^*\Fscr,(i_L)_*(\id\times\pi^{n-j})^*\Gscr\bigr)=\\
&\Ext^*_{(\Z/2)^n}\bigl((i_L)_*(\id\times\pi^{n-j})^*\Fscr,(i_L)_*(\id\times\pi^{n-j})^*\Gscr\bigr)=\\
&\Ext^*(\Fscr,\Gscr),
\end{align*}
where the last equality is obtained from the repeated application of the fully faithfulness in
\eqref{main-Z-2-decomp-eq}. Passing to $S_j\times S_{n-j}$-invariants we deduce our claim.

Next, we need to show that each subcategory $\im(\Phi_j)$ is admissible. It suffices to show that the
left and right adjoint functors to $\Phi_j$ exist. But this immediately 
follows from the existence of left and right adjoints for the functors 
$(q_j)_*$ and $(\pi^{n-j})^*$.

Finally, let us check that the subcategories $\Cscr_j=\Phi_j(D^b_{S_j\times S_{n-j}}(R^j\times Y^{n-j}))$,
$j=0,\ldots,n$, generate $D^b_W(X^n)$. Let $L$ be an ample $S_n$-equivariant line bundle on
$Y^n$. Since the powers of the $W$-equivariant line bundle $(\pi^n)^*L$ generate
the non-equivariant category $D^b(X^n)$, it is enough to check that the $W$-equivariant bundles
$(\pi^n)^*L^i\ot k[W]$ belong to the subcategory generated by $(\Cscr_j)$.
Using the decomposition \eqref{Bn-regular-representation} we can write
$$(\pi^n)^*L^i\ot k[W]=\bigoplus_j (\pi^n)^*(L^i\ot k[S_n])\ot V_j.$$
Thus, it is enough to prove that for any $S_n$-equivariant bundle $E$ and any $k=0,\ldots,n$,
the $W$-equivariant bundle $(\pi^n)^*E\ot V_k$ belongs to the subcategory generated by $(\Cscr_j)$.
Furthermore, since tensoring with the $W$-bundles of the form $(\pi^n)^*E$ preserves all the subcategories
$\Cscr_j$, we only have to consider the bundles $\Oscr_{X^n}\ot V_k$.
Let us start with the exact sequence of $\Z/2$-equivariant sheaves on $X$,
$$0\to \Oscr_X\ot\xi\to \Oscr_X(R)\ot \xi\to \Oscr_R\to 0.$$
where $\xi$ is a nontrivial character of $\Z/2$. Since $\Z/2$ acts trivially on the fibers of the 
$\Z/2$-equivariant line bundle $\Oscr_X(R)\ot\xi$ at fixed points, there exists a line bundle $L$ on $Y$ such that
$\Oscr_X(R)\ot\xi\simeq \pi^*L$. Thus, the above sequence can be rewritten as
\begin{equation}\label{Z-2-equiv-sequence}
0\to \Oscr_X\ot\xi\to \pi^*L\to \Oscr_R\to 0.
\end{equation}
Let $\xi_1,\ldots,\xi_n$ be the natural basis of characters of $(\Z/2)^n$.
Taking the exterior tensor product of $k$ sequences \eqref{Z-2-equiv-sequence} and 
pulling back via the projection $X^n\to X^k$ we obtain a resolution for
$\Oscr_{X^n}\ot \xi_1\ldots\xi_k$ by $(\Z/2)^n$-equivariant sheaves of the form
$$0\to \Oscr_{X^n}\ot \xi_1\ldots\xi_k\to\Fscr_0\to\Fscr_1\to\ldots\to \Fscr_k\to 0,$$
where
$$\Fscr_j=\bigoplus_{J\sub\{1,\ldots,k\}, |J|=j}(i_J)_*(\id\times\pi^{n-j})^*
\left(\Oscr_{R^j}\boxtimes L^{\boxtimes k-j}\boxtimes \Oscr_{Y^{n-k}}\right),$$
where $i_J:R^j\times X^{n-j}\to X^n$ are the embeddings considered before.
Finally, taking direct sum of such resolutions over all subsets in $\{1,\ldots,n\}$
of cardinality $k$ we get a resolution of $\Oscr_{X^n}\ot V_k$ by $W$-equivariant sheaves of the form
$$\tilde{\Fscr}_j=\Phi_j\left(\bigoplus_{K\sub\{1,\ldots,n-j\},|K|=k-j}\Oscr_{R^j}\boxtimes L^{\boxtimes K}\right)$$
where $L^{\boxtimes K}$ is the line bundle on $Y^{n-j}$ obtained as the tensor product of pull-backs of $L$
in the components corresponding to $K$.
\end{proof}

Let $C$ be a smooth connected curve over $\C$ equipped with a non-trivial involution $i:C\to C$, let
$\pi:C\to Q$ be the corresponding double covering, where $Q=C/(\Z/2)$, and let $R=C^i$ be
the set of ramification points of $\pi$.
Then as above the group $W=(\Z/2)^n\rtimes S_n$ acts naturally on $C^n$.
Now we can combine the decomposition of Proposition \ref{global-B-n-prop} with
the semiorthogonal decompositions of
$D^b_{S_j\times S_{n-j}}(R^j\times Q^{n-j})$ constructed in Section \ref{global-A-n-sec}.
More precisely, for each $j=0,\ldots,n$ we have to combine such a decomposition with the natural
orthogonal decomposition coming from the decomposition of $R^j$ into $S_j$-orbits.
The resulting semiorthogonal decomposition of $D^b_W(C^n)$ is described in Theorem \ref{global-Bn-thm} below.
It will be further generalized in Theorem \ref{Gn-thm}.

Let $R=\{p_1,\ldots,p_r\}$. 
For each collection of non-negative integers $k_1,\ldots,k_r$ such that $k_1+\ldots+k_r=j\le n$
let us consider the functor
$$\Phi_{k_1,\ldots,k_r}=(q_{k_1,\ldots,k_r})_*(\pi^{n-j})^*:
D^b_{S_{k_1}\times\ldots\times S_{k_r}\times S_{n-j}}(Q^{n-j})\to D^b_W(C^n)$$
associated with the correspondence
\begin{diagram}
&&[C^{n-j}/W_{k_1,\ldots,k_r}]&&\\
&\ldTo{\pi^{n-j}}&&\rdTo{q_{k_1,\ldots,k_r}}\\
[(Q^{n-j})/(S_{k_1}\times\ldots\times S_{k_r}\times S_{n-j})]&&&&[C^n/W]
\end{diagram}
where $W_{k_1,\ldots,k_r}=(\Z/2)^n\rtimes (S_{k_1}\times\ldots\times S_{k_r}\times S_{n-j})\sub W$, and 
$q_{k_1,\ldots,k_r}$ is induced by the closed embedding 
$$C^{n-j}\to C^n: (x_1,\ldots,x_{n-j})\mapsto (p_1^{k_1},\ldots,p_r^{k_r},x_1,\ldots,x_{n-j}).$$

\begin{thm}\label{global-Bn-thm} The functors $\Phi_{k_1,\ldots,k_r}$, for $k_1+\ldots+k_r\le n$, 
are fully faithful and the subcategories
$\im(\Phi_{k_1,\ldots,k_r})\sub D^b_W(C^n)$ are admissible. We have 
$$\Ext^*(\im(\Phi_{k_1,\ldots,k_r}),\im(\Phi_{k'_1,\ldots,k'_r}))=0$$
unless $k_1\le k'_1,\ldots,k_r\le k'_r$. 
We have a semiorthogonal decomposition of $D^b_W(C^n)$ into these subcategories ordered compatibly with
the above semiorthogonalities.
Furthermore, for each $(k_1,\ldots,k_r)$ we have a further semiorthogonal decomposition of
$$\im(\Phi_{k_1,\ldots,k_r})\simeq D^b_{S_{k_1}\times\ldots\times S_{k_r}\times S_{n-j}}(Q^{n-j}),$$
where $j=k_1+\ldots+k_r$, 
with pieces numbered by collections of partitions $\mu^{(1)},\ldots,\mu^{(r)},\nu$, where
$|\mu^{(i)}|=k_i$, $|\nu|=n-j$, such that the corresponding sheaves are supported
on the stratum in $Q^{n-j}$ associated with $\nu$.
\end{thm}

Note that in the case when $R$ is one point the semiorthogonal
decomposition of the above theorem is numbered by pairs of partitions $(\mu,\nu)$ such that 
$|\mu|+|\nu|=n$. Thus, in the case when $C=\AA^1$ and $i(x)=-x$ we get a semiorthogonal
decomposition for the standard representation of
the Weyl group of type $B_n$ on its vector representation.

\subsection{Decompositions for actions of $G^n\rtimes S_n$}\label{global-B-sec}

More generally, assume that we have a smooth connected curve $C$ with an effective action of a
finite group $G$, so that the quotient $Q=C/G$ is still smooth. 
%Assume that the characteristic of the ground field is zero.
Let $R=D_1\sqcup\ldots\sqcup D_r$ be the decomposition  into $G$-orbits of the ramification locus
of the projection $\pi:C\to Q$ (note that $D_i$'s are also fibers of $\pi$).
For each $i=1,\ldots,r$, pick a point $p_i\in D_i$.
Note that the stabilizer subgroup $G_i\sub G$ of $p_i$ is a cyclic
group of some order $m_i$ (since we are in characteristic zero). 
The following result was proved
in \cite[Thm.\ 1.2]{P-orbifold}:

\begin{thm}\label{CP-thm} For each $i$ the collection of $G$-equivariant sheaves 
\begin{equation}\label{curve-exc-coll}
\Oscr_{(m_i-1)D_i},\ldots,\Oscr_{2D_i},\Oscr_{D_i}
\end{equation}
is exceptional in $D_G(C)$. Let $\Bscr_i\sub D_G(C)$ be the subcategory generated by the collection
\eqref{curve-exc-coll}.
Then the subcategories $\Bscr_1,\ldots,\Bscr_r$
are mutually orthogonal and we have a semiorthogonal decomposition
\begin{equation}
D^b_G(C)=\lan \pi^*D^b(Q), \Bscr_1,\ldots,\Bscr_r\ran.
\end{equation}
\end{thm}

\begin{rem}\label{curves-rem}
It is not hard to see that the pieces of the semiorthogonal decomposition of Theorem \ref{CP-thm} (where we decompose
each $\Bscr_i$ into the subcategories generated by the individual exceptional sheaves) matches with the motivic
decomposition \eqref{mot-dec2-eq}, and thus confirms Conjecture A in this case. Since all the motivic pieces except for $Q$ are
points, this is equivalent to the identity
$$\sum_{g\in G\setminus\{1\}/\sim} |C(g)\backslash C^g|=\sum_{i=1}^r (m_i-1).$$
One way to prove this is to compare 
the corresponding decomposition of Hochschild homology of $[C/G]$ with the general decomposition 
\eqref{HH-decomposition}.
A more direct matching is obtained by comparing contributions to both sides corresponding to each special fiber 
$D_i$ of $\pi$. More precisely if $H=G_i$, the stabilizer subgroup of one of the ramification
points $p_i$ then the equality between these contributions has form
\begin{equation}\label{conjugacy-identity-eq}
\sum_{h\in H\setminus\{1\}/\sim_G} |C(h)\backslash (G/H)^h|=|H|-1.
\end{equation}
Here $\sim_G$ denotes the equivalence relation ``being conjugate in $G$", and 
$(G/H)^h$ denotes the $h$-invariant $G/H$ (which is in bijection with $C^h\cap D_i$). 
Now we observe that the identity \eqref{conjugacy-identity-eq}
holds for any abelian subgroup $H$ in a finite group $G$ and follows from from the
bijection
$$C(h)\backslash (G/H)^h\to H\cap \Ad(G)(h): C(h)gH\mapsto g^{-1}hg,$$
where $\Ad(G)(h)$ is the conjugacy class of $h$ in $G$.
\end{rem}

Now let us consider the induced action of the group
$$\G:=G^n\rtimes S_n$$
on $C^n$. We are going to construct a semiorthogonal decomposition of the corresponding
equivariant category $D^b_{\G}(C^n)$, which reduces to that of Theorem \ref{global-Bn-thm} in the case
$G=\Z/2$.

We start by constructing a coarser decomposition with pieces numbered by collections of $r$ partitions
$(\la^{(1)},\ldots,\la^{(r)})$ such that $\ell(\la^{(1)})+\ldots+\ell(\la^{(r)})\le n$ and nonzero parts of $\la^{(i)}$ are
restricted by $1\le \la^{(i)}_j\le m_i-1$ (here we denote by $\ell(\la)$ the number of nonzero parts
in a partition $\la$). Recall that for a partition $\la$ we denote 
$W_\la=S_{r_1}\times S_{r_2}\times \ldots \sub S_{\ell(\la)}$,
where $r_1,r_2,\ldots$ are the multiplicities of parts in $\la$.
For each such $(\la^{(1)},\ldots,\la^{(r)})$ with $\ell(\la^{(1)})+\ldots+\ell(\la^{(r)})=j$
we will define a functor
$$\Phi_{\la^{(1)},\ldots,\la^{(r)}}:D^b_{W_{\la^{(1)}}\times\ldots W_{\la^{(r)}}\times S_{n-j}}(Q^{n-j})\to
D^b_{\G}(C^n).$$
Note that here the group $W_{\la^{(1)}}\times\ldots W_{\la^{(r)}}\times S_{n-j}$ acts on $Q^{n-j}$ via
the projection to $S_{n-j}$.
The relevant quotient stacks are connected by the diagram
\begin{diagram}
&&[C^n/\G_{\la^{(1)},\ldots,\la^{(r)}}]&&\\
&\ldTo{\pi^{n-j}\pr_{n-j}}&&\rdTo{q}\\
[(Q^{n-j})/(W_{\la^{(1)}}\times\ldots W_{\la^{(r)}}\times S_{n-j})]&&&&[C^n/\G]
\end{diagram}
where 
$$\G_{\la^{(1)},\ldots,\la^{(r)}}=G^n\rtimes(W_{\la^{(1)}}\times\ldots W_{\la^{(r)}}\times S_{n-j})\sub \G,$$
$\pr_{n-j}:C^n\to C^{n-j}$ is the projection onto the last $n-j$ coordinates and $q$ is the natural projection.
For each $i$ let us consider the coherent sheaf
$$K_{\la^{(i)}}:=\Oscr_{(m_i-\la^{(i)}_1)D_i}\boxtimes\Oscr_{(m_i-\la^{(i)}_2)D_i}\boxtimes\ldots$$
on $C^{\ell(\la^{(i)})}$.
Note that it has a natural $G^{\ell(\la^{(i)})}\rtimes W_{\la^{(i)}}$-equivariant
structure. Now we define a $G^j\rtimes(W_{\la^{(1)}}\times\ldots\times W_{\la^{(r)}})$-equivariant sheaf $K$ on $C^j$ by 
$$K_{\la^{(1)},\ldots,\la^{(r)}}=K_{\la^{(1)}}\boxtimes\ldots\boxtimes K_{\la^{(r)}}$$
and set
$$\Phi_{\la^{(1)},\ldots,\la^{(r)}}(\Fscr)=q_*(K_{\la^{(1)},\ldots,\la^{(r)}}\boxtimes (\pi^{n-j})^*\Fscr).$$
For a pair of partitions $\la=(\la_1\ge\la_2\ge\ldots)$ and $\mu=(\mu_1\ge\mu_2\ge\ldots)$ 
we write $\la\sub \mu$ if $\la_j\le \mu_j$ for each $j$.

\begin{thm}\label{Gn-thm} 
The functors $\Phi_{\la^{(1)},\ldots,\la^{(r)}}$ are fully faithful, and
their images are admissible subcategories. 
We have
\begin{equation}
\Ext^*_{\G}(\im(\Phi_{\la^{(1)},\ldots,\la^{(r)}}),\im(\Phi_{\mu^{(1)},\ldots,\mu^{(r)}}))=0
\end{equation}
unless $\la^{(i)}\sub \mu^{(i)}$ for $i=1,\ldots,r$.
Furthermore, we have a semiorthogonal decomposition of $D^b_{\G}(C^n)$ into
these subcategories, ordered compatibly with the above semiorthogonalities.
Each subcategory $\im(\Phi_{\la^{(1)},\ldots,\la^{(r)}})$ has a further semiorthogonal decomposition
into pieces numbered by irreducible representations of 
$W_{\la^{(1)}}\times\ldots\times W_{\la^{(r)}}\times S_{n-j}$ 
(where $j=\ell(\la^{(1)})+\ldots+\ell(\la^{(r)})$).
\end{thm}

We start with some auxiliary statements. Suppose a group $H$ acts on a space $X$ and
$H'\sub H$ is a subgroup. Then the push-forward with respect to the natural projection
$[X/H']\to [X/H]$ gives a natural functor between the categories of equivariant sheaves that we denote
$$\Ind_{H'}^H: D^b_{H'}(X)\to D^b_{H}(X).$$
Note that in the definition of the functors $\Phi_{\la^{(1)},\ldots,\la^{(r)}}(\Fscr)$ we have
$q_*=\Ind_{\G_{\la^{(1)},\ldots,\la^{(r)}}}^{\G}$.

\begin{lem}\label{ind-comp-lem} 
Let $a_1,\ldots,a_m$ be nonnegative integers such that $a_1+\ldots+a_m=n$.
For each $i=1,\ldots,m$, let $a_{i1},a_{i2},\ldots$ be nonnegative integers such that
$a_{i1}+a_{i2}+\ldots=a_i$. Set $H=G^n\rtimes(S_{a_1}\times\ldots\times S_{a_m})$.
Then 
\begin{align*}
&\Ind_{H}^\G\left(
\Ind_{G^{a_1}\rtimes(S_{a_{11}}\times  S_{a_{12}}\times\ldots)}^{H}(\Fscr_1)\boxtimes\ldots
\boxtimes
\Ind_{G^{a_m}\rtimes(S_{a_{m1}}\times  S_{a_{m2}}\times\ldots)}^{H}(\Fscr_m)\right)\simeq\\
&\Ind_{G^n\rtimes(S_{a_{11}}\times  S_{a_{12}}\times\ldots\times S_{a_{m1}}\times  S_{a_{m2}}\times\ldots)}^{\G}(\Fscr_1\boxtimes\ldots\boxtimes\Fscr_m).
\end{align*}
\end{lem}

\Pf . This immediately follows from the fact that for a chain of subgroups $H''\sub H'\sub H$,
where $H$ acts on a space $X$, one has $\Ind_{H'}^H\circ \Ind_{H''}^{H'}\simeq\Ind_{H''}^H$.
\ed

\begin{lem}\label{En-gen-lem} 
Let $E$ be a $G$-equivariant vector bundle on $C$. Then the $\G$-equivariant
vector bundle $E^{\boxtimes n}$ belongs to the triangulated subcategory generated by the images
of the functors $\Phi_{\la^{(1)},\ldots,\la^{(r)}}$.
\end{lem}

\Pf . First, we have a canonical exact sequence in $\Coh_G(C)$
$$0\to E\to \pi^*E'\to \Fscr\to 0$$
where $\Fscr$ is some sheaf in the subcategory generated by $\Bscr_1,\ldots,\Bscr_r$ (see Theorem \ref{CP-thm}).
By taking the $n$th exterior tensor power, we get a resolution of $E^{\boxtimes n}$ with terms of the form
$$\Ind_{G^n\rtimes (S_k\times S_{n-k})}^{\G}\left(\Fscr^{\boxtimes k}\boxtimes \pi^*(E')^{\boxtimes n-k}\right)$$
Hence, by Lemma \ref{ind-comp-lem}, it is enough to prove that the assertion is true with $E$ replaced
by $\Fscr$.
Next, we claim that if $\Fscr$ fits into an exact sequence of $G$-equivariant sheaves
$$0\to \Fscr'\to \Fscr\to \Fscr''\to 0$$
and the assertion holds for $\Fscr'$ and $\Fscr''$ then it holds for $\Fscr$. Indeed, we have an sequence
of the form
$$0\to (\Fscr')^{\boxtimes n}\to \Fscr^{\boxtimes n}\to \ldots$$
where other terms are induced by exterior tensor powers involving $<n$ factors $\Fscr$ (and some factors
$\Fscr''$). Hence, our claim follows, using the induction in $n$ and Lemma \ref{ind-comp-lem}.

Thus, the statement reduces to the case when $\Fscr$ is one of the simple $G$-equivariant sheaves
supported on one $G$-orbit $D_i$. 
Note that for $p=1,\ldots,m_i-1$, we have
$$\Oscr_{pD_i}^{\boxtimes s}\in\im(\Phi_{0,\ldots,0,(p)^s,0,\ldots,0)}),$$
where we take $\la^{(i)}=(p)^s$ ($p$ repeated $s$ times) and the remaining partitions are zero.
Also, 
$$\Oscr_{mD_i}=\pi^*\Oscr_{\pi(p_i)},$$ 
so that $\Oscr_{mD_i}^{\boxtimes s}$ is in the image
of $(\pi^s)^*$. Thus, the assertion holds for the exterior tensor powers of the sheaves $\Oscr_{pD_i}$, with $p=1,\ldots,m_i$.
Let $\xi:G_i\to\GG_m$ be the character of the stabilizer subgroup
corresponding to the action on the fiber of $\om_C$ at $p_i$. Then the simple $G$-equivariant sheaves supported
on $D_i$ are of the form $\xi^j\ot \Oscr_{D_i}$, $j=0,\ldots,m_i-1$.
Now the assertion for the exterior tensor powers of these sheaves follows as before by induction on $j$,
using the exact sequences of $G$-equivariant sheaves
$$0\to \xi\ot\Oscr_{D_i}\to \Oscr_{2D_i}\to \Oscr_{D_i}\to 0,$$
$$0\to \xi^2\ot\Oscr_{D_i}\to \Oscr_{3D_i}\to \Oscr_{2D_i}\to 0,$$
etc.
\ed

\noindent {\it Proof of Theorem \ref{Gn-thm}}. 
For 
$$\Fscr\in D^b_{W_{\la^{(1)}}\times\ldots W_{\la^{(r)}}\times S_{n-j}}(Q^{n-j})\ \text{ and }\ 
\Gscr\in D^b_{W_{\la^{(1)}}\times\ldots W_{\la^{(r)}}\times S_{n-k}}(Q^{n-k}),$$ 
where $j=\ell(\la^{(1)})+\ldots+\ell(\la^{(r)})$, $k=\ell(\mu^{(1)})+\ldots+\ell(\mu^{(r)})$,
we have
\begin{align*}
&\Ext^*_{\G}\left(\Phi_{\la^{(1)},\ldots,\la^{(r)}}(\Fscr),\Phi_{\mu^{(1)},\ldots,\mu^{(r)}}(\Gscr)\right)=\\
&\Ext^*_{\G_{\mu^{(1)},\ldots,\mu^{(r)}}}\left(\Ind_{\G_{\la^{(1)},\ldots,\la^{(r)}}}^{\G}(K_{\la^{(1)},\ldots,\la^{(r)}}
\boxtimes (\pi^{n-j})^*\Fscr), K_{\mu^{(1)},\ldots,\mu^{(r)}}\boxtimes (\pi^{n-k})^*\Gscr\right).
\end{align*}
As in the proof of Proposition \ref{global-B-n-prop}, considering only the action of $G^n$
and using the semiorthogonalities of Theorem \ref{CP-thm}, we see that
the above $\Ext^*$ vanishes unless $\la^{(1)}\sub \mu^{(1)},\ldots,\la^{(r)}\sub\mu^{(r)}$
(recall that each part $p$ in $\la^{(i)}$ corresponds to the sheaf $\Oscr_{(m_i-p)D_i}$ on
the corresponding factor). Similarly, in the case when $\la^{(i)}=\mu^{(i)}$, using only $G^n$-invariance
in the above $\Ext^*$ we get $\Ext^*(\Fscr,\Gscr)$ computed in $D^b(Q^{n-j})$.
Taking the invariants with respect to $W_{\la^{(1)}}\times\ldots W_{\la^{(r)}}\times S_{n-j}$
gives $\Ext^*$ in the correct equivariant category, so we deduce that the functor
$\Phi_{\la^{(1)},\ldots,\la^{(r)}}$ is fully faithful.
Note that each category $D^b_{W_{\la^{(1)}}\times\ldots W_{\la^{(r)}}\times S_{n-j}}(Q^{n-j})$ is
saturated, so it remains to show that the images of the functors $\Phi_{\la^{(1)},\ldots,\la^{(r)}}$
generate the category $D^b_{\G}(C^n)$.

Let $L$ be an ample line bundle on $Q$. Then the line bundles $(\pi^*L^i)^{\boxtimes n}$
generate the non-equivariant category $D^b(C^n)$. Hence, to
check that the subcategories $\im(\Phi_{\la^{(1)},\ldots,\la^{(r)}})$ generate $D^b_{\G}(C^n)$,
it is enough to check that the $\G$-equivariant bundles $(\pi^*L^i)^{\boxtimes n}\ot k[\G]$
belong to the subcategory generated by $(\im(\Phi_{\la^{(1)},\ldots,\la^{(r)}}))$.

Let $V_1,\ldots,V_d$ be all non-isomorphic irreducible representations of $G$.
For non-negative integers $m_1,\ldots,m_d$, such that $m_1+\ldots+m_d=n$, consider the representation
$V_1^{\ot m_1}\ot\ldots\ot V_d^{\ot m_d}$ of $G^n\rtimes (S_{m_1}\times\ldots\times S_{m_d})$, and 
define a $\G$-representation by
$$V_{m_1,\ldots,m_d}=
\Ind_{G^n\rtimes (S_{m_1}\times\ldots\times S_{m_d})}^{\G}(V_1^{\ot m_1}\ot\ldots\ot V_d^{\ot m_d}).$$
Note that if $V$ is any $G$-representation then
$V^{\ot n}$ decomposes as a $\G$-representation into a direct sum of representations
$V_{m_1,\ldots,m_d}$ taken with some multiplicities.
On the other hand, we have an isomorphism of $\G$-representations
$$k[G^n]\otimes k[S_n]\simeq k[\G],$$
where $\G$ acts on $k[S_n]$ via the projection $\G\to S_n$, and on $G^n$ by conjugation.
Furthermore, $k[G^n]\simeq k[G]^{\ot n}$, where $G$ acts on $k[G]$ by conjugation.
Hence, every irreducible $\G$-representation is a direct summand in a representation of the form
$V_{m_1,\ldots,m_d}\otimes k[S_n]$.

Thus, it is enough to prove that the $\G$-equivariant bundles
$$(\pi^*L^i)^{\boxtimes n}\ot V_{m_1,\ldots,m_d}\ot k[S_n]$$
belong to the 
subcategory generated by $(\im(\Phi_{\la^{(1)},\ldots,\la^{(r)}}))$.
It is easy to see that tensoring with $k[S_n]$ preserve each subcategory
$\im(\Phi_{\la^{(1)},\ldots,\la^{(r)}})$, so we are reduced to proving the same assertion for
$(\pi^*L^i)^{\boxtimes n}\ot V_{m_1,\ldots,m_d}$.
We have
$$(\pi^*L^i)^{\boxtimes n}\ot V_{m_1,\ldots,m_d}=
\Ind_{G^n\rtimes (S_{m_1}\times\ldots\times S_{m_d})}^{\G}
\left((\pi^*L^i\ot V_1)^{\boxtimes m_1}\ot\ldots\ot (\pi^*L^i\ot V_d)^{\boxtimes m_d}\right).$$
By Lemma \ref{En-gen-lem}, for each $j=1,\ldots,d$ the object
$$(\pi^*L^i\ot V_j)^{\boxtimes m_j}\in D^b_{G^{m_i}\rtimes S_{m_i}}(C^{m_i})$$
belongs to the subcategory generated by $(\im(\Phi_{\la^{(1)},\ldots,\la^{(r)}}))$
(with $n$ replaced by $m_i$).
It remains to apply Lemma \ref{ind-comp-lem}.

We have an orthogonal decomposition of
$$\im(\Phi_{\la^{(1)},\ldots,\la^{(r)}})=D^b_{W_{\la^{(1)}}\times\ldots W_{\la^{(r)}}\times S_{n-j}}(Q^{n-j})$$
numbered by irreducible representations of $W_{\la^{(1)}}\times\ldots W_{\la^{(r)}}$ with each piece
equivalent to $D^b_{S_{n-j}}(Q^{n-j})$. The latter category has a semiorthogonal decomposition
constructed in Section \ref{global-A-n-sec}.
\ed

%Assume that $k=\C$.
Note that in the case $G=\mu_m$ (roots of unity of order $m$) the group $\G$ above becomes the reflection group $G(m,1,n)$,
and in the case when $C=\A^1$ and $\mu_m$ acts on $\A^1$ through the embedding $\mu_m\sub\GG_m$,
the action of $G(m,1,n)$ on $C^n=\A^n$ becomes the standard vector representation of $G(m,1,n)$.
We can check that in this case the pieces of the semiorthogonal decomposition
of Theorem \ref{Gn-thm} match with the pieces of the motivic decomposition \eqref{mot-dec2-eq}.

\begin{cor}\label{Gm1n-cor} 
Conjecture A holds for the standard action of $G(m,1,n)$ on $\A^n$.
\end{cor}

\Pf . The matching of pieces of the semiorthogonal decomposition of Theorem \ref{Gn-thm} (the finer of the two
decompositions considered in this Theorem) with the varieties of $V^g/C(g)$ follows from Propositions 
\ref{HH-semiorth-match-prop} and \ref{star-rank2-Gmpn-prop}. 
\ed

\begin{rem}
Here is an explicit bijection between the pieces of the semiorthogonal decomposition in the above Corollary
and conjugacy classes of $G(m,1,n)$.
In our case $r=1$, so the pieces of the rougher decomposition are numbered by partitions $\la$ such that $\ell(\la)=j\le n$, and 
$\la$ has all parts $\le m-1$. The pieces of the finer decomposition are numbered by further choosing
partitions $\mu^{(1)},\ldots,\mu^{(m-1)}$ with $|\mu^{(a)}|$ equal to the multiplicity of the part $a$ in $\la$, and in addition
a partition $\mu^{(0)}$ such that $|\mu^{(0)}|=n-j$. On the other hand, conjugacy classes in $G(m,1,n)$ are also numbered
by such collections of partitions $(\mu{(0)},\ldots,\mu^{(m-1)})$: $\mu^{(i)}$ records the cycles $C$ in the cycle type of an element with $z(C)=\zeta_m^i$, where $\zeta_m$ is a fixed primitive $m$th root of unity (see \eqref{z-C-eq}).
The corresponding semiorthogonal piece is the same as the piece corresponding to the partition $\mu^{(0)}$ in the semiorthogonal decomposition of $D^b_{S_{n-j}}(\A^{n-j})$, so it is $D^b(\A^{\ell(\mu^{(0)})})$.
\end{rem}

\noindent
{\it Proof of Theorem C}. Most cases are covered by Corollary \ref{Gm1n-cor}. The two remaining cases $G_2$ and $F_4$ 
were considered in Proposition \ref{exceptional-prop}.
\ed
%E.g., in the case $G=\Z/2$, the motivic pieces are numbered by pairs of partitions $(\mu,\nu)$,
%where $\mu$ is obtained by combining parts of $\mu^{(1)},\ldots,\mu^{(r)}$.

%We also claim that for any vector bundle $E$ on $Q$
%tensoring with $(\pi^*E)^{\boxtimes n}$ preserves

\section{Types $B$ and $D$ via Springer correspondence}\label{BD-sec}

In this case we consider in detail what our method in Section \ref{Springer-sec} gives for groups of type $B$ and $D$.
We show that the semiorthogonal decomposition of Theorem \ref{Springer-thm} can be further refined in these cases.
For type $B$ we obtain in this way the decomposition matching the motivic decomposition \eqref{mot-dec2-eq},
while for type $D$ we get a less refined decomposition containing some ``noncommutative" pieces (due to the presence
of singular quotients $X^g/C(g)$).

\subsection{Type $B_n$}\label{Springer-B-n-sec}

We are using the explicit description of the Springer correspondence given in \cite{Lusztig-int-coh} (see also \cite{Carter}).

Recall that representations of $W_{B_n}$ are parametrized by the set $\Pi^{(2)}_n$
of pairs of partitions $(\xi,\nu)$ such that
$|\xi|+|\nu|=n$. We write such a pair of partitions in the form $\xi=(\xi_0\le \xi_1\le\ldots \le \xi_r)$,
$\nu=(\nu_0\le \nu_1\le\ldots\le \nu_{r-1})$, where we complete one of the paritions with zero parts, and 
associate with it the symbol 
\begin{equation}\label{symbol-eq}
\left(\begin{matrix} \xi_0 & & \xi_1+2 & & \xi_2+4 &\ldots & & \xi_r+2r\\
                                         &\nu_0&& \nu_1+2&\ldots& \nu_{r-1}+2(r-1)\end{matrix}\right)
\end{equation}
In \cite{Lusztig-int-coh} this symbol is described as a pair $(R,R')$ of sets of numbers, where $R$ and $R'$ are
the first and the second rows in \eqref{symbol-eq}, respectively.

On the other hand, the nilpotent orbits in $\SO(2n+1)$ are parametrized by the set $\Pi^+_{2n+1}$ of
partitions of $2n+1$ such that each even part appears with
even multiplicity. We denote by $O_\la$ the nilpotent orbit corresponding to $\la\in\Pi^+_{2n+1}$.
For each $m$ let $r_m$ be the multiplicity with which $m$ appears in $\la$. Then the reductive part
of the centralizer $\bG_\la$ of an element in $O_\la$ is 
\begin{equation}\label{B-stabilizer-eq}
R\bG_{\la}=S\left(\prod_{m \text{\ even}} \De_m\Sp(r_m)\times \prod_{m \text{\ odd}} \De_m \Ort(r_m)\right),
\end{equation}
where $\De_m$ denotes the diagonal embedding into product of $m$ identical copies, $S(?)$ denotes passing
to matrices of determinant $1$.
To get the reductive part of $\bG^0_\la$ one has to replace each
$\Ort(r_m)$ by $\SO(r_m)$ (see \cite[Sec.\ 6.1]{CM}).
Let 
$$\II_\la=\{ m \ |\ m \text{ is odd and } r_m>0\}.$$ 
Then from \eqref{B-stabilizer-eq}
we see that the component group $F_\la$ of $\bG_\la$ is given by
\begin{equation}\label{B-component-group-eq}
F_\la=F(\II_\la)\sub (\Z/2)^{\II_\la}, 
\end{equation}
where for a finite set $S$ we define the group $F(S)$ as
the kernel of the addition map
$$(\Z/2)^S\to \Z/2: \sum n_s e_s\mapsto \sum n_s.$$

Note that the algebra 
\begin{equation}\label{B-la-orth-eq}
B_\la:=H^*_{\bG^0_\la}(pt)=\bigotimes_m B_{\la,m}
\end{equation}
is the tensor product of the algebras of polynomials on the maximal tori, invariant with respect to the
Weyl group, for the groups $\Sp(r_m)$ and $\SO(r_m)$, where $m$ ranges over parts in $\la$.
Furthermore, we can explicitly determine the action of $F_\la$ on $B_\la$. 

Let $\II^+_\la\sub\II_\la$ (resp., $\II^-_\la\sub\II_\la$)
be the subset of odd $m$ appearing in $\la$ with even (resp., odd) multiplicity.
Note that if $r$ is odd then we can find an involution in $O(r)\setminus SO(r)$ acting trivially
on the Lie algebra of the maximal torus in $SO(r)$. On the other hand, in the case of even $r$ any involution 
in $O(r)\setminus SO(r)$ gives rise to an outer involution of $SO(r)$.
Furthermore, in this case, setting $r=2k$, we can choose the basic $W$-invariant polynomials to
be $f_1,\ldots,f_{k-1},f_{k}$, where for $i<k$,
$f_i=e_i(x_1^2,\ldots,x_k^2)$ is the $i$th elementary symmetric function in
the squares of the coordinates, and $f_k=x_1x_2\ldots x_k$. Then the outer involution preserves
$f_1,\ldots,f_{k-1}$
and sends $f_k$ to $-f_k$. This leads to the following 

\begin{lem}\label{B-component-group-action-lem}
The group $F_\la$ acts on $B_\la$ through the natural projection
\begin{equation}\label{pi-la-hom}
\pi_\la:F_\la\to F'_\la:=(\Z/2)^{\II^+_\la}.
\end{equation}
The action of $F'_\la$ on $B_{\la,m}$ is trivial if either $m$ is even or $r_m$ is odd. 
For $m\in \II^+_\la$ the action of $F'_\la$ on 
$B_{\la,m}=\C[f_1,\ldots,f_{m/2}]$ factors through the projection to the corresponding $\Z/2$-factor and
sends $f_{m/2}$ to $-f_{m/2}$ while leaving other coordinates invariant.
Thus, we have
$$\Spec(B_\la)\simeq \AA^l\times \AA^N,$$
where $F'_\la=(\Z/2)^l$ acts by the $A_1^l$-type action on $\AA^l$ (and acts trivially on $\AA^N$).
\end{lem}

We define a map 
$$\Spr_1:\Pi^+_{2n+1}\to \Pi^{(2)}_n$$
as follows: for $\la=(\la_0\le\la_1\le\ldots)$ consider $\la'_i=\la_i+i$. Then let
$$\wt{\xi}_0<\wt{\xi}_1<\ldots \ \ (\text{resp.}\ \wt{\nu}_0<\wt{\nu}_1<\ldots)
$$ 
be all the odd (resp., even) numbers among $(\la'_i)$. We set $\xi'_i=(\wt{\xi}_i-1)/2$, 
$\nu'_i=\wt{\nu}_i/2$, and define $\Spr_1(\la)=(\xi,\nu)$, where
$\xi_i=\xi'_i-i$, $\nu_i=\nu'_i-i$ (there can be some zero parts in $\xi$ and $\nu$).

The map $\Spr_1$ describes the image of a pair of the form $(O_\la, \one)$ under the Springer correspondence,
where $\one$ is the trivial character of the component group $F_\la$ of the centralizer $\bG_\la$ of 
an element in $O_\la$.
The pairs $(\xi,\nu)$ in the image of $\Spr_1$ are characterized by the inequalities
$$\xi_0\le\nu_1\le \xi_1+2\le \nu_2+2\le \xi_2+4\le\ldots,$$
i.e., the entries in the symbol \eqref{symbol-eq} are (weakly) increasing from left to right (regardless of the row). 

Let $(R_\la,R'_\la)$ be the row sets of the symbol of $\Spr_1(\la)$. 
The maximal intervals of integers
in $(R_\la \cup R'_\la)\setminus (R_\la\cap R'_\la)$ 
are in natural bijection with the odd numbers appearing as parts in $\la$ (where
the leftmost interval corresponds to the smallest odd part, etc.). Furthermore, the length of the interval
is equal to the multiplicity of the corresponding odd part in $\la$.

Two symbols associated with pairs of partitions
are called similar if they contain the same entries with the same multiplicities.
Each similarity class contains the unique symbol in the image of $\Spr_1$.
Let $\Sscr_\la\sub\Pi^{(2)}_n$ be the set of all pairs of partitions whose symbol is 
similar to the symbol of $\Spr_1(\la)$. The set of symbols of $\Sscr_\la$
has the following explicit description. Recall that
$\II_\la$ is the set of odd numbers $m$ appearing as parts in $\la$.
For each $m\in\II_\la$ let $I_m\sub (R_\la\cup R'_\la)\setminus (R_\la\cap R'_\la)$ be the corresponding 
interval. 
Let us define a function $\ell$ on $\II_\la$ by
$$\ell(m)=|I_m\cap R_\la|-|I_m\cap R'_\la|.$$
Note that $\ell(m)=0$ precisely when $m\in \II^+_\la$, and
$\ell(m)=\pm 1$ for $m\in\II^-_\la$. Consider the subset
$$\Si_\la=\{ (c_m)\in (\Z/2)^{\II_\la} \ |\ \sum_{m: c_m\neq 0}\ell(m)=0\}  \sub (\Z/2)^{\II_\la}.$$ 
Then there is a bijection $\Si_\la\to \Sscr_\la$ which associates with $(c_m)\in\Si_\la$
the symbol $(R,R')$ obtained from $(R_\la,R'_\la)$ by swapping $I_m\cap R_\la$ and $I_m\cap R'_\la$
for each $m$ such that $c_m\neq 0$ (i.e., $I_m\cap R_\la$ goes into the second row
and $I_m\cap R'_\la$ goes into the first row). The condition $(c_m)\in\Si_\la$ ensures
that $|R|=|R'|+1$.

Let $\Spr_\la\sub \widehat{F_\la}$ be the set of characters $\xi$ of $F_\la$
such that $(O_\la,\xi)$ appears in the Springer correspondence
for $\SO(2n+1)$. We have a bijection
$$\Si_\la\rTo{\sim}\Spr_\la$$
defined as follows: we view an element $(c_m)\in \Si_\la$ as a character $\xi$ of $(\Z/2)^{\II_\la}$ and then
restrict it to $F_\la=F(\II_\la)$.  
The image of this character $\xi\in \widehat{F_\la}$ under the Springer correspondence
is precisely the pair of partitions corresponding to the element of $\Sscr_\la$ associated with $(c_m)$.

Note that $|\II^-_\la|$ is odd (since $\la$ is a partition of $2n+1$), so 
$\II^+_\la$ is always a proper subset of $\II_\la$, hence, the homomorphism $\pi_\la$ (see
\eqref{pi-la-hom}) is surjective.

\begin{prop}\label{Bn-Spr-cosets} 
Let $|\II^-_\la|=2k+1$. Then $\Spr_\la$ is the union of ${2k+1\choose k}$ cosets of the subgroup
$$\widehat{\pi}_\la(\widehat{F'_\la})\sub \widehat{F_\la}.$$
\end{prop}

\Pf . It is enough to check that the set $\Si_\la$ is the union of ${2k+1\choose k}$ cosets of
the subgroup $(\Z/2)^{\II^+_\la}\sub (\Z/2)^{\II_\la}$.
The set $\II^-_\la$ is partitioned into two subsets
$\II^-_\la(\pm 1)$ depending on the value $\ell(m)=\pm 1$. Since $|R_\la|=|R'_\la|+1$, 
we have $|\II^-_\la(1)|=k+1$ and $|\II^-_\la(-1)|=k$.
The set $\Si_\la$ is the union of $(\Z/2)^{\II^+_\la}$-cosets corresponding
to arbitrary choices of pairs of subsets $P(1)\sub \II^-_\la(1)$, $P(-1)\sub \II^-_\la(-1)$, such that
$|P(1)|=|P(-1)|$. Associating with the pair $(P(1), P(-1))$ the set
$P(1)\cup (\II^-_\la(-1)\setminus P(-1))$ we get a bijection with the set of all subsets of $\II^-_\la$
of cardinality $k$.
\ed

For a partition $\la\in \Pi^+_{2n+1}$, let $\ft_{\bG_\la}\sub \ft$ be the Lie algebra of the maximal torus
of $R\bG_\la$ embedded into the Lie algebra $\ft$ of the maximal torus of $\bG$ in a standard way
(where $R\bG_\la$ is given by \eqref{B-stabilizer-eq}).
On the other hand, for every partition $\mu$ with $|\mu|\le n$ let us consider the subspace
\begin{equation}\label{t-mu-eq}
\ft_\mu:=\{t\in\ft\ |\ t_1=\ldots=t_{\mu_1}, \ t_{\mu_1+1}=\ldots=t_{\mu_1+\mu_2},\ \ldots, \ t_{|\mu|+1}=\ldots= t_n=0\},
\end{equation}
where we use the standard
identification of $\ft$ with the $n$-dimensional space. Let also $W_{B,\mu}$ be the reflection group of
type $B_{k_1}\times\ldots\times B_{k_s}$, where $k_1,\ldots,k_s$ are multiplicities of parts in $\mu$, acting on the space $\ft_\mu$ by permutations of groups of coordinates corresponding
to equal parts in $\mu$ and by changing signs (simlutaneously in each group corresponding to a part of $\mu$).

For $\la\in \Pi^+_{2n+1}$ let $\la^{\red}$ denote the partition in which each $i$ appears with multiplicity 
$\lfloor r_i/2 \rfloor$, where $r_i$ is the multiplicity with which $i$ appears in $\la$.
Then using \eqref{B-stabilizer-eq} one can easily check that 
\begin{equation}\label{Lie-tori-eq}
\ft_{\bG_\la}=\ft_{\la^{\red}}.
\end{equation}

\begin{thm}\label{type-B-semiorth-dec-thm} (i) Let $A_W$ be the algebra \eqref{AW-eq} defined for the Weyl
group of type $B_n$.
For each $\la\in \Pi^+_{2n+1}$ 
let $\Cscr_\la\sub D(A_W-\mod)$ be the thick subcategory generated by modules $M_{\la,\xi}$ where $\xi$
varies over $\Spr_\la$. Then we have a semiorthogonal decomposition
$$D^f(A_W)=\lan \Cscr_{\la_N},\ldots,\Cscr_{\la_1}\ran,$$
where we order the nilpotent orbits $O_{\la_1},\ldots,O_{\la_N}$ in such a way that $O_{\la_i}$ is in
the closure of $O_{\la_j}$ only if $i<j$.

\noindent
(ii) Fix $\la\in\Pi^+_{2n+1}$, and let $|\II^-_\la|=2k+1$. We have a decomposition of $\Cscr_\la$
into ${2k+1\choose k}$ mutually orthogonal subcategories $\Cscr_\la(C)$, where $C$ runs over the
$\widehat{\pi}_\la(\widehat{F'_\la})$-cosets comprising $\Spr_\la$, and
$\Cscr_\la(C)$ is generated by all the modules $M_{\la,\xi}$ with $\xi\in C$.
Furthermore, for each such coset $C=\xi_0+\widehat{\pi}_\la(\widehat{F'_\la})$ we have an equivalence
\begin{equation}\label{C-la-C-eq}
\Cscr_\la(C)\simeq D^f(F'_\la\ltimes B_\la),
\end{equation}
where $B_\la$ is given by \eqref{B-la-orth-eq},
sending the module $M_{\la,\xi}$, with $\xi=\xi_0+\widehat{\pi}_\la(\eta)$,
to the projective module $\eta\otimes B_\la$. 

\noindent
(iii) For each $\la$ and $C=\xi_0+\widehat{\pi}_\la(\widehat{F'_\la})$ as in (ii) 
there exists a semiorthogonal decomposition of $\Cscr_\la(C)$
into subcategories generated by certain $A_W$-modules $M'_{\la,\xi}$, where $\xi\in C$ are arranged
in some order, such that $\Ext^{>0}_{A_W}(M'_{\la,\xi},M'_{\la,\xi})=0$, and
$$\End_{A_W}(M'_{\la,\xi_0+\widehat{\pi}_\la(\eta)})\simeq \C[\ft_\mu]^{W_{B,\mu}},$$
where the partition $\mu$ is defined as follows.
Let $m_1,\ldots,m_s\in\II^+_{\la}$ be all the elements 
such that the restriction of $\eta$ to the corresponding factor in $F'_\la$ is nontrivial.
Then $\mu$ is obtained from $\la^{\red}$ by reducing by $1$ the multiplicities of
each of the parts $m_1,\ldots,m_s$.
Furthermore, the support of the module $M'_{\la,\xi_0+\widehat{\pi}_\la(\eta)}$ is equal to
$W\cdot \ft_\mu$.
%\AA^l\times \AA^N
%where $l={|\II^+_\la|}$, $N=2k+1+\sum_{m \text{ even}}r_m$. Here $F'_\la=(\Z/2)^l$ acts naturally on 
%$\AA^l$ by changing signs of coordinates. 
\end{thm}

\Pf . (i) This follows from Theorem \ref{Springer-thm}.
%Indeed, for $\la\in\Pi^+_{2n+1}$
%we need to check that the algebra of endomorphisms of the sum of projective modules of
%$F_\la\ltimes B_\la$ corresponding to characters in $\Spr_\la$, has finite global dimension.
%By Proposition \ref{Bn-Spr-cosets}, $\Spr_\la$ is the union of ${2k+1\choose k}$
%$\widehat{\pi}_\la(\widehat{F'_\la})$-cosets.
%We claim that if $\xi$ and $\xi'$ belong to different cosets then
%the projective modules $\xi\otimes B_\la$ and $\xi'\otimes B_\la$ 
%are mutually orthogonal in $D(F_\la\ltimes B_\la)$. Indeed,
%this immediately follows from the fact that $\xi$ and $\xi'$ have different restrictions to the subgroup
%$\ker(\pi_\la)\sub F_\la$ that acts trivially on $B_\la$. 

\noindent
(ii) By Proposition \ref{Ext-calculation-prop}(ii), there is fully faithful functor
$$\Cscr_\la\to D(F_\la\ltimes B_\la),$$
sending $M_{\la,\xi}$ to the projective module $\xi\ot B_\la$.
Recall that for $\la\in\Pi^+_{2n+1}$ the set $\Spr_\la$ is the union of ${2k+1\choose k}$
$\widehat{\pi}_\la(\widehat{F'_\la})$-cosets (see Proposition \ref{Bn-Spr-cosets}).
If $C$ and $C'$ are different $\widehat{\pi}_\la(\widehat{F'_\la})$-cosets then the projective $F_\la\ltimes B_\la$-modules
$\xi\otimes B_a$ and $\xi'\otimes B_a$ are mutually orthogonal for any $\xi\in C$, $\xi'\in C'$
(see the proof of Lemma \ref{E-a-lem}). This implies that the
subcategories $\Cscr_\la(C)$ and $\Cscr_\la(C')$ are orthogonal for $C\neq C'$
Furthermore, for a coset $C=\xi_0+\widehat{\pi}_\la(\widehat{F'_\la})$ we have an equivalence of $\Cscr_\la(C)$
with the subcategory of $D(F_\la\ltimes B_\la)$ generated by $\xi\ot B_\la$ with $\xi\in C$.
Hence using the isomorphism \eqref{typeB-coset-End} with $F=F'_\la$
(which involves tensoring with $\xi_0$) we get the required equivalence
$$\Cscr_\la(C)\simeq D^b_{f}(F'_\la\ltimes B_\la).$$

\noindent
(iii) Recall that by Lemma \ref{B-component-group-action-lem}, the action of $F'_\la$ on
$B_\la$ is of type $A_1^l$ (on some of the coordinates). Thus, the required semiorthogonal
decomposition arises from the equivalence \eqref{C-la-C-eq} and the semiorthogonal decomposition for actions of type $A_1^l$ (see \S\ref{A1-sec}). 
Furthermore, we have
$$\End_{A_W}(M'_{\la,\xi_0+\widehat{\pi}_\la(\eta)})\simeq
\left(B_\la/(g_1,\ldots,g_s)\right)^{F'_\la},$$
where $g_i=f_{m_i/2}\in B_{\la,m_i}$ (see Lemma \ref{B-component-group-action-lem}),
It remains to observe 
that for each factor $\SO(r)$ in $\bG^0_\la$
with even $r$ the action of the Weyl group of $\SO(r)$ on its maximal ($r/2$-dimensional) torus
together with the action of the corresponding factor in $F'_\la$ generate the action which is equivalent to
the action of $W_{B_{r/2}}$. Also, we note that the embedding of the hyperplane given by vanishing
of one coordinate induces an isomorphism of $B_{\la,m_i}/(g_i)$ with the algebra of invariant polynomials
of type $B_{r-1}$, where $r$ is the multiplicity of $m_i$ in $\la$.

The assertion about the support follows from Proposition 
\ref{modules-Springer-prop}(iv) and from \eqref{Lie-tori-eq}.
\ed

\begin{rem} The semiorthogonal decomposition of $D^f(A_W)$ obtained from the above theorem
is very different from the one obtained in \S\ref{global-B-sec}. For example, the modules $M'_{\la,\xi}$
supported at the origin appear for all $\la$ which has only odd parts appearing with multiplicities $1$ or $2$,
whereas in the decomposition of \S\ref{global-B-sec} all such modules are consecutive.
It is likely that one decomposition is obtained from another by a sequence of mutations.
\end{rem}

By Propositions \ref{HH-semiorth-match-prop} and \ref{star-reflection-prop}, the pieces
of the semiorthogonal decomposition of the above theorem (where we replace each $\Cscr_\la$ with
its natural semiorthogonal decomposition) match with the pieces in the corresponding motivic decomposition
\eqref{mot-dec2-eq}. We are going to present an independent proof of this using the explicit description
of the Springer correspondence in this case. It turns out that this matching reflects in a certain combinatorial identity related
to the Jacobi triple identity (see Proposition \ref{Bn-match-prop} and Corollary \ref{Bn-id-cor} below).

Recall that in the motivic decomposition we look at subspaces in $\ft$ of the form
$\ft^w$, where $w\in W$ runs through representatives
of conjugacy classes in $W$, and at the action of the centralizer $C(w)$ on $\ft^w$. 
%We will refer to these subvarieties as {\it strata}.
The conjugacy classes in $W$ are given by pairs of partitions $(\mu,\mu')$ such that
$|\mu|+|\mu'|=n$, where $\mu$ (resp., $\mu'$) records positive (resp., negative) cycle type (see \cite{Carter-conj}).
The corresponding subspace $\ft^w$ is precisely $\ft_\mu$ (see \eqref{t-mu-eq}).
The centralizer $C(w)$ of $w$ acts on $\ft^w=\ft_\mu$ through the quotient
$C(w)\to W_{B,\la}$ which leads to the $B_{k_1}\times\ldots\times B_{k_s}$-type action on $\ft_\mu$.
Thus, 
the corresponding piece in the motivic decomposition of $[\ft/W]$ is 
\begin{equation}\label{mu-motivic-piece}
\ft_\mu/W_{B,\mu}=\prod_i \A^{k_i}/W_{B_{k_i}},
\end{equation}
where $k_i$ is the multiplicity of $i$ in $\mu$. Note that this motive is determined by the single partition $\mu$.
By Theorem \ref{type-B-semiorth-dec-thm}(iii),
the motive associated with each module $M'_{\la,\xi}$ is of the same type for a certain partition $\mu$
associated with $(\la,\xi)$. Thus, the fact that the motivic pieces from the semiorthogonal
decomposition of Theorem \ref{type-B-semiorth-dec-thm} match the parts of the decomposition \eqref{mot-dec2-eq},
amounts to the equality between the number of times that each partition $\mu$ with $|\mu|\le n$ arises from pairs $(\la,\xi)$
occuring in the Springer correspondence and the number of partitions of $n-|\mu|$.

Let $u_1,u_2,\ldots$ be independent variables. We associate with every partition $\mu$ the monomial 
\begin{equation}\label{u-mu-eq}
u(\mu)=u_1^{k_1}u_2^{k_2} \ldots.
\end{equation}
where $k_i$ is the multiplicity of $i$ in $\mu$. 
The formal sum (with multiplicities) of the monomials $u(\mu)$ 
associated with all conjugacy classes $(\mu,\mu')$ in $W$ appears as the coefficient of $x^n$ in
the generating function
$$F^W(x,u_1,u_2,\ldots)=\prod_{n\ge 1}(1-u_nx^n)^{-1}(1-x^n)^{-1}=\sum_{\mu,\mu'} u(\mu)x^{|\mu|+|\mu'|}.$$

\begin{lem}\label{Bn-comb-lem}
For $\la\in \Pi^+_{2n+1}$ consider the polynomial
$$P_\la(u_1,u_2\ldots):=\sum_{\xi\in\Spr_\la} u(\la,\xi),$$
where $u(\la,\xi)=u(\mu)$ is the monomial associated with the partition $\mu$ obtained from $(\la,\xi)$
as in Theorem \ref{type-B-semiorth-dec-thm}(iii).
Then
$P_\la$ is equal to the coefficient of $x^{2n+1}$ in
the series $S(F^\Nscr)$, where 
$$F^{\Nscr}(x,u_1,u_2,\ldots)=\prod_{m\ge 1}(1-u_mx^{2m})^{-1}\cdot \prod_{m \text{\rm \ odd}}(1+tx^m+x^{2m}),$$  
where $t$ is one more independent variable, and
$S$ is the $\Z[u_1,u_2,\ldots]$-linear operator sending $t^{2k+1}$ to 
${2k+1\choose k}$, $x^{2k}$ to $0$ and $x^{2k+1}$ to itself.
\end{lem}

\Pf . 
Let $\Pi^+=\cup_n \Pi^+_{2n+1}$ be the set of all partitions $\la$ with odd $|\la|$, in which even parts appear with even multiplicity.
We want to prove the identity
\begin{equation}\label{P-la-sum-eq}
\sum_{\la\in \Pi^+}P_\la(u_1,u_2,\ldots)x^{|\la|}=S(F^{\Nscr}).
\end{equation}

The description of $\mu$ in Theorem \ref{type-B-semiorth-dec-thm}(iii) easily leads to the equality
$$P_\la(u_1,u_2,\ldots)=
S\left(\prod_{i \text{\ even}}u_i^{r_i/2}\cdot \prod_{i \text{\ odd}, r_i \text{\ odd}}(tu_i^{(r_i-1)/2})\cdot
\prod_{i \text{\ odd}, r_i \text{\ even}} (u_i^{r_i/2}+u_i^{r_i/2-1})\right).$$
Note that the presence of $t$ and the use of the operator $S$ correspond to 
the multiplicity ${2k+1\choose k}$ in Theorem \ref{type-B-semiorth-dec-thm}(ii). 
Thus, the left-hand side of \eqref{P-la-sum-eq} is equal to $S(F)$, where 
\begin{align*}
&F=\prod_{i \text{\ even}}(1+u_ix^{2i}+u_i^2x^{4i}+u_i^3x^{6i}+\ldots)\cdot\\
&\prod_{i \text{\ odd}}\left(1+t[x^i+u_ix^{3i}+u_i^2x^{5i}+\ldots]+(1+u_i)[x^{2i}+u_ix^{4i}+u_i^2x^{6i}+\ldots]
\right).
\end{align*}
Simplifying, one immediately gets that $F^\Nscr=F$.
\ed

\begin{prop}\label{Bn-match-prop} 
One has
$$S(F^\Nscr)(x,u_1,u_2,\ldots)=xF^W(x^2,u_1,u_2,\ldots).$$
Hence, the multiplicity with which each motivic piece \eqref{mu-motivic-piece}
arises from a module in the semiorthogonal
decomposition of Theorem \ref{type-B-semiorth-dec-thm} is equal
to the multiplicity with which it appears in the decomposition \eqref{mot-dec2-eq}. 
\end{prop}

\Pf . We know that the stated equality of the generating functions holds after the substitution
$u_1=u_2=\ldots=1$. Indeed, this follows from the fact that the cardinality of the set of
pairs $(O,\xi)$ appearing in the Springer correspondence is the same as the number of conjugacy classes in
$W$. Set 
$$F_0^{\Nscr}(x,u_1,u_2,\ldots)=\prod_{n \text{ odd}}(1+tx^n+x^{2n})=F^{\Nscr}|_{u_1=u_2=\ldots=0}.$$
Note that
\begin{equation}\label{S-F-F0-eq}
S(F^\Nscr)=\prod_{n\ge 1}(1-u_nx^{2n})^{-1}\cdot S(F_0^{\Nscr}),
\end{equation}
since the first factor does not depend on $t$ and contains only even powers of $x$.
Hence, 
$$S(F^\Nscr)|_{u_1=u_2=\ldots=1}=\prod_{n\ge 1}(1-x^{2n})^{-1}\cdot S(F_0^{\Nscr}).$$
Thus, the identity 
$$S(\Fscr^{\Nscr})|_{u_1=u_2=\ldots=1}=xF^W(x^2,1,1,\ldots)$$
becomes
$$\prod_{n\ge 1}(1-x^{2n})^{-1}\cdot S(F_0^{\Nscr})=x\cdot \prod_{n\ge 1}(1-x^{2n})^{-2}.$$
Therefore, we have
$$S(F_0^{\Nscr})=x\cdot \prod_{n\ge 1}(1-x^{2n})^{-1}.$$
Plugging this into \eqref{S-F-F0-eq} we get the result.
\ed

\begin{cor}\label{Bn-id-cor} One has
\begin{equation}\label{Bn-comb-id}
S\left(\prod_{n \text{ odd}}(1+tx^n+x^{2n})\right)=x\cdot \prod_{n\ge 1}(1-x^{2n})^{-1}.
\end{equation}
\end{cor}

The following independent proof of \eqref{Bn-comb-id} was found by Ben Young.
Start with the Jacobi triple identity
$$
\prod_{m\ge 1}(1-x^{2m})(1+x^{2m-1}y^2)(1+x^{2m-1}y^{-2})=\sum_{n\in\Z}x^{n^2}y^{2n}.
$$
Now set $y=\sqrt{q}$ and divide both sides by $\prod_m (1-x^{2m})$:
$$\prod_{m\ge 1}(1+x^{2m-1}q)(1+x^{2m-1}q^{-1})=\sum_{n\in\Z}x^{n^2}q^{n}\prod_{m\ge 1}(1-x^{2m})^{-1}.$$
Thus, we get
\begin{equation}\label{Jacobi-triple-id}
\prod_{m \text{ odd}}(1+(q+q^{-1})x^m+x^{2m})=\sum_{n\in\Z}x^{n^2}q^{n}\prod_{m\ge 1}(1-x^{2m})^{-1}
\end{equation}
Now we observe that the left-hand side of \eqref{Bn-comb-id}
is exactly the coefficient of $q^1$ in the $q$-expansion of the left-hand side of \eqref{Jacobi-triple-id}.

\begin{rem} The identity of Corollary \ref{Bn-id-cor} has the following
combinatorial interpretation.
%which gives hope for a bijective proof. 
Consider partitions 
of $2n+1$ with the following restrictions: all parts are odd, multiplicity of each part is $\le 2$.
If there are $2k+1$ parts of multiplicity $1$ then this partition should be
counted with weight ${2k+1\choose k}$. Then the sum of weights of all such partitions is $p(n)$,
the usual partition number. 
\end{rem}

\subsection{Type $D_n$}\label{type-D-sec}
Now let us consider the Springer correspondence for the group $\bG=\SO(2n)$. The picture is going to
be very similar to the one for $B_n$ but with some important differences, which lead to the presence
of ``noncommutative" pieces in the corresponding semiorthogonal decomposition.

Representations of $W_{D_n}$ are parametrized by the set $\Pi^{(2)}_n/S_2$ of unordered
pairs of partitions $(\xi,\nu)$ such that $|\xi|+|\nu|=n$, except that
there are two irreducible representations associated with the pair
when $\xi=\nu$. Similarly to the case of type $B_n$ we associate with $(\xi,\nu)$ its symbol, which is an
{\it unordered} pair of sets $(R,R')$, setting
$$R=\{\xi_0,\xi_1+2,\ldots,\xi_r+2r\},
R'=\{\eta_0,\eta_1+2,\ldots,\eta_r+2r\},$$
where $\xi=(\xi_0\le\xi_1\le\ldots\le\xi_r)$, $\eta=(\eta_0\le\eta_1\le\ldots\le\eta_r)$
(if necessary we extend one of the partitions by $0$s).

The nilpotent orbits in $\SO(2n)$ are parametrized by the set
$\Pi^+_{2n}$ of partitions $\la$ of $2n$ such that each even part appears
with even multiplicity, except that there are two nilpotent orbits associated with $\la$ in the case
when all parts of $\la$ are even.
The reductive part of the centralizer $\bG_\la$ of an element of the orbit $O_\la$ 
and its component group $F_\la$ are still given by \eqref{B-stabilizer-eq} and \eqref{B-component-group-eq},
where $\II_\la$ is still the set of odd $m$ appearing as parts in $\la$.  

Similarly to type $B_n$, there is a map
$$\Spr_1:\Pi^+_{2n}\to \Pi_n^{(2)}/S_2$$
defined as follows. For $\la=(\la_0\le\la_1\le\ldots)$ define the numbers
$\xi_1\le\ldots\le\xi_r$ and $\eta_1\le\ldots\le\eta_r$ in the same way as in the case of type $B_n$.
It is easy to see that $\la\in\Pi^+_{2n}$ has all parts even if and only if $\Spr_1(\la)$ has $\xi=\eta$.
In this case there is some rule which of the two nilpotent orbits associated with $\la$ corresponds
to which of the two irreducible representations associated with $\Spr_1(\la)$, but the precise form
of the correspondence is not important for us.

Let $(R_\la,R'_\la)$ be the symbol associated with $\Spr_1(\la)$. Assume that $R_\la\neq R'_\la$,
i.e., $\la$ has at least one odd part. As in the case of type $B_n$
the maximal intervals of integers in $(R_\la\cup R'_\la)\setminus (R_\la\cap R'_\la)$ are in
bijection with odd numbers $i$ appearing as parts of $\la$, and the length of the interval is equal
to the multiplicity $r_i$.
We define similarity classes of symbols of pairs of partitions in $\Pi^{(2)}_n/S_2$ as before, so that each similarity class contains the unique symbol in the image of $\Spr_1$. 

In the case when $\II_\la=\emptyset$, i.e., all parts of $\la$ are even, the component group $F_\la$ is trivial,
and the Springer correspondence matches two nilpotent orbits corresponding to $\la$ with two representations
of $W_{D_n}$ associated with $\Spr_1(\la)$ (the precise rule is not important for us).

Now assume that $\II_\la\neq\emptyset$, and let $\Si_\la$ have
the same meaning as in the case of type $B_n$. Note that we have $\sum_{m\in\II_\la}\ell(m)=0$,
so $\Si_\la$ is a coset for the subgroup $\lan\be\ran\sub (\Z/2)^{\II_\la}$ generated by the element
$\be=(1,\ldots,1)\in (\Z/2)^{\II_\la}$.
Then we have a bijection
$$\Si_\la/\lan\be\ran \rTo{\sim} \Sscr_\la$$
to the set $\Sscr_\la$ of symbols $(R,R')$ similar to $(R_\la,R'_\la)$, defined as in the case of type $B_n$
(the action of $\be$ has the effect of swapping $R$ and $R'$, so it acts trivially on $\Sscr_\la$
since we view $(R,R')$ as unordered pairs).
We still can identify the component group $F_\la$ with $F(\II_\la)$, and 
view elements of $(c_m)\in \Si_\la/\lan\be\ran$ as characters of $F_\la$
(by definition, $\be$ restricts to the trivial character of $F_\la$). 
This gives a bijection
$\Si_\la/\lan \be\ran\rTo{\sim} \Spr_\la$ that computes the Springer correspondence 
for the nilpotent orbit associated with $\la$. 

Since $\la$ is a partition of $2n$, we have $|\II_\la^-|=2k$ for some $k$. Assume first that $k>0$.
Then the natural projection
$$\pi_\la:F_\la\to F'_\la:=(\Z/2)^{\II_\la^+}.$$
is surjective. We claim that in this case similarly to Proposition \ref{Bn-Spr-cosets}
the set of characters of $F_\la$ that appear in the Springer correspondence
is the union of ${2k \choose k}/2$ cosets of the subgroup
$\widehat{\pi_\la}(\widehat{F'_\la})\sub\widehat{F_\la}$.
Indeed, it is enough to check that $\Si_\la$ is the union of ${2k\choose k}$ cosets of
the subgroup $(\Z/2)^{\II_\la^+}$. Using the notation from Proposition \ref{Bn-Spr-cosets}
we have $|\II_\la^-(1)|=|\II_\la^-(-1)|=k$. The number in question is obtained as the number
of choices of pairs of subsets $P(1)\sub \II_\la^-(1)$, $P(-1)\sub \II_\la^-(-1)$ with $|P(1)|=|P(-1)|$,
which is equal to ${2k\choose k}$. 

Next, assume that $k=0$. Then $\Si_\la=(\Z/2)^{\II_\la}$ and the restriction map 
$$\Si_\la/\lan\be\ran\to \widehat{F_\la}$$ 
is an isomorphism, so in this case all characters of the component group appear in the Springer correspondence.
Note also that in this case the map $\pi_\la:F_\la\to F'_{\la}$ is the natural embedding
of the index $2$ subgroup $F(\II_\la)\sub (\Z/2)^{\II_\la}$.

Now we can formulate the analog of Theorem \ref{type-B-semiorth-dec-thm} for type $D_n$.
As before, for each partition $\mu$ with $|\mu|\le n$ we consider the subspace $\ft_{\mu}\sub \ft$ 
given by \eqref{t-mu-eq}. Let $W_{D,\mu}$ be the group acting on $\ft_\mu$ by permutations of
groups of coordinates corresponding to equal parts in $\mu$, and by changing signs
of an even number of coordinates $t_1,\ldots,t_n$ (simultaneously in each group of coordinates corresponding to a part in
$\mu$). Note that if $|\mu|<n$ then $W_{D,\mu}=W_{B,\mu}$,
since in this case $t_n=0$ and hence arbitrary changes of signs are allowed.

Below we will use for $\la\in\Pi^+_{2n}$
the notations $B_\la$, $\la^{\red}$, etc., introduced in \S\ref{Springer-B-n-sec}
for $\la\in\Pi^+_{2n+1}$. We also denote by $\Nilp_{2n}$ the set of nilpotent orbits, so that we have a surjective
map $\Nilp_{2n}\to \Pi^+_{2n}$ with fibers of cardinality $2$ over $\la$ that have only even parts.

\begin{thm}\label{type-D-semiorth-dec-thm} (i) Let $A_W$ be the algebra \eqref{AW-eq} defined for the Weyl
group of type $D_n$.
For each $\wt{\la}\in \Nilp_{2n}$ over $\la\in\Pi^+_{2n}$
let $\Cscr_{\wt{\la}}\sub D(A_W-\mod)$ be the thick subcategory generated by modules $M_{\la,\xi}$, where $\xi$
varies over $\Spr_\la$. Then we have a semiorthogonal decomposition
$$D^f(A_W)=\lan \Cscr_{\wt{\la}_N},\ldots,\Cscr_{\wt{\la}_1}\ran,$$
where we order the nilpotent orbits $O_{\wt{\la}_1},\ldots,O_{\wt{\la}_N}$ in such a way that $O_{\wt{\la}_i}$ is in
the closure of $O_{\wt{\la}_j}$ only if $i<j$.
The category $\Cscr_{\wt{\la}}$ depends only on $\la\in\Pi^+_{2n}$, up to equivalence, so we denote it
by $\Cscr_\la$.

\noindent
(ii) Fix $\la\in\Pi^+_{2n}$, and let $|\II_\la^-|=2k$. 
Assume first that $k>0$. Then we have a decomposition of $\Cscr_\la$
into ${2k\choose k}/2$ mutually orthogonal subcategories $\Cscr_\la(C)$, where $C$ runs over the
$\widehat{\pi}_\la(\widehat{F'_\la})$-cosets comprising $\Spr_\la$, and
$\Cscr_\la(C)$ is generated by all the modules $M_{\la,\xi}$ with $\xi\in C$.
Furthermore, for each such coset $C=\xi_0+\widehat{\pi}_\la(\widehat{F'_\la})$ we have an equivalence
%\begin{equation}\label{C-la-C-eq}
$$\Cscr_\la(C)\simeq D^f(F'_{\la}\ltimes B_\la).$$
%\end{equation}
sending the module $M_{\la,\xi}$, with $\xi=\xi_0+\widehat{\pi}_\la(\eta)$
to the projective module $\eta\otimes B_\la$. 

\noindent
(ii)' Now let $\la\in\Pi^+_{2n}$ be such that $\II_\la^-=\emptyset$. Then we have
$$\Cscr_\la\simeq D^b_{F(\II_\la)}(B_\la)\simeq D^b([\ft_\mu/W_{D,\mu}]),$$
where $\mu=\la^{\red}$ (note that in this case all multiplicities of parts in $\la$ are even, so $|\mu|=n$).
%sending the modules $(M_{\la,\xi})$ to all indecomposble projective modules over $F(\II_\la)\ltimes B_\la$.

\noindent
(iii) For each $\la$ with $|\II_\la^-|>0$ and $C=\xi_0+\widehat{\pi}_\la(\widehat{F'_\la})$ as in (ii), 
there exists a semiorthogonal decomposition of $\Cscr_\la(C)$
into subcategories generated by certain $A_W$-modules $M'_{\la,\xi}$, where $\xi\in C$ are arranged
in some order, such that $\Ext^{>0}_{A_W}(M'_{\la,\xi},M'_{\la,\xi})=0$, and
$$\End_{A_W}(M'_{\la,\xi_0+\widehat{\pi}_\la(\eta)})\simeq \C[\ft_\mu]^{W_{B,\mu}},$$
where the partition $\mu$ is defined as follows.
Let $m_1,\ldots,m_s\in\II^+_{\la}$ be all the elements 
such that the restriction of $\eta$ to the corresponding factor in $F'_\la$ is nontrivial.
Then $\mu$ is obtained from $\la^{\red}$ by reducing by $1$ the multiplicities of
each of the parts $m_1,\ldots,m_s$.
Furthermore, the support of the module $M'_{\la,\xi_0+\widehat{\pi}_\la(\eta)}$ is equal to
$S_n\cdot \ft_\mu$.
%\AA^l\times \AA^N
%where $l={|\II^+_\la|}$, $N=2k+1+\sum_{m \text{ even}}r_m$. Here $F'_\la=(\Z/2)^l$ acts naturally on 
%$\AA^l$ by changing signs of coordinates. 
\end{thm}

The proof is absolutely analogous to that of Theorem \ref{type-B-semiorth-dec-thm}, where
in the case of $\la\in\Pi^+_{2n}$ with $\II^-_\la=\emptyset$ we use the fact that all characters of $F_\la$
appear in the Springer correspondence to check the applicability of Theorem \ref{Springer-thm}.                                                                                                                                                                                                                                              
Note also that in (iii) we have $|\mu|<n$ so that $W_{D,\mu}=W_{B,\mu}$.

Thus, we obtain a semiorthogonal decomposition of $D^f(A_W)$ for type $D_n$ into two kinds of subcategories:
1) the subcategories generated by the modules $M'_{\la,\xi}$ where $\la$ is such that $|\II_\la^-|>0$;
2) the subcategories $\Cscr_\la$ for $\la$ such that $|\II_\la^-|=0$.
We will refer to pieces of type 2) as {\it noncommutative pieces} in our semiorthogonal decomposition.
%(since these subcategories are not equivalent to derived categories of any smooth variety).

Now we are going to match this decomposition with the motivic decomposition
\eqref{mot-dec2-eq}, where the subcategories of the second kind absorb several pieces of the motivic 
decomposition.
 
Recall that 
conjugacy classes in $W=W_{D_n}$ are parametrized by (ordered) pairs of partitions $(\mu,\mu')$ such that
$|\mu|+|\mu'|=n$ and such that $\mu'$ has even number of parts, except that in the case when $\mu'=\emptyset$
and all parts of $\mu$ are even there are two corresponding conjugacy classes.
As in the case of type $B_n$, the corresponding subspace $\ft^w$ is $\ft_\mu$.
The centralizer $C(w)$ acts on $\ft^w=\ft_\mu$ through the quotient $W_{D,\mu}$. 
In the case $|\mu|<n$ (i.e., $\mu'\neq\emptyset$) then $W_{D,\mu}=W_{B,\mu}$, so the quotient is smooth:
$$\ft_\mu/W_{D,\mu}\simeq \prod_i \A^{k_i}/W_{B_{k_i}},$$
where $k_i$ is the multiplicity of $i$ in $\la$.
If $\mu'=\emptyset$ and there is only one odd part $i_0$ (appearing with some multiplicity $k_{i_0}$) in $\mu$ then
$W_{D,\mu}$ acts by type $B$ action on each group of coordinates corresponding to equal even parts and 
by type $D$ action on coordinates corresponing to the part $i_0$. Thus, in this case the quotient is still smooth:
$$\ft_\mu/W_{D,\la}\simeq \prod_{i \text{ even}} \A^{k_i}/W_{B_{k_i}}\times \A^{k_{i_0}}/W_{D_{k_{i_0}}}.$$
Finally, if $\mu'=\emptyset$ and the odd parts appearing in $\mu$ are $i_1,\ldots,i_s$, where $s>1$,
appearing with multiplicities $k_{i_1},\ldots,k_{i_s}$, then $W_{D,\mu}$ acts by changing signs in even number
of coordinates corresponding to odd parts (in addition to permuting coordinates corresponding to equal parts).
In this case the quotient is singular, and we have an isomorphism of stacks
$$[\ft_\mu/W_{D,\mu}]\simeq \prod_{i \text{ even}}\A^{k_i}/W_{B_{k_i}}\times 
[\left(\prod_{m=1}^s (\A^{k_{i_m}}/W_{D_{k_{i_m}}})\right)/F_s],$$
where $F_s=F(\{1,\ldots,s\})\sub(\Z/2)^s$ acts on each space $\A^{k_{i_m}}/W_{D_{k_{i_m}}}$
via the projection to the $m$th factor $\Z/2$, using the outer involution of $\A^{k_{i_m}}/W_{D_{k_{i_m}}}$. 
Note that by choosing basic invariant polynomials as in Lemma \ref{B-component-group-action-lem},
we get an isomorphism
$$[\left(\prod_{m=1}^s (\A^{k_{i_m}}/W_{D_{k_{i_m}}})\right)/Z_s]\simeq \prod_{m=1}^s \A^{k_{i_m}-1}\times [\A^s/Z_s].$$

In any case the corresponding piece in the motivic decomposition is $\ft_\mu/W_{D,\mu}$, 
to which we associate as before
the monomial $u(\mu)$ (see \eqref{u-mu-eq}). In the case when $\mu'=\emptyset$ and
all parts of $\mu$ are even our convention
is that this monomial stands for the two pieces in the decomposition (corresponding to the two
conjugacy classes associated with $(\mu,\emptyset)$).
The formal sum of the obtained monomials over all conjugacy classes in $W$ is the coefficient of $x^n$
in 
$$F^W(x,u_1,u_2,\ldots)=\prod_{n\ge 1}(1-u_nx^n)^{-1}\cdot P^+,$$
where 
$$P^+(x)=\frac{1}{2}(\prod_{n\ge 1}(1-x^n)^{-1}+\prod_{n\ge 1}(1+x^n)^{-1})$$
is the generating function for partitions with even number of parts.

Now let us count contributions to pieces of the motivic decomposition from the semiorthogonal decomposition
of Theorem \ref{type-D-semiorth-dec-thm}.
%Recall that there is a correspondence between nilpotent orbits and partitions $\la\in\Pi^+_{2n}$, such that
%to each $\la$ that has only even parts there correspond two nilpotent orbits.
Let $\la$ be a partition in $\Pi^+_{2n}$ and let $r_i$ denote the multiplicity of $i$ in $\la$.
We distinguish two cases: 1) $\la$ has at least one odd part appearing with odd multiplicity;
2) all odd parts in $\la$ appear with even multiplicity. In case 1) the counting is similar to the case of
type $B_n$, so the coresponding contribution is
$$P_\la=
S'\left(\prod_{i \text{\ even}}u_i^{r_i/2}\cdot \prod_{i \text{\ odd}, r_i \text{\ odd}}(tu_i^{(r_i-1)/2})\cdot
\prod_{i \text{\ odd}, r_i \text{\ even}} (u_i^{r_i/2}+u_i^{r_i/2-1})\right),$$
where $S'$ is the $\Z[u_1,u_2,\ldots]$-linear operator sending $t^{2k+1}$ to $0$, $t^{2k}$ to
${2k \choose k}/2$ for $k>0$ and $1$ to $1$.
In case 2) we have to replace a noncommutative piece in the semiorthogonal
decomposition by the corresponding motivic pieces 
$$(1+D_2+D_4+\ldots)(\prod_i u_i^{r_i/2}).$$
Here for a monomial $M$ in $u_i$ we set
$$D_{p}(M)=\left(\sum_{2m_1-1<\ldots<2m_p-1}\frac{M}{u_{2m_1-1}\ldots u_{2m_p-1}}\right)_{reg},$$
where we throw away summands that are not polynomials.
%In this case we associate with $(\la,\mu)$ the noncommutative piece $[\ft^w/C(w)]$, which absorbs
%the strata $(1+D_2+D_4+\ldots)(u(\la))$, where 
Finally, in the case when all parts of $\la$ are even we have two corresponding nilpotent orbits
and the corresponding stratum is $u(\la^{\red})=\prod_i u_i^{r_i/2}$.
The strata appearing in this way are those corresponding to pairs $(\mu,\emptyset)$,
where all parts of $\mu$ are even, which are exactly the strata that have to be taken
twice in the motivic decomposition (since there are two associated conjugacy classes).

The resulting contribution from all $\la\in\Pi^+_{2n}$ is equal to the coefficient of $x^{2n}$ in
$$F^\Nscr(x,u_1,u_2,\ldots)=
\prod_{m \text{ even}}(1-u_mx^{2m})^{-1}\cdot [S'(F_1)+F_2],$$
where $S'(F_1)$ and $F_2$ correspond to cases 1) and 2) above:
$$F_2=\prod_{m \text{ odd}}(1-u_mx^{2m})^{-1}\cdot\frac{1}{2}(P^{\text{dist,odd}}(x^2)+P^{\text{dist,odd}}(-x^2)) \ \text{ with}$$
$$P^{\text{dist,odd}}(x)=\prod_{m \text{ odd}}(1+x^m);$$
$$F_1=\wt{F}_1(x,t)-\wt{F}_1(x,0), \ \text{  where }$$
$$\wt{F}_1=\prod_{m \text{ odd}}(1+tx^m[1-u_mx^{2m}]^{-1}+(1+u_m)x^{2m}[1-u_mx^{2m}]^{-1})=
\prod_{m \text{ odd}}(1-u_mx^{2m})(1+tx^m+x^{2m}).$$
Thus, we can rewrite
$$F_1=\prod_{m \text{ odd}}(1-u_mx^{2m})\cdot 
\left(\prod_{m \text{ odd}}(1+tx^m+x^{2m})-P^{\text{dist,odd}}(x^2)\right),$$
$$F^\Nscr=\prod_{n\ge 1}(1-u_nx^{2n})^{-1}\cdot\left[S'\left(\prod_{m \text{ odd}}(1+tx^m+x^{2m})\right)-
\frac{1}{2}P^{\text{dist,odd}}(x^2)+\frac{1}{2}P^{\text{dist,odd}}(-x^2)\right].$$

Now we claim that there is an equality
\begin{equation}\label{Dn-main-comb-id}
F^\Nscr(x,u_1,u_2,\ldots)=F^W(x^2,u_1,u_2,\ldots).
\end{equation}
The equality of the corresponding coefficients of $x^{2n}$ shows that we get the right number of
pieces in the motivic decomposition from our semiorthogonal decomposition.
Both sides in the equality have the factor $\prod_n (1-u_nx^{2n})^{-1}$, and the remaining parts do not
depend on $u_i$. Hence, \eqref{Dn-main-comb-id} is equivalent to the identity obtained from it
by setting $u_1=u_2=\ldots=1$, which holds due to the Springer correspondence.
The obtained combinatorial identity is
\begin{equation}\label{Dn-comb-id}
S'\left(\prod_{m \text{ odd}}(1+tx^m+x^{2m})\right)-
\frac{1}{2}P^{\text{dist,odd}}(x^2)+\frac{1}{2}P^{\text{dist,odd}}(-x^2)=P^+(x^2).
\end{equation}
Using the well known identity 
$$\prod_{n\ge 1}(1+x^n)=\prod_{m \text{ odd}}(1-x^m)^{-1}$$   
we obtain
$$P^{\text{dist,odd}}(-x)=\prod_{m \text{ odd}}(1-x^m)=\prod_{n\ge 1}(1+x^n)^{-1}.$$
Thus,
$$P^+(x)-\frac{1}{2}P^{\text{dist,odd}}(-x)=\frac{1}{2}\prod_{n\ge 1}(1-x^{2n})^{-1},$$
and the idenity \eqref{Dn-comb-id} is equivalent to
$$S'\left(\prod_{m \text{ odd}}(1+tx^m+x^{2m})\right)-
\frac{1}{2}P^{\text{dist,odd}}(x^2)=\frac{1}{2}\prod_{n\ge 1}(1-x^{2n})^{-1}.$$
This in turn is equivalent to
$$\left[\prod_{m \text{ odd}}(1+(q+q^{-1})x^m+x^{2m})\right]_{q^0}=\prod_{n\ge 1}(1-x^{2n})^{-1},$$
which follows from \eqref{Jacobi-triple-id} by taking the constant coefficient of the expansion in powers of $q$.

%\section{The case of complex reflection groups of rank $2$}

%First, construct a semiorthogonal decomposition.
%Next, use Propositions \ref{HH-semiorth-match-prop} and \ref{star-reflection-prop}.

\appendix

%\input{support_dg_frobenius}

%\begin{document}

\section{Appendix I: DG-models and Frobenius}\label{dg-sec}

%\subsection{Introduction and statement of the main result}
\subsection{Lifting the Frobenius at the DG-level}
%If $X$ is a scheme of finite type over $\CC$ then we denote by
%$X(\CC)$ the corresponding topological space. 
%\emph{Throughout a ``constructible'' sheaf on $X(\CC)$
%is a sheaf which is constructible with respect 
%to an algebraic stratification}.

In this section we will prove the following ``abstract nonsense'' result.

\begin{thm}
  \label{anr} Let $X$ be a scheme of finite type over $\CC$, and let
  $\Fscr$ be a bounded complex of $\CC$-vector spaces with
  constructible cohomology on $X(\CC)$ which ``can be defined over a
  finite field'' (see below).  Then $R\Hom_{\CC}(\Fscr,\Fscr)$ may be
  represented by a DG-algebra $A$ which is in addition equipped with a
  DG-endomorphism $F$ such that the action of $F$ on $H^\ast(A)$
  coincides with the action of the Frobenius endomorphism on
  $\Ext^\ast_\Lambda(\Fscr_s,\Fscr_s)\otimes_\Lambda \CC$ where
  $\Fscr_s$ which is a reduction of $\Fscr$ to the algebraic closure
  of a finite field and $\Lambda\subset \CC$ is a suitable coefficient ring which is a complete
  discrete valuation ring with finite residue field.

The same result holds if $\Fscr$ lies in the bounded Bernstein-Lunts equivariant 
derived category of $[X(\CC)/G(\CC)]$ for a linear algebraic group $G$ acting
on $X$.
\end{thm}

We will say that $\Fscr\in D_c^b(X(\CC),\CC)$ or $\Fscr\in
D_c^b([X(\CC)/G(\CC)],\CC)$ is ``defined over a finite field'' if it can
be obtained from an $\Fscr_s$ defined over a finite field using the
procedure ``De $\FF$ \`a $\CC$'' outlined in \cite[\S6]{BBD} (this
procedure will be reviewed below). Note that while in loc.\ cit. it is
shown that any $\Fscr$ is defined over the \emph{algebraic closure of
  a finite field}, being actually defined over a finite field itself
is a very subtle property\footnote{A simple perverse sheaf of
  ``geometric'' origin is defined over a finite field by
  \cite[\S6.2.4]{BBD}.}. In particular, this property is not stable
under extension and so one must be careful with categorical
constructions.  In practice one usually completely sidesteps this
problem by first defining the relevant objects over a finite field and
then performing the lift to~$\CC$.

\medskip

The main point of Theorem \ref{anr} is not the existence of the
DG-algebra $A$, which in the topological case is obvious since we are
working in a genuine derived category, but the fact that $A$ may be
chosen in such a way as to be equipped with a lift of the Frobenius
endomorphism, as the latter only exists over a finite field.

\medskip

To prove Theorem \ref{anr} we have to check that the procedure of
reducing to finite characteristic as explained in \cite[\S6]{BBD} (for
schemes) can be lifted to the DG-level (i.e., it can be ``enhanced'').
We have to deal with the fact, however, that the
$l$-adic derived category constructed in \cite{BBD,deligneweil2} is not actually a
subcategory of a genuine derived category, and hence it is not
naturally enhanced. Luckily this defect has been repaired in recent
years \cite{Behrend1,Ek,LO2}.

\medskip

Another technical problem, which arises only in the $G$-equivariant case, is
the following: 
a morphism of
algebraic stacks $f:\Xscr\r \Yscr$ does not define a morphism between
the corresponding lisse-\'etale topoi, as the pullback functor does not
commute with finite limits (this was independently observed by Behrend
and Gabber, see \cite[Example 3.4]{Ollson} for an explicit counterexample 
between schemes!). This does not present a problem on the derived
level since a functor $Lf^{-1}: D_c(\Yscr)\r D_c(\Xscr)$  has been
constructed \cite{Ollson}. 
However it does present a problem on the DG-level.  The functors
connecting $\CC$ and $\FF$ are mostly given by inverse images, and
since the DG-enhancements we use are based on injective resolutions, a
non-exact inverse pullback functor is very inconvenient.

Luckily these problems disappear if we represent stacks by simplicial
schemes equipped with the \'etale site. In that case inverse images may
be computed levelwise, and hence it is obvious that they are exact.

Since it is well-known that the derived category of constructible sheaves
on an algebraic stack is equivalent to the derived category of
contructible (cartesian) sheaves on the simplicial scheme associated
to a smooth covering by a scheme, and since furthermore this
equivalence extends to the $l$-adic case (see \cite[Prop.\
10.3]{LO2}), our approach is equivalent to the usual one.

Let us finally mention that derived inverse images and derived direct
images for a morphisms $f:\Xscr\r \Yscr$ between algebraic stacks may indeed be computed on the
corresponding simplicial schemes. For $Lf^{-1}$ this is in fact true by
the construction, and for $Rf_\ast$ this is true for representable morphisms
\cite[Corollary 5.5.6]{Behrend1}.

\begin{rem}
A result like Theorem \ref{anr} for
$\Fscr=\underline{\CC}$ was proved by Deligne in~\cite{deligneweil2}
using hypercoverings. Note that in this particular case Deligne
constructs $A$ in such a way that $A_{<0}=0$, so 
that  Sullivan's theory of minimal models applies to it. Our
construction does not have this property, which is why need the more
general Theorem \ref{anr}.
\end{rem}

\begin{rem} Theorem \ref{anr} has an obvious analogue for
a finite collection of complexes of $(\Fscr_i)_{i=1,\ldots,n}$. In this case
$A$ becomes a DG-category with $n$ objects. We leave the precise formulation
to the reader.
\end{rem}

The proof of Theorem \ref{anr} occupies the remainder of this section. 
After some preliminaries, we discuss in \S\ref{passing-finite-field-sec} the reduction to a finite
field purely on the level of derived categories (essentially following \cite[\S6]{BBD}, but with appropriate changes to
deal with the equivariant categories). Then in
\S\ref{DG-passing-finite-field-sec} we explain how to lift the entire picture to the DG-level, using standard DG-enhancements
(see \S\ref{DG-enhancement-sec}).

%as well the extension to the equivariant case.
%\section{Preliminaries}

\subsection{The $m$-adic derived category of a topos}
\label{sectopoi}
Let $X$ be a topos, and let $X^{\NN}$ be the associated topos of $\NN$-indexed\footnote{We follow the
American convention that $0\not\in \NN$.} inverse systems
over $X$. Let $\Lambda$ be a complete discrete valuation ring with maximal ideal $m$. 

%For use below we note that 
There is a morphism of topoi
\[
\pi:X^{\NN}\r X
\]
such that $\pi^{-1}\Fscr$ is the constant inverse system $(\Fscr)^n$ and $\pi_\ast\Gscr=\underset{n}{\varprojlim}
%\invlim_n 
\Gscr^n$.
Also, let $i_n:X\r X^{\NN}$ be the morphism of topoi such that
$i^{-1}_n\Fscr=\Fscr^n$.

Put $\Lambda^n=\Lambda/m^n$ and let
$\Lambda^\bullet=(\Lambda^n)_n$ be the corresponding inverse system of rings. We consider $X^{\NN}$
to be ringed by the constant sheaf $\Lambda^\bullet$. 

Let $C(X^{\NN},\Lambda^\bullet)$ denote the category of complexes of $\Lambda^\bullet$-modules on $X$, 
and let $D(X^{\NN},\Lambda^\bullet)$
be the corresponding derived category. 

The composition
\[
L\pi^\ast R\pi_\ast:D(X^{\NN},\Lambda^\bullet)\r D(X^{\NN},\Lambda^\bullet)
\]
is called the \emph{normalization functor} and is denoted by $\widehat{?}$. By construction there is a natural transformation
\[
\Fscr\r \widehat \Fscr
\]
An object $\Fscr$ in $D(X^{\NN},\Lambda^\bullet)$ is said to be
\emph{normalized} if $\Fscr\r\widehat{\Fscr}$ is an isomorphism. 
 We
write $D(X^{\NN},\Lambda^\bullet)^{\text{norm}}$ for the full
subcategory of $D(X^{\NN},\Lambda^\bullet)$ consisting of normalized
complexes.
Obviously $D(X^{\NN},\Lambda^\bullet)^{\text{norm}}$ is triangulated, and since it
is a subcategory of a derived category, it is naturally enhanced, but
it is not clear if it has any other desirable properties.

\medskip

The good behaviour of $D(X^{\NN},\Lambda^\bullet)^{\text{norm}}$
depends crucially on the good behaviour of the normalization functor
$\widehat{?}$, which in turn depends on the good behaviour of inverse
limits. The good behaviour of inverse limits on sites depends on the local
boundedness of cohomological dimension---see \cite[\S2]{LO2} and in
particular \cite[Assumption 2.1.2]{LO2}. Luckily these hypotheses are
satisfied in the cases we consider (nice algebraic
stacks and simplicial schemes, see \cite{LO2} for the precise setting).

\medskip

So from now on we assume that we are in the setting of \cite{LO1,LO2}. 
An important feature is that the normalization functor may be computed
locally on stacks \cite[Cor.\ 3.8]{LO2} and levelwise on simplicial
schemes (see the proof of
\cite[Prop.\ 10.2]{LO2}).

Also,
we have that a complex is normalized if and only 
for $n\ge m$ the canonical maps 
\begin{equation}
\label{canmaps}
\Fscr^n\Lotimes_{\Lambda^n}\Lambda^m\mapsto \Fscr^m
\end{equation}
are isomorphisms, where here and below we use the convention
$\Fscr^n=i^\ast_n\Fscr$.  For algebraic stacks this is \cite[Prop.\
3.5]{LO2}, and for simplicial schemes it follows from the fact that the
normalization functor may be computed levelwise.

\medskip

Now assume that $c$ is the category of Cartesian constructible complexes in $X$
(this notion is defined both in the stack and in the simplicial context). 
 Then we 
define $D_c(X,\Lambda)$ as the full subcategory of $D(X^{\NN},\Lambda^\bullet)^{\text{norm}}$, consisting of
objects whose cohomology are AR-$m$-adic objects
in $c$ (see \cite{Jouanolou}). 

\medskip

It follows from \cite{LO2} that $D_c(X,\Lambda)$ may be constructed
differently as the category 
\begin{equation}
\label{stackcategory}
\mathsf{D}_c(X,\Lambda)\overset{\text{def}}{=}D_c(X^{\NN},\Lambda^\bullet)/
D_{\text{null}}(X^{\NN},\Lambda^\bullet),
\end{equation}
where $D_c(X^{\NN},\Lambda^\bullet)$ and
$D_{\text{null}}(X^{\NN},\Lambda^\bullet)$ are the full subcategories of
$D(X^{\NN},\Lambda^\bullet)$ spanned by objects $\Fscr$ whose
cohomology are, respectively, locally AR-$m$-adic objects and locally AR-null objects in $c$
(see \cite{Jouanolou}). There are inverse equivalences
\begin{equation}
\label{inverse}
\xymatrix@1{
D_{c}(X,\Lambda)\ar@/^1em/[r]^{can} & \mathsf{D}_c(X,\Lambda)\ar@/^1em/[l]^{\widehat{?}}
}
\end{equation}
For algebraic stacks this is \cite[Thm 3.9]{LO2} and for simplicial schemes it is \cite[Prop.\ 10.2]{LO2}.

\medskip

It is clear that the standard truncation functors $\tau_{\le n}$,
$\tau_{\ge n}$ on $D(X^{\NN},\Lambda^\bullet)$ descend to
$\mathsf{D}_c(X,\Lambda)$, and hence they define a $t$-structure on
$\mathsf{D}_c(X,\Lambda)$ with the heart consisting of $m$-adic sheaves (see \cite{Jouanolou})
in $c$, and the same holds via \eqref{inverse} also for
$D_c(X,\Lambda)$. 

\medskip

Using the $t$-structure introduced in the previous paragraph we now
define the bounded $l$-adic derived category $\mathsf{D}^b_c(X,\Lambda)
\subset \mathsf{D}_c(X,\Lambda)$, and via \eqref{inverse}, also
${D}^b_c(X,\Lambda)
\subset {D}_c(X,\Lambda)$. Note however, that ${D}^b_c(X,\Lambda)$
is not a subcategory of  $D^b(X^{\NN},\Lambda^\bullet)$!

\subsection{A spectral sequence for D-topoi}

The following result should be standard but we did not find the exact
result in the literature.

Let $D$ be a small category and let $\pi:X_\bullet\r D$ be a $D$-topos \cite{SGA4}.
We write $X_i$ for $\pi^{-1}(i)$ and 
 $X_\bullet$ for the associated total topos. Thus, objects 
in $X_\bullet$  are collections   of objects
$(\Fscr_i\in X_i)_i$ and maps
$(\phi_{\alpha}:\alpha^{-1} \Fscr_j\r \Fscr_i)_\alpha$, for  $\alpha:i\r j$, satisfying the standard
cocycle condition. 

We assume that $X_\bullet$ is ringed with 
a  sheaf of rings $\Lambda_\bullet$ which is flat in the sense that
the transition morphisms $\alpha:(X_i,\Lambda_i)\r (X_j,\Lambda_j)$ 
have the property that $\alpha^\ast$ is exact.
We say that $\Lambda_\bullet$-module $\Fscr$ is Cartesian if the $\phi_\alpha$'s induce
isomorphisms $\alpha^{\ast} \Fscr_j\r \Fscr_i$. We will denote the latter isomorphisms
also by $(\phi_\alpha)_\alpha$. Note that $\Lambda_\bullet$ is automatically Cartesian.
\begin{prop} 
\label{ffcor}
 Let $\Fscr,\Gscr$ be $\Lambda_\bullet$-modules with $\Fscr$ being Cartesian. Then the
assignment
\[
i\mapsto \Ext^{j}_{\Lambda_i}(\Fscr_i,\Gscr_i)
\]
defines a contravariant functor $D\r \Ab$ (where $\Ab$ is the category of abelian groups).
Furthermore, there is a convergent spectral sequence
\[
E^{pq}_2=R^p
%\invlim_u 
\underset{u}{\varprojlim}
\Ext^q_{\Lambda_u}(\Fscr_u,\Gscr_u)\Rightarrow \Ext_{\Lambda_\bullet}^n(\Fscr,\Gscr)
\]
\end{prop}
\begin{proof} Let $\alpha:v\r u$ be a morphism in $D$. We obtain a morphism
\begin{align*}
\Ext^{j}_{\Lambda_u}(\Fscr_u,\Gscr_u)
\r
\Ext^{j}_{\Lambda_v}(\alpha^\ast\Fscr_u,\alpha^\ast\Gscr_u)
\cong\Ext^{j}_{\Lambda_v}(\Fscr_v,\alpha^\ast\Gscr_u)
\r
\Ext^{j}_{\Lambda_v}(\Fscr_v,\Gscr_v)
\end{align*}
where the first morphism is induced by $\alpha^\ast$ and the latter two  are induced by $\phi^{-1}_\alpha$ and $\phi_\alpha$
respectively. This yields the asserted functoriality. 
It is easy to see that
\[
\Hom_{\Lambda}(\Fscr,\Gscr)=\underset{u}{\varprojlim}
%\invlim_u 
\Hom_{\Lambda_u}(\Fscr_u,\Gscr_u).
\]
\def\Sh{\operatoname{Sh}}
We will construct the spectral sequence as the Grothendieck spectral sequence for the above
 composition of functors. For this we have to show that $\Iscr_u$ is injective for $\Iscr$ injective
and furthermore that
$\Hom_{\Lambda_u}(\Fscr_u,\Iscr_u)$ is acyclic for $\underset{u}{\varprojlim}$.
%$\invlim_u$.

As in \cite{SGA4,LO1}, we may assume that $\Iscr=i_{v,\ast} I$
where $i_v:X_v\r X$ is the morphism of topoi $(\Fscr_v)_v\mapsto \Fscr_v$ and $I$ is an injective
$\Lambda_v$-module in $X_v$. 
We have the general formula
\[
(i_{v,\ast}\Iscr)_u=\prod_{\alpha:v\r u} \alpha_\ast \Iscr
\]
(which does not require $\Iscr$ to be injective). Since $\alpha_\ast$ has an exact left
adjoint (by flatness), this implies that $(i_{v,\ast}\Iscr)_u$ is injective.

\def\Set{\operatorname{Set}}
We also get
\begin{align*}
\Hom_{\Lambda_u}(\Fscr_u,\Iscr_u)&=\Hom_{\Lambda_u}(\Fscr_u,(i_{v,\ast} I)_u)\\
&=\prod_{\alpha:v\r u} \Hom_{\Lambda_u}(\Fscr_u,\alpha_\ast I )\\
&=\prod_{\alpha:v\r u} \Hom_{\Lambda_v}(\alpha^\ast \Fscr_u, I )\\
&=\prod_{\alpha:v\r u} \Hom_{\Lambda_v}(\Fscr_v, I )\\
&=\Set\bigl(D(v,u),\Hom_{\Lambda_v}(\Fscr_v, I )\bigr),
\end{align*}
where $\Set(?,?)$ denotes morphisms in the category of sets.
Thus, we have to show that 
\[
\Set\bigl(D(v,-),\Hom_{\Lambda_v}(\Fscr_v, I )\bigr)
\]
is acyclic for ${\varprojlim}$. 
Let $\Inv(D)$ denote the category of inverse systems
of functors $D\to\Ab$. 
For $W$ in $\Inv(D)$ we have
\[
\Hom_{\Inv(D)}\bigl(W,\Set\bigl(D(v,-),\Hom_{\Lambda_v}(\Fscr_v, I )\bigr)\bigr)=\Hom_{\Lambda_v}(W_v\otimes_\ZZ \Fscr_v,I).
\]
We claim that the derived functors of 
\[
W\mapsto \Hom_{\Lambda_v}(W_v\otimes_\ZZ \Fscr_v,I)
\]
are given by
\[
\Hom_{\Lambda_v}(\Tor^{\ZZ}_i(W_v,\Fscr_v),I)
\]
This is clear if we replace $W$ by a projective resolution and use the fact that
projectives are summands of direct sums of objects of the form $\ZZ D(-,w)$, which are $\ZZ$-flat when evaluated
at every $v$. Thus, we conclude that
\[
\Ext^i_{\Inv(D)}\bigl(W,\Set(D(v,-),\Hom_{\Lambda_v}(\Fscr_v, I ))\bigr)
=\Hom_{\Lambda_v}(\Tor^{\ZZ}_i(W_v,\Fscr_v),I).
\]

Let $\underline{\ZZ}$ be the constant inverse system on $D$ with value $\ZZ$. We finally conclude that
for $i>0$ one has
\begin{align*}
R^i{\varprojlim}
%\invlim 
\Set(D(v,-),\Hom_{\Lambda_v}(\Fscr_v, I ))&=
\Ext^i_{\Inv(D)}(\underline{\ZZ}, \Set(D(v,-),\Hom_{\Lambda_v}(\Fscr_v, I )))\\
&=\Hom_{\Lambda_v}(\Tor^{\ZZ}_i(\ZZ,\Fscr_v),I)\\
&=0\qed
\end{align*}
\def\qed{}\end{proof}

\subsection{DG-enhancements of derived categories and derived functors}\label{DG-enhancement-sec}
\subsubsection{The ``standard'' DG enhancement of a derived category}
\label{secstandard}
If $\Ascr$ is an abelian category then the derived category $D(\Ascr)$ of $\Ascr$ is
the category of complexes $C(\Ascr)$ over $\Ascr$ localized at quasi-isomorphisms.

If $\Ascr$ is a Grothendieck category  then we denote by
$D^{\dg}(\Ascr)$ the ``standard'' DG-enhancement of $D(\Ascr)$ using fibrant resolutions. In other words,
$\Ob(D^{\dg}(\Ascr))=\Ob(D(\Ascr))=\Ob(C(\Ascr))$, and
 for every $C\in \Ob(C(\Ascr)$ we fix a quasi-isomorphism
$C\r I_C$ such that $I_C$ is fibrant for the standard injective model structure on $C(\Ascr)$.
The $\Hom$-complexes in $D^{\dg}(\Ascr)$ are then given by 
\[
\underline{\Hom}_{D^{\dg}(\Ascr)}(C,D)\overset{\text{def}}{=}\underline{\Hom}_{C(\Ascr)}(I_C,I_D)
\]
Although it is not  essential we will assume that if $C$ is left bounded
then~$I_C$ is left bounded as well. In this case $I_C$ is just a classical injective resolution.
\subsubsection{Co-quasi-functors}
\label{co-quasi-sec}
The DG-functor is often too rigid a notion to compare different DG-categories.
A suitable weakened notion is given by ``quasi-functors''  introduced
in \cite{Keller1}. 
For technical reasons we will use the dual
version which we call ``co-quasi-functors''.

If $\mathfrak{a}$, $\mathfrak{b}$ are DG-categories then a co-quasi-functor
$M:\mathfrak{b}\r \mathfrak{a}$ is simply a quasi-functor \cite{Keller1} $\mathfrak{b}^\circ\r \mathfrak{a}^\circ$.
To be more concrete: an $\mathfrak{a}-\mathfrak{b}$-bimodule $M$ is a DG-bifunctor $M:\mathfrak{b}^\circ \times
\mathfrak{a}\r C(\Ab)$.
 We say that such an $M$ is a co-quasi-functor $\mathfrak{b}\r \mathfrak{a}$
if for every
$B\in \mathfrak{b}$ there is an $A\in \mathfrak{a}$, as well as a morphism of DG-functors
$\mathfrak{a}(A,-)\r M(B,-)$ (by enriched Yoneda this is the same as an element of $x_B\in Z^0 M(B,A)$),
such that for all $A'\in\mathfrak{a}$ the induced map $\mathfrak{a}(A,A')\r M(B,A')$
is a quasi-isomorphism. It will often be convenient to choose for every $B$ a particular $A$,
which we will then denote by $M(B)$.

A co-quasi-functor $M$  induces an honest functor $H^0(\mathfrak{a})\r H^0(\mathfrak{b})$
sending $B$ to $M(B)$. We denote this functor by $H^0(M)$.
\begin{ex}
\label{standardexample} Let $F$ be a left exact functor between Grothendieck categories $\Cscr\r\Dscr$. Then
the right derived functor
\[
RF:D(\Cscr)\r D(\Dscr)
\]
can be lifted to a co-quasi-functor
\[
RF^{\dg}:D^{\dg}(\Cscr)\r D^{\dg}(\Dscr)
\]
by setting
\[
RF^{\dg}(C,D)=\underline{\Hom}_{C(\Dscr)}(FI_C,I_D)
\]
\end{ex}
We will need the following standard result \cite{Keller1}.
\begin{lem} \label{co-quasi}
Assume that $M:\mathfrak{b}\r \mathfrak{a}$ is a co-quasi-functor such that $H^0(M)$ is fully
faithful. Then for every $B\in \mathfrak{b}$ one has that $\mathfrak{b}(B,B)$ and $\mathfrak{a}(M(B),M(B))$
are isomorphic in the homotopy category of DG-algebras.
\end{lem}
\begin{proof} Put $A=M(B)$ and consider the DG-algebra
\[
\Lambda=\begin{pmatrix}
\mathfrak{a}(A,A) & M(B,A)[-1]\\
0 & \mathfrak{b}(B,B)
\end{pmatrix}
\]
with the differential
\[
\begin{pmatrix}
f&s^{-1}m\\
0&g
\end{pmatrix}
\mapsto
\begin{pmatrix}
df & s^{-1}(x_B g-fx_B-dm)\\
0&dg
\end{pmatrix}.
\]
Then we see that $\Lambda$ projects to $\mathfrak{a}(A,A)$ and $\mathfrak{b}(B,B)$,
and it is easy to see that both maps are quasi-isomorphisms.
\end{proof}
\begin{rem} The notions  introduced in this section have obvious $k$-linear versions in the case when $k$ is a commutative
base ring. We will use this without further comment.
\end{rem}

%\subsection{Comparing derived categories}
\subsection{Passing to a finite field}\label{passing-finite-field-sec}

For a scheme $Y$ of finite type over $\CC$ and a noetherian
ring $R$ of finite global dimension we temporarily write
$D^b_c(Y(\CC),R)_{\text{lit}}$ for the category of bounded complexes 
of sheaves of $R$-modules on $Y(\CC)$ with constructible homology, in order to avoid confusion with the
$l$-adic derived category which will be introduced later.

\subsubsection{The Bernstein-Lunts derived category}
Let $X$ be a scheme of finite type over $\CC$
and let $G$ be a linear algebraic group acting on $X$. The (topological) derived
category $D^b_c([X(\CC)/G(\CC)],\CC)_{\text{lit}}$ of $G(\CC)$-equivariant $\CC$-linear sheaves on $X(\CC)$
was defined by Bernstein and Lunts in \cite{BerLun}. 

Let $(X(\CC)/G(\CC))_\bullet$ be the standard simplicial space
associated to the $G(\CC)$-action on $X(\CC)$.  We view
$(X(\CC)/G(\CC))_\bullet$ as a topos indexed by $\Delta^{\circ}$, where
$\Delta=\{[n]\mid n\in \NN\}$ is the standard simplicial category.
Let $R$ be a noetherian ring of finite global dimension.  The model of
the $G(\CC)$-equivariant derived category we will use is
\begin{equation}
\label{level}
D^b_c([X(\CC)/G(\CC)],R)_{\text{lit}}\overset{\text{def}}{=}D^b_{c,\text{cart}}\bigl((X(\CC)/G(\CC))_\bullet,R\bigr)_{\text{lit}}
\end{equation}
The right-hand side of \eqref{level} is the full subcategory
of $D\bigl((X(\CC)/G(\CC))_{\bullet},R\bigr)_{\text{lit}}$ consisting of bounded complexes with Cartesian, levelwise constructible
homology.

%Let $X$ be a scheme of finite type over $\CC$.
We let $Y$ be either $X_{\text{et}}$ (the small \'etale site of $X$) or else $(X/G)_{\bullet,\text{et}}$, the
simplicial scheme of a linear algebraic group $G$ acting on $X$. We write $Y(\CC)$ for the
corresponding topological topoi $X(\CC)$ or $(X(\CC)/G(\CC))_\bullet$.
In the course of the proof the
ground field will change. 

\subsubsection{Changing the coefficient ring}
\label{secliteral}
For use below we make the following definition. Let $R\r S$ be a morphism between commutative rings and let $\Ascr$
be an $R$-linear category.  Then the category $\Ascr_S$
has the same objects as $\Ascr$ but its $\Hom$-spaces are defined by
\[
\Hom_{\Ascr_S}(\Fscr,\Gscr)=S\otimes_R \Hom_{\Ascr}(\Fscr,\Gscr)
\]
Note that if $\Ascr$ is abelian or triangulated then this will usually not be the case
for~$\Ascr_S$. 

We recall the following fact.
\begin{prop}
\label{changecoeff}
\begin{enumerate}
\item Let $S$ be commutative noetherian ring of finite global dimension which is a flat ring extension of $R$. Then the coefficient extension functor
\[
D_c^b(Y(\CC),R)_{\text{lit}}\r D_c^b(Y(\CC),S)_{\text{lit}}:\Fscr\mapsto S\otimes_R \Fscr
\]
induces a fully faithful functor
\[
D_c^b(Y(\CC),R)_{\text{lit},S}\r D_c^b(Y(\CC),S)_{\text{lit}}.
\]
\item We have
\[
D_c^b(Y(\CC),\CC)_{\text{lit}}=\bigcup_S D_c^b(Y(\CC),S)_{\text{lit},\CC},
\]
where $S$ runs through subrings of $\CC$ which are finitely generated over $\ZZ$ and of
finite global dimension.
\item We have
\[
D_c^b(Y(\CC),\CC)_{\text{lit}}=\bigcup_S D_c^b(Y(\CC),\Lambda)_{\text{lit},\CC},
\]
where $\Lambda$ runs through complete discrete valuation rings in $\CC$ with finite residue field.
\end{enumerate}
\end{prop}
\begin{proof}
Assume first that $Y$ is a scheme. 
  By Hironaka's semi-algebraic triangulation theorem \cite{Hironaka},
  the study of constructible sheaves on $X(\CC)$ can be reduced to the
  study of constructible sheaves on \emph{finite} simplicial complexes,
  which is a purely combinatorial problem \cite[\S8.1]{KSI}. From this
  (1) and (2) follow easily.

To deduce (3) from (2) one needs a bit of field theory. Let $S\subset \CC$ be finitely generated over $\ZZ$. We claim  
that there is a complete discrete valuation ring with finite residue field $\Lambda$ such that $S\subset \Lambda\subset \CC$. This clearly implies (3).

First, note that from the fact that $\ZZ_l/\ZZ$ has infinite
transcendence degree one easily constructs an injective map $S\r
\Lambda'$, where $\Lambda'$ is a complete discrete valuation ring with
finite residue field. Indeed, write $S$ as an integral extension of
some $S_0=\ZZ[X_1,\ldots,X_N]$. Then we may embed $S_0\subset \ZZ_l$
as $\ZZ$-algebra. Extend the induced valuation on the fraction field
of $S_0$ to the fraction field of $S$. This gives an injective map
$S\r \Lambda_0$, where $\Lambda_0$ is a discrete valuation ring with the
finite residue field. Now replace $\Lambda_0$ by its completion
$\Lambda'$.

Thus, in particular $\Lambda'$ is a finite
extension of some $\ZZ_l$. Thus, the algebraic closure of the quotient
field of $\Lambda'$ is isomorphic to $\overline{\QQ}_l$.

Let $K$ be the quotient field of $S$. We now have the following diagram
\[
\xymatrix{
K\ar[r]\ar[d] & \CC\\
\overline{\QQ}_l
}
\]
Now $\CC$ and $\overline{\QQ}_l$ are both algebraically closed fields.
Furthermore, since $K$ is countable, $\CC/K$ and $\overline{\QQ}_l/K$ have the same transcendence
degree. It follows by \cite[Theorem 1.12]{Hungerford} that there is an isomorphism
$\tau:\overline{\QQ}_l\xrightarrow{\cong} \CC$ which is the identity on~$K$. It now
suffices to take $\Lambda=\tau(\Lambda')$.

\medskip

Now consider the case $Y=[X/G]$. (1) follows immediately from  Proposition \ref{ffcor} and the scheme case.

For (2)  we note that a Cartesian constructible sheaf
on $X(\CC)$ is the same as a constructible sheaf with a finite amount of descent data.  Hence,
this descent data can be descended to a finite extension of~$\ZZ$. A similar argument applies
to the extension data which build a complex from its cohomology.
The proof of (3) does not have to be adapted.
\end{proof}

\subsubsection{Passing to the topological $m$-adic  derived category}
\label{sectopm}
Now assume that we are in the setup of \S\ref{sectopoi}, so
$(\Lambda,m)$ is a complete
discrete valuation ring. We assume in addition that its residue field is finite. 
Recall that we have a morphism of topoi $\pi:X^{\NN}\r X$, where $X^{\NN}$ is ringed
by $\Lambda^\bullet=(\Lambda^n)_n$ with $\Lambda^n=\Lambda/m^n$.

\begin{thm}
The functors $L\pi^\ast$ and $R\pi_\ast$ define inverse equivalences between
$D^b_c(Y(\CC),\Lambda)_{\text{lit}}$ and $D^b_c(Y(\CC),\Lambda)$.
\end{thm}
\begin{proof}  For schemes this is proved in \cite{Ek}, so let us consider the case $Y=(X/G)_\bullet$. 

\medskip

Note that an object in $D^b((X(\CC)/G(\CC))_{\bullet}^{\NN},\Lambda^\bullet)^{\text{norm}}$ is first and foremost a system
of complexes $(\Fscr^n_m)_{n,m}$ of sheaves of $\Lambda^n$-modules on $G(\CC)^{m}\times X(\CC)$ equipped with various maps. The fact that
$(\Fscr^n_m)_{n,m}$ is normalized implies that the morphism
\[
\Fscr^n_\bullet \Lotimes_{\Lambda^n}\Lambda^{n'}\r \Fscr^{n'}_\bullet
\]
is an isomorphism (in the derived category). It is easy to see that this is equivalent to
\[
\Fscr^n_{m} \Lotimes_{\Lambda^n}\Lambda^{n'}\r \Fscr^{n'}_m
\]
being an isomorphism for each $m$. In other words, this is equivalent to $\Fscr^\bullet_m$ being normalized as a complex of sheaves on the ringed topos
 $((G(\CC)^{m}\times X(\CC))^{\NN},\Lambda^\bullet)$ for each $m$. 

Now let $X_1=G^{m_1}\times X$, $X_2=G^{m_2}\times X$, and let $f:[m_2]\r [m_1]$ be a morphism in $\Delta$.
It corresponds to a morphism of schemes $X_1\r X_2$.
The object $(\Fscr^n_m)_{m,n}$ is Cartesian  if and only if for each
such $f$ the map $f^{\ast} \Fscr_{m_2}^\bullet\r \Fscr^\bullet_{m_1}$ is an isomorphism for all $n$ in 
$D^b(X_1(\CC)^\NN,\Lambda^\bullet)^{\text{norm}}$.

We now claim that $R\pi_\ast$ and $L\pi^\ast$ define inverse equivalences
\begin{equation}\label{X/G-norm-equivalence}
D^b_{c}((X(\CC)/G(\CC))^{\NN}_\bullet,\Lambda^\bullet)^{\text{norm}}
\longleftrightarrow D^b_{c}((X(\CC)/G(\CC))_\bullet,\Lambda)_{\text{lit}}.
\end{equation}
So we have to show that the unit and counit  morphisms $\Id\r R\pi_\ast L\pi^\ast$, $L\pi^\ast R\pi_\ast\r \Id$ are isomorphisms.
Since $R\pi_\ast$ and $L\pi^\ast$ can be computed levelwise, we are reduced to the scheme case which is known.

If $f$ is as above then we have on the nose $f^\ast
L\pi^\ast=L\pi^\ast f^\ast$. Hence, if $\Fscr\in
D^b_{c}((X(\CC)/G(\CC))_\bullet,\Lambda)_{\text{lit}}$ has Cartesian
cohomology then so does $L\pi^\ast \Fscr$.  Since $R\pi_\ast$ and
$L\pi^\ast$ are inverses, we also have $f^\ast R\pi_\ast =R\pi_\ast
f^\ast$, from which it follows that if $\Gscr\in
D^b_{c}((X(\CC)/G(\CC))^{\NN}_\bullet,\Lambda^\bullet)^{\text{norm}}$
has Cartesian cohomology then so does $R\pi_\ast\Gscr$.

Thus, to show that \eqref{X/G-norm-equivalence} induces an
equivalence of the subcategories
$$
D^b_{c,\text{cart}}((X(\CC)/G(\CC))^{\NN}_\bullet,\Lambda^\bullet)^{\text{norm}}
\simeq D^b_{c,\text{cart}}((X(\CC)/G(\CC))_\bullet,\Lambda)_{\text{lit}},
$$
it remains to verify the compatibility with the condition that the cohomology of
objects in \break
$D^b_{c,\text{cart}}((X(\CC)/G(\CC))_\bullet,\Lambda)_{\text{lit}}$ are
  Cartesian locally AR-$m$-adic objects.
 However, it is clear that if $\Fscr$ is Cartesian then it
  is locally AR-$m$-adic if and only this is the case for
  $\Fscr_1$. So we are again reduced to the scheme case which is known.
 \qed

\def\qed{}
\end{proof}
\subsubsection{Passing to the \'etale $m$-adic  derived category}
In \cite[Cor.\ 3.17]{LO2} it is shown (for algebraic stacks) that 
 $D_c^b(Y,\Lambda)$ coincides with the $m$-adic derived
category as defined by Deligne in \cite{deligneweil2} (see also \cite{BBD}).
In other words,
for
$\Fscr,\Gscr\in D_c^b(Y,\Lambda)$ the obvious map
\begin{equation}
\label{deligne}
\Hom_{\Lambda}(\Fscr,\Gscr)\r \underset{n}{\varprojlim}
%\invlim_n 
\Hom_{\Lambda^n}(\Fscr^n,\Gscr^n)
\end{equation}
is an isomorphism.\footnote{The term 
$\underset{n}{\varprojlim}^1
%\invlim^1_n
\Ext^{-1}(\Fscr^n,\Gscr^n)$
 in \cite[Cor.\ 3.17]{LO2} is zero since the groups $\Ext^{-1}(\Fscr^n,\Gscr^n)$ are finite in the current setting.}

Now there is a morphism of topoi 
\[
\epsilon:Y(\CC)\r Y_{\text{et}}.
\]
In the scheme case this is \cite[p149]{BBD}.  In the $G$-equivariant case we define it levelwise.

It is easy to see that $\epsilon^\ast$ preserves normalized complexes with AR-$m$-adic cohomology and thus defines a functor
\[
\epsilon^\ast:D_c^b(Y,\Lambda)\r D_c^b(Y(\CC),\Lambda).
\]
\begin{prop} \cite[p149]{BBD}
The functor $\epsilon^\ast$ defines an equivalence between 
$D^b_c(Y,\Lambda)$ and $D^b_c(Y(\CC),\Lambda)$.
\end{prop}
\begin{proof}
  Using \eqref{deligne} and its analogue for $Y(\CC)$ together with
  Proposition \ref{ffcor}, it follows immediately that $\epsilon^\ast$
  is fully faithful. To prove that it is essentially surjective it
  will be convenient to use $\mathsf{D}^b_c(Y,\Lambda)$ (see \eqref{stackcategory}).  We have a
  commutative diagram
\[
\xymatrix{
{D}^b_c(Y_{\text{et}},\Lambda)\ar[r]^{\text{can}}_{\cong}\ar[d]_{\epsilon^\ast}& 
\mathsf{D}^b_c(Y_{\text{et}},\Lambda)\ar[d]^{\epsilon^\ast}\\
{D}^b_c(Y(\CC),\Lambda)\ar[r]_{\text{can}}^{\cong}& 
\mathsf{D}^b_c(Y(\CC),\Lambda)
}
\]
So it is sufficient to prove that the right vertical arrow $\epsilon^\ast$ is
essentially surjective. To this end it suffices to prove that if
$\Fscr$ is a constructible AR-$m$-adic object on $Y(\CC)$ then it can
be lifted under $\epsilon^\ast$. But this is clear since
$\epsilon^\ast$ is an equivalence on constructible sheaves.
\end{proof}

\subsubsection{Reduction to the algebraic closure of a finite field}
\label{secfinchar}
The following result is stated in \cite[p155]{BBD}.
\begin{lem}
\label{lemma1}
Suppose we have a finite diagram of schemes of finite type over a noetherian scheme $S$ and on each scheme
a finite number of constructible $\Lambda^n$-modules. 

If we perform a finite sequence of standard homological operations 
$R^\ast f_\ast$, $R^\ast f_!$, $f^\ast$, $R^\ast f^!$, $\uTor_p$, $\uExt^p$ on the given constructible sheaves,
then the result is constructible and compatible with the base change after replacing $S$ with
a suitable dense open subset. 
\end{lem}
\begin{proof} (see \cite{BBD})
The crucial case consists of $R^\ast f_\ast$
applied to a single constructible sheaf. This case is \cite[Th.\ finitude 1.9]{Finitude}.
\end{proof}
We recall a construction given in \cite[p156]{BBD}. Let $X/\CC$ be of finite type. Let $\Tscr$ be
a stratification of $X$ with smooth strata $T\in \Tscr$. For each $T\in \Tscr$ let $\Lscr(T)$
be a finite family of locally constant sheaves  (for the \'etale topology) of $\Lambda^1$-modules on $T$.

\begin{lem}
\label{lemma2}
There exists $A$, a subring $\CC$ of finite type over $\ZZ$, such that for $S=\Spec A$ 
there exist $(X_S,\Tscr_S,\Lscr_S)$ that give $(X,\Tscr,\Lscr)$  through base extension, 
such that in addition the following properties hold.
\begin{enumerate}
\item The strata $T\in \Tscr_S$ are smooth over $S$ with geometrically connected fibers.
\item For $F,G$ of the form $j_!L$ with $j:T\hookrightarrow X_S$ in $\Tscr_S$ and $L\in \Lscr(T)$, the sheaves $\uExt^q_{X_S}(F,G)$
are compatible with the base change in $S$.
\item Let $a_S:X_S\r S$ be the structure morphism. With $F,G$ as above, the sheaves
$R^p a_\ast \uExt^q_{X_S}(F,G)$ are locally constant and compatible
with the base change. 
\end{enumerate}
\end{lem}

\begin{proof} This is stated without proof in \cite{BBD}.  One may
  argue as follows. As the first step we may choose $S=\Spec A$, such that
  $(X_S,\Lscr_S,\Tscr_S)$ exist and the strata in $\Tscr_S$ are
  smooth over $S$ with geometrically connected fibers.  Then we apply
  Lemma \ref{lemma1} to obtain constructibility and compatibility with
 the base change of (the finite number of) the sheaves $R^p a_\ast
  \uExt^q_{X_S}(F,G)$.  Shrinking $S$ further we may assume that the
  $R^p a_\ast \uExt^q_{X_S}(F,G)$ are locally constant.
\end{proof}

Let $A$ be as in the previous lemma.  Pick $s\in S$. Then according to \cite{BBD}, there exists a strict Henselian discrete valuation ring
$A\subset V\subset \CC$, whose residue field is the algebraic closure of the residue field of $s$.  We denote the closed point of $\Spec V$ also by $s$.
Then we have a base extension diagram
\[
X\xrightarrow{u} X_V\xleftarrow{i} X_s
\]
Let $D^b_{\Tscr,\Lscr}(X,\Lambda^n)$ be the full subcategory of $D^b_c(X,\Lambda^n)$ consisting of complexes $K$, such
that $\Tor_j^{\Lambda}(\Lambda^1,H^i(K)|T)$, for $T\in \Tscr$, is an extension of objects in $\Lscr(T)$ for all $i$, $j$. We define
$D^b_{\Tscr_V,\Lscr_V}(X_V,\Lambda^n)$, $D^b_{\Tscr_s,\Lscr_s}(X_s,\Lambda^n)$ in a similar way.

\begin{lem} 
\label{lemma4}
There are equivalences of categories
\[
D^b_{\Tscr,\Lscr}(X,\Lambda^n)    \xleftarrow{u^\ast} D^b_{\Tscr_V,\Lscr_V}(X_V,\Lambda^n) \xrightarrow{i^\ast} D^b_{\Tscr_s,\Lscr_s}(X_s,\Lambda^n).
\]
\end{lem}

\begin{proof} (see \cite{BBD})
Let $F=j_{1\ast} L_1$, $G=j_{2\ast} L_2$, where $j_i:T_i\r X_V$ in $\Tscr_V$, for $i=1,2$, are the embeddings of strata, and
$L_i\in \Lscr(T_i)$ (these are sheaves of $\Lambda^1$-modules). We have to prove that  
$\Ext^p_{X_V,\Lambda^n}(F,G)$ does not change if we replace $V$ by its generic or special point. Since 
$F$ and $G$ are sheaves of $\Lambda^1$-modules, we have
\[
\Ext^p_{X_V,\Lambda^n}(F,G)=\Ext^p_{X_V,\Lambda^1}(F ,R\Hom_{\Lambda^n}(\Lambda^1,G)),
\]
and $H^i(R\Hom_{\Lambda^n}(\Lambda^1,G))=G$ for all $i\ge 0$. Hence, 
it suffices to consider
the case $n=1$, by invoking the appropriate spectral sequence. 

Now we have
\[
R\Hom_{X_V}(F,G)=R\Gamma(\Spec V, R a_\ast  \uRHom_{X_V}(F,G))
\]
Using Lemma \ref{lemma2}(3), we are reduced to checking
for a locally constant sheaf $L$ on $(\Spec V)_{\text{et}}$ one has
\[
H^i(\Spec \CC,u^\ast L)=H^i(\Spec V, L)=H^i(s,i^\ast L).
\]
But this is obvious, since $V$ is stricly Henselian, and so $L=(\underline{\Lambda}^1)^{\oplus n}$.
\end{proof}

Let $G$ be an algebraic group defined over $\ZZ$, acting on $X$. We keep the above setup, but we assume in addition
that the strata in
$\Tscr$ are $G$-invariant, and the elements of $\Lscr(T)$ are $G$-equivariant. We have the following result.

\begin{lem} 
\label{lemma5}
We can choose $V$ in such a way 
that there are equivalences of categories
\[
D^b_{\Tscr,\Lscr}(X/G,\Lambda)    \xleftarrow{u^\ast} D^b_{\Tscr_V,\Lscr_V}(X_V/G_V,\Lambda) \xrightarrow{i^\ast} D^b_{\Tscr_s,\Lscr_s}(X_s/G_s,\Lambda).
\]
\end{lem}

\begin{proof} We follow the above proof but replace $a$ by the structure map $X\r \Spec \CC/G$. 
Mimicking the the non-equivariant case we have to show that for
a $G$-equivariant locally constant sheaf of $\Lambda^1$-modules $L$ on $\Spec V$ one has
\[
H^i(\Spec \CC/G,u^\ast L)=H^i(\Spec V/G_V, L)=H^i(s/G_s,i^\ast L).
\]
Note that if we forget the $G$-structure then $L$ is a constant sheaf. Thus, using Proposition \ref{ffcor} 
we reduce to the scheme case.
\qed
\def\qed{}
\end{proof}

\subsubsection{Passing to a finite field: summary}
\label{secff}
For the benefit of the reader let us summarize the functors that were introduced in \S\ref{secliteral}-\ref{secfinchar}.
\begin{multline}
\label{functors}
D^b_c(Y(\CC),\CC)_{\text{lit}}\xleftarrow{\CC\otimes_{\Lambda}-}
D^b_c(Y(\CC),\Lambda)_{\text{lit}}\xleftarrow[\cong]{R\pi_\ast} D^b_c(Y(\CC),\Lambda)
\xleftarrow[\cong]{\epsilon^\ast}\\ D^b_c(Y,\Lambda)\supset  D^b_{\Tscr,\Lscr}(Y,\Lambda) \xleftarrow[\cong]{u^\ast} D^b_{\Tscr_V,\Lscr_V}(Y_V,\Lambda)
\xrightarrow[\cong]{i^\ast} D^b_{\Tscr_s,\Lscr_s}(Y_s,\Lambda)
\end{multline}
After tensoring with $\CC$ we then obtain functors
\begin{multline*}
D^b_c(Y(\CC),\CC)_{\text{lit}}\xleftarrow[f.f.]{\CC\otimes_{\Lambda}-}
D^b_c(Y(\CC),\Lambda)_{\text{lit},\CC}\xleftarrow[\cong]{R\pi_\ast} D^b_c(Y(\CC),\Lambda)_{\CC}
\xleftarrow[\cong]{\epsilon^\ast}\\ D^b_c(Y,\Lambda)_{\CC}\supset  D^b_{\Tscr,\Lscr}(Y,\Lambda)_{\CC} \xleftarrow[\cong]{u^\ast} D^b_{\Tscr_V,\Lscr_V}(Y_V,\Lambda)_{\CC}
\xrightarrow[\cong]{i^\ast} D^b_{\Tscr_s,\Lscr_s}(Y_s,\Lambda)_{\CC}
\end{multline*}
Furthermore, any object $\Fscr$ in $D^b_c(Y(\CC),\CC)$ is obtained from $\Fscr_s$ in some $D^b_{\Tscr_s,\Lscr_s}(Y_s,\Lambda)$.

\medskip

Assume $s=\Spec \bar{\FF}_p$.  There is some 
$s_0=\Spec \FF_q$ such that $Y_{s}$
is obtained by the base extension from $Y_{0}/s_0$. 
Unfortunately one cannot draw a similar conclusion for $\Fscr_s$. E.g., $\Fscr_s$
might be a rank one $\Lambda$-local system on $\AA^1_{\bar{\FF}_p}-\{0\}$ with infinite monodromy.

%The category $D^b_c(Y,\Lambda)$ is $\Lambda$-linear. 
For use below we make the following definition.
Let $a:Y_0\r s_0$, $p:s\r s_0$ be the structure
maps. Then we define $D^b_c(Y_0,\Lambda)_F$ as the category 
that has the same objects as $D^b_c(Y_0,\Lambda)$ but whose $\Hom$-spaces
are given by
\begin{equation}
\label{Homdef}
\Hom_{D^b_c(Y_0,\Lambda)_F}(\Fscr,\Gscr)=H^0R\Hom_Y(p^\ast \Fscr,p^\ast\Gscr).
\end{equation}
The right hand side 
of \eqref{Homdef} is equipped with a canonical Frobenius action, compatible with the $\Lambda$-module structure. 
 Thus, we can think of $D^b_c(Y_0,\Lambda)_F$ as being enriched
in $\Mod(\Lambda[F,F^{-1}])$ (the monoidal structure is given by tensoring over $\Lambda$), and more precisely, in the full subcategory
of $\Mod(\Lambda[F,F^{-1}])$ consisting of modules that are finitely generated over $\Lambda$.

We have a functor 
\[
D^b_c(Y_0,\Lambda)_F\r  D^b_c(Y,\Lambda):\Fscr\r p^\ast\Fscr
\]
which on the level of objects is given by the base extension and on the level of $\Hom$-spaces amounts to forgetting the $F$-action. 
%\marginpar{XXX Reference?} 
%This follows from the 
%standard fact that for constructible sheaves $\Escr\! \textit{xt}$ is 
%compatible with base field extension. 

\begin{rem} The category $D^b_c(Y_0,\Lambda)_F$ is not the same 
as $D^b_c(Y_0,\Lambda)$. For example, the former is not triangulated.
\end{rem} 

\subsection{End of proof of Theorem \ref{anr}}\label{DG-passing-finite-field-sec}

\subsubsection{Passing to the DG-context}
All the functors in \eqref{functors}
are induced from left derived functors between suitable
abelian categories (the inverse images are exact, and so they are both left and
right derived functors). Hence, they are represented by co-quasi-functors
on the DG-level (as in Example \ref{standardexample}).  We may then extend
the resulting bimodules to obtain co-quasi-functors
\def\dg{\operatorname{dg}}
\begin{multline*}
D^{b,\dg}_c(Y(\CC),\CC)_{\text{lit}}\l
D^{b,\dg}_c(Y(\CC),\Lambda)_{\text{lit},\CC}\l D^{b,\dg}_c(Y(\CC),\Lambda)_{\CC}
\l \\ D^{b,\dg}_c(Y,\Lambda)_\CC\supset  D^{b,\dg}_{\Tscr,\Lscr}(Y,\Lambda)_\CC \l D^{b,\dg}_{\Tscr_V,\Lscr_V}(Y_V,\Lambda)_{\CC}
\r D^{b,\dg}_{\Tscr_s,\Lscr_s}(Y_s,\Lambda)_{\CC},
\end{multline*}
where all the arrows induce equivalences or fully faithful embeddings on the $H^0$-level. By Lemma \ref{co-quasi}, 
it follows that the DG-algebras
$
\underline{\Hom}_{D^{b,\dg}_c(Y(\CC),\CC)}(\Fscr,\Fscr)
$
and
$
\underline{\Hom}_{D^{b,\dg}_c(Y_s,\Lambda)_{\CC}}(\Fscr_s,\Fscr_s)
$
are isomorphic in the homotopy category of DG-algebras over $\CC$.  

Thus, it remains to construct a DG-algebra, quasi-isomorphic
to $
\underline{\Hom}_{D^{b,\dg}_c(Y_s,\Lambda)_{\CC}}(\Fscr_s,\Fscr_s)
$, which carries a DG-endomorphism inducing the Frobenius on cohomology,
assuming of course that $\Fscr_s$ is indeed defined over a finite field. 
We do this in the next section.

\subsubsection{Lifting the Frobenius to the DG-level}
\def\Sh{\operatorname{Sh}} The technical problem we have to deal with
is that $p^\ast$ does not preserve injectives. So we have to follow
a more complicated procedure. 

For objects $\Fscr_0,\Gscr_0$ in $\Sh(Y_0^{\NN},\Lambda^\bullet)$ we set
\begin{equation}\label{new-hom-eq}
\Hom_{\Lambda}(\Fscr_0,\Gscr_0)_F=\pi_\ast p^{\NN,\ast} a^{\NN}_\ast \uHom_{\Sh(Y_0^{\NN},\Lambda^\bullet)}(\Fscr_0,\Gscr_0)\in \Mod(\Lambda[F,F^{-1}]),
\end{equation}
where as before $a:Y_0\r s_0$, $p:s\r s_0$ are the structure maps and $\pi$ is
the standard morphism of topoi $s^{\NN}\r s$.

Note that $p^{\NN,\ast} a^{\NN}_\ast \uHom_{\Sh(Y_0^{\NN},\Lambda^\bullet)}(\Fscr_0,\Gscr_0)$
is an object in $\Sh(s^{\NN},\Lambda^\bullet)$, i.e., an inverse system
of $\Lambda$-modules.
\begin{lem} 
\label{mllemma} If $\Gscr_0$ is injective then 
$p^{\NN,\ast} a^{\NN}_{\ast} \uHom_{\Sh(Y_0^{\NN},\Lambda^\bullet)}(\Fscr_0,\Gscr_0)$
is an inverse system with surjective structure maps. 
\end{lem}
\begin{proof} By the smooth base change and the compatibility of \'etale
  cohomology with inverse limits \cite[VII.5.8]{SGA4}, we have $p^\ast
  a_\ast=a_\ast p^\ast$. Unfortunately if $\Fscr$ is not constructible
  then $\uHom$ does not commute with the base change.  However, straight
  from the definitions we get
\[
( a^{\NN}_\ast p^{\NN,\ast} \uHom_{\Sh(Y_0^{\NN},\Lambda^\bullet)}(\Fscr_0,\Gscr_0))(s)^n=
\dirlim_{s'} \Hom_{\Sh(Y_{0,s'}^{\le n},\Lambda^\bullet)}(\Fscr_{0,s'}^{\le n},\Gscr_{0,s'}^{\le n}),
\]
where $s'$ runs through the factorizations $s\r s'\r s_0$ with $s'/s_0$ \'etale
(=finite), and $Y^{\le n}$ denotes the topos of inverse systems of length $n$. 

Now if $\Gscr$ is injective it is easy to see that $\Gscr_{s'}$ is injective
as well (since $Y_{s'}/Y$ is \'etale, there is an exact left adjoint, ``extension by zero''). 

So we have to prove that if $Y$ is a topos, $\Fscr$ and $\Gscr$
sheaves of $\Lambda^{\bullet}$-modules on $Y^{\NN}$, with $\Gscr$ injective,
then 
\[
\Hom_{\Sh(Y^{\le n+1},\Lambda^\bullet)}(\Fscr^{\le n+1},\Gscr^{\le n+1})
\r \Hom_{\Sh(Y^{\le n},\Lambda^\bullet)}(\Fscr^{\le n},\Gscr^{\le n})
\]
is surjective. This is an easy excercise.
\end{proof}

We define a DG-model  $D^{b,\dg}_c(Y_0,\Lambda)_F$ for $D^{b}_c(Y_0,\Lambda)_F$ (see \S\ref{secff})
as follows: $D^{b,\dg}_c(Y_0,\Lambda)_F$
has the same objects as $D^{b}_c(Y_0,\Lambda)_F$, and it is enriched in 
$C(\Lambda[F,F^{-1}])$ by setting
\[
\underline{\Hom}_{D^{b,\dg}_c(Y_0,\Lambda)_F}(\Fscr_0,\Gscr_0)=\underline{\Hom}_\Lambda(I_{\Fscr_0},I_{\Gscr_0})_F,
\]
where the right-hand side is defined by \eqref{new-hom-eq}, applied to injective resolutions.
By the fact that $\underline{\Hom}_\Lambda(I_{\Fscr_0},I_{\Gscr_0})$ is acyclic for $a^{\NN}_\ast$ and  
by Lemma \ref{mllemma}, we see that
$\underline{\Hom}_\Lambda(I_{\Fscr_0},I_{\Gscr_0})_F$ computes
\begin{align*}
R\pi_\ast p^{\NN,\ast} Ra^{\NN}_\ast \uRHom_{\Sh(Y_0^\NN,\Lambda^\bullet)}(\Fscr_0,\Gscr_0)
&=R\pi_\ast  Ra^{\NN}_\ast p^{\NN,\ast}\uRHom_{\Sh(Y_0^\NN,\Lambda^\bullet)}(\Fscr_0,\Gscr_0)\\
&=R\pi_\ast  Ra^{\NN}_\ast \uRHom_{\Sh(Y_s^\NN,\Lambda^\bullet)}(p^\ast \Fscr_0,p^\ast\Gscr_0)
\end{align*}
where we have used the fact that for constructible (cartesian in the
$G$-equivariant case) sheaves $\uRHom$ commutes with
%\marginpar{Reference???}  
the base change.  Since $s$ is a point topos,
$a_\ast$ on $Y_s$ is just taking global sections. Thus,
\begin{align*}
R\pi_\ast  Ra_\ast^{\NN} \uRHom_{\Sh(Y_s^\NN,\Lambda^\bullet)}(p^\ast \Fscr_0,p^\ast\Gscr_0)&=
R\underset{n}{\varprojlim}
%\invlim_n  
R\Hom_{\Sh(Y_s^{\le n},\Lambda^\bullet)}(p^\ast \Fscr_0^{\le n},p^\ast\Gscr_0^{\le n})\\
&=R\Hom_{D^b_c(Y_s,\Lambda)}(p^\ast \Fscr_0,p^\ast \Gscr_0),
\end{align*}
so that $D^{b,\dg}_c(Y_0,\Lambda)_F$ is indeed a DG-model for $D^b(Y_0,\Lambda)_F$.

Let $D^{b,\dg}_c(Y_0,\Lambda)^0_F$, $D^{b}_c(Y_0,\Lambda)^0_F$ be obtained from $D^{b,\dg}_c(Y_0,\Lambda)_F$, 
$D^{b}_c(Y_0,\Lambda)_F$ by forgetting
the $F$-action. Then the fully faithful functor
\[
 D^{b}_c(Y_0,\Lambda)^0_F\r D^{b}_c(Y,\Lambda):\Fscr\mapsto p^\ast \Fscr
\]
can be lifted to a co-quasi-functor
\[
 D^{b,\dg}_c(Y_0,\Lambda)^0_F\r D^{b,\dg}_c(Y,\Lambda)
\]
as in Example \ref{standardexample}. Furthermore, by Lemma \ref{co-quasi}, we deduce that for 
$\Fscr_0\in D^b_c(Y_0,\Lambda)$,
\[
A\overset{\text{def}}{=}\underline{\Hom}_{ D^{b,\dg}_c(Y_0,\Lambda)^0_{F,\CC}}(\Fscr_0,\Fscr_0)
\]
and
\[
\underline{\Hom}_{ D^{b,\dg}_c(Y,\Lambda)_{\CC}}(p^\ast \Fscr_0,p^\st \Fscr_0)_{\CC}
\]
are isomorphic in the homotopy category of DG-algebras over $\CC$.
Since $A$ is a ring object in $C(\CC[F,F^{-1}])$,  we are done.

%\end{document}

\section{Appendix II: Formality for some algebras over operads}\label{formality-sec}

\subsection{Formulation of the main result on formality}
\label{ref-1-0}
Throughout thi section $k$ is a field.
If $V=\bigoplus_n V_n$ is a $\ZZ$-graded $k$-vector space and $F$ is a graded $k$-automorphism of $V$
then we say that $F$ acts \emph{locally finitely} on $V$ if every $v\in V$ is
contained in a finite dimensional $F$-invariant graded subspace of $V$. 

Assume now that $F$ acts locally finitely on $V$ and assume that $k$
is algebraically closed. If $\alpha\in k$ then we denote by $V_\alpha$
the generalized eigenspace of $V$ corresponding to $\alpha$.  In other
words, $v\in V_\alpha$ iff there is some $n$ such that
$(F-\alpha)^nv=0$.  This yields an $F$-invariant direct sum
decomposition $V=\bigoplus_{\alpha\in k^\ast} V_\alpha$ of graded
vector spaces. We will think of this decomposition as a refinement of the
given $\ZZ$-grading to a $\ZZ\times k^\ast$ grading.

We now assume in addition that $k=\CC$. In that case we fix throughout some strictly positive 
real number $\xi$ and we say that a pair $(V,F)$, with $V$ a graded
$k$-vector space and $F$ a $k$-automorphism of $V$, is \emph{pure of weight} $m$ (with 
respect to $\xi$) 
  if $F$ acts locally finitely on $V$ and 
if $n\in \ZZ$, $\alpha\in k^\ast$ are such that $V_{n,\alpha}\neq 0$ 
then $|\alpha|=\xi^{n+m}$.

If $V$ is a complex and $F$ is an endomorphism of degree zero of $V$ (i.e., $V$ commutes with the
differential) then we say
that $(V,F)$ is pure of weight $m$ if $(F,H^\ast(V))$ is pure of weight $m$.

We now fix some $k$-DG-operad $\Oscr$. 
Below we will
prove the following result. 
\begin{thm} \label{ref-1.1-1} Assume that $k=\CC$. Let $F:B\r B$ be an $\Oscr$-algebra quasi-isomorphism
such that $(B,F)$ is pure of weight zero. Then $B$ is formal. In other words,
$B$ is isomorphic to $(H^\ast(B),d=0)$ in the \emph{homotopy
category} of $\Oscr$-algebras. The latter is by definition the category of $\Oscr$-algebras with 
quasi-isomorphisms inverted.

Assume now that $B$ itself has zero differential, and let $N$ be a module 
over~$B$ equipped
with a $k$-linear endomorphism $F:N\r N$ of complexes,
such that for any $l\in \NN$ and $\mu\in \Oscr(l)$ we have $F(\mu(b_1,\ldots,b_{l-1},n))=\mu(F(b_1),\ldots,F(b_{l-1}),Fn)$
 for $b_1,\ldots,b_{l-1}\in B$, $n\in N$. Assume in addition that $(N,F)$ is pure (of some weight). 
Then $N$ is formal. In other words, $N$ is isomorphic in $D(B)$ to
$(H^\ast(N),d=0)$.
\end{thm}
This theorem puts a number of partial results in the literature in the proper context---see e.g., 
\cite[Cor. 5.3.7]{deligneweil2}, \cite{Petersen}, and \cite[Thm 12.7]{Sullivan}. Among other
things these prior results assume that $B_n=0$ for $n<0$ and $B_0=k$, since they depend
on Sullivan's theory of minimal models.

Theorem \ref{ref-1.1-1} will be a consequence of the following result.
\begin{thm} 
\label{ref-1.2-2} Assume that either $k$ is a field of characteristic zero or that $\Oscr$ is obtained from
an asymmetric operad. Let $F:B\r B$ be an $\Oscr$-algebra quasi-iso\-morph\-ism
such $F$ acts locally finitely on $H^\ast(B)$. Then there is a quasi-isomorphism
of $\Oscr$-algebras $\alpha:A\r B$, where $A$ is in addition equipped with
a $k$-linear locally finite automorphism $F'$, such that the diagram
\[
\xymatrix{
A\ar[r]^\alpha\ar[d]_{F'} & B\ar[d]^F\\
A\ar[r]_\alpha & B
}
\]
is commutative in the homotopy category of $\Oscr$-algebras.

Similarly assume that $B$ itself has zero differential. For a module
$P$ over $B$ write ${}_F P$ for the module over $B$, which is the same as
 $P$ as a complex but with the twisted  $B$-action 
 $${}_F\mu(b_1,\ldots,b_{l-1},n)=\mu(F(b_1),\ldots,F(b_{l-1}),n),$$ 
 where $\mu \in \Oscr(l)$, $b_1,\ldots,b_{-1}\in B$, $n\in N$.

Let $N$ be
a module over $B$ equipped with a $k$-linear quasi-isomorphism $F:N\r {}_F N$ which acts
locally finitely on $H^\ast(N)$.
 Then
there is a quasi-isomorphism of $B$-modules $\alpha:M\r N$, where $M$ is in
addition equipped with a locally finite $k$-isomorphism $F':M\r {}_F M$, 
such that the diagram
\[
\xymatrix{
M\ar[r]^{\alpha}\ar[d]_{F'} & N\ar[d]^F\\
{}_FM\ar[r]_{\alpha} & {}_FN
}
\]
is commutative in $D(B)$. 
\end{thm}

%\subsection{Proof of Theorem \ref{ref-1.2-2}}

We will prove Theorem \ref{ref-1.2-2} first in the characteristic zero case. In Section \S\ref{ref-2.5-16} we will give
trivial modifications needed for the asymmetric operad case. Theorem \ref{ref-1.1-1} will be deduced from
Theorem  \ref{ref-1.2-2} in \S\ref{proof-of-formality-sec}.

So, until further notice, $k$ is a field of characteristic zero, $A,B,C,D$ are $\Oscr$-algebras, $V$, $W$ are complexes 
over $k$, and $v$, $w$, $a$, $b$, $c$, $d$ are
 elements of $V,W, A,B,C,D$, respectively.

Unspecified maps in diagrams below are assumed to be $k$-linear maps of complexes, unless
they are between $\Oscr$-algebras in which case they are assumed to
be $\Oscr$-algebra morphisms. Maps indexed by $t$ are homotopies between $\Oscr$-algebra
morphisms (see \S\ref{ref-2.1-3} below for details). We denote by $s$ the suspension functor on complexes.

Some maps between complexes serve as homotopies, or even as 2-homotopies, in which
case they have degree -1,-2 respectively---this will be clear from
the context.

\subsection{Homotopies between $\Oscr$-algebra morphisms}
\label{ref-2.1-3}
If $\alpha,\beta:A\r B$ are morphisms of $\Oscr$-algebras then a homotopy
between $\alpha$ and $\beta$
is a pair $(\phi_t,h_t)$ where $\phi_t:A\r B[t]$ is an $\Oscr$-algebra morphism such that
$\phi_0=\alpha$, $\phi_1=\beta$,
and $h_t:A\r B[t]$ is a $\phi_t$-derivation of degree $-1$ satisfying
\begin{equation}
\label{ref-2.1-4}
\frac{\partial \phi_t}{\partial t}=dh_t\overset{\text{def}}{=}d_B\circ h_t+h_t\circ d_A.
\end{equation}
Note that \eqref{ref-2.1-4} implies that $\phi_t$ is completely determined by
$h_t$. Hence, it is not necessary to specify $\phi_t$.

This definition of homotopy corresponds to a right homotopy in the sense of \cite{quillen22} associated to the path-object $B^I=B\otimes \Omega_{\AA^1}=B[t]\oplus B[t]dt$.
It is well-known and easy to see that if there exists a homotopy between $\alpha$ and $\beta$ 
then $\alpha$ is equal to $\beta$ in the homotopy category of $\Oscr$-algebras (see e.g., \cite[Prop.\ B.2]{Berest}).

\subsection{Adding variables to kill cocycles} 
\label{ref-2.2-5}
We recall a construction from \cite{hinichhomotopy}.
Let $i:s^{-1}V\r A$ be a map of complexes. Then as a graded $\Oscr$-algebra $A\langle V,i\rangle$ is 
the coproduct of $A$
and $T_{\Oscr}V$ over $k$. 
We equip $A\langle V,i\rangle$  with a  differential $\tilde{d}$ 
 given by
\begin{align*}
\tilde{d}a&=d_Aa,\\
\tilde{d}v&=d_V v+c_ii(s^{-1}v),
\end{align*}
where $c_i:A\r A\langle V,i\rangle$ is the canonical map. 
Note that $c_i i=dh_i$  where $h_i(s^{-1}v)=v$. In particular, the composition
\[
s^{-1}V\xrightarrow{i} A\xrightarrow{c_i} A\langle V,i\rangle
\]
is zero on cohomology. Whence the title of this section.

A diagram of the form
\[
\xymatrix{
s^{-1}V\ar@/^1.5em/@{.>}[rr]^{h_\alpha}\ar[r]_i& A\ar[r]_\alpha& B
}
\]
with $\alpha i=dh_{\alpha}$ can be transformed into a commutative diagram
of $\Oscr$-algebras
\[
\xymatrix{
A\ar[r]_-{c_i}\ar@/^1.5em/[rr]^{\alpha}& A\langle V,i\rangle \ar[r]_-{\tilde{\alpha}} & B
}
\]
by setting
$
\tilde{\alpha}(v)=h_\alpha(s^{-1}v)
$.

\medskip

Suppose now we have the following diagram
\begin{equation}
\label{ref-2.2-6}
\xymatrix{
s^{-1}V\ar@{-->}`l[d]`[dd]`[rr]_p[rrd]\ar@{.>}[dr]|{h_1}\ar[r]^i\ar[d]_{F_1}\ar@{.>}@/^2em/[rr]^{h_\alpha}& A\ar@{.>}[dr]|{h_{2,t}}\ar[r]^\alpha\ar[d]_{F_2}& B\ar[d]^{F_3}\\
s^{-1}W\ar[r]_j\ar@{.>}@/_2em/[rr]_{h_\gamma} & C\ar[r]_\gamma & D\\
&}
\end{equation}
such that besides the identities $dh_\alpha=\alpha i$, $dh_\gamma=\gamma j$
we have that
 $h_{2,t}$ is a homotopy of DG-algebras between $F_3\alpha$ and $\gamma F_2$, 
and 
\begin{align}
jF_1-iF_2&=dh_1,\\
h_\gamma F_1-F_3h_\alpha&=\gamma h_1+\left(\int_0^1h_{2,t}dt\right) i+dp. \label{ref-2.4-7}
\end{align}
We can transform this into a diagram of the form
\begin{equation}
\label{ref-2.5-8}
\xymatrix{
A \ar@{.>}`l[d]`[dd]`[rr]_{h_{2,t}}[rrd]\ar[r]^{c_i}\ar[d]_{F_2}\ar@/^2em/[rr]^{\alpha}& A\langle V,i\rangle\ar@{.>}[dr]|{\tilde{h}_{2,t}}\ar[r]^{\tilde{\alpha}}\ar[d]_{\tilde{F}_2}& B\ar[d]^{F_3}\\
C\ar[r]_{c_j}\ar@/_2em/[rr]_{\gamma} & C\langle W,j\rangle\ar[r]_{\tilde{\gamma}} & D\\
&
}
\end{equation}
by setting
\begin{align*}
\tilde{F}_2(a)&=F_2(a),\\
\tilde{F}_2(v)&=F_1(v)-h_1(s^{-1}v),\\
\tilde{h}_{2,t}(a)&=h_{2,t}(a),\\
\tilde{h}_{2,t}(v)&=p(s^{-1}v),
\end{align*}
where we have interpreted $F_1$ as a map $V\r W$ via
$F_1(v)=sF_1(s^{-1}v)$. 
\subsection{The proof for Theorem \ref{ref-1.2-2}  for $\Oscr$-algebras}
Theorem \ref{ref-1.2-2} is a consequence of the following slightly
more specific theorem.
\begin{thm} 
\label{ref-2.1-9} Let $B$ be an $\Oscr$-algebra and $F:B\r B$ an $\Oscr$-algebra quasi-isomorphism which acts locally finitely on $H^\ast(B)$.
Then there exists a diagram
\begin{align*}
\xymatrix{
A\ar[r]^\alpha \ar[d]_{F'}\ar@{.>}[dr]|{h_t}& B\ar[d]^F\\
A\ar[r]_\alpha & B
}
\end{align*}
where $\alpha$ is a quasi-isomorphism, $F':A\r A$ is a 
$k$-linear $\Oscr$-algebra automorphism which acts locally finitely,
and
$h_t$ is a   homotopy of $\Oscr$-algebras between $F\alpha$ and $\alpha F'$.
\end{thm}
\begin{proof}
We will construct a sequence of $\Oscr$-algebra morphisms
\[
\xymatrix{
k=A_0\ar[r]_{\beta_0}\ar@/^3em/[rrrr]|{\alpha_0} &A_1\ar[r]_{\beta_1}\ar@/^2em/[rrr]|{\alpha_1}& A_2\ar[r]_{\beta_2}\ar@/^1em/[rr]|{\alpha_2}&\cdots& B
}
\]
where 
\begin{enumerate}
\item $\alpha_0:k\r B$, $\beta_0:k\r A$ are the unit maps;
\item  $\alpha_{i+1}\beta_i=\alpha_i$;
\item $\alpha_i$ for $i>0$ is surjective on the level of cohomology;
\item the composition 
\begin{equation}
\label{ref-2.6-10}
s^{-1}(\operatorname{cone} \alpha_i)\r A_i\xrightarrow{\beta_i} A_{i+1}
\end{equation}
is zero on cohomology.
\end{enumerate}
From this it will easily follow that the induced map
\[
A\overset{\text{def}}{=} \varinjlim_i A_i\xrightarrow{\alpha\overset{\text{def}}{=}\varinjlim\alpha_i} B
\]
is a quasi-isomorphism.

 In addition, we will construct
diagrams of the form
\begin{equation}
\label{ref-2.7-11}
\xymatrix{
A_i\ar@{.>}[dr]|{h_{i,t}} \ar[r]^{\alpha_i}\ar[d]_{F_i} & B\ar[d]^{F}\\
A_i \ar[r]_{\alpha_i} & B
}
\end{equation}
such that $F_i$ is a locally finite $\Oscr$-algebra automorphism of $A_i$, 
and  $h_{i,t}$ is a  homotopy between $F\alpha_i$ and $\alpha_i F_i$ such that
\[
h_{i+1,t}\alpha_i=h_{i,t}.
\]
We proceed inductively. The induction starts with the given $\alpha_0,\beta_0$ and $h_{0,t}=0$.
Next, assume we have a diagram like \eqref{ref-2.7-11}. Set 
$C_i=s^{-1}(\operatorname{cone} \alpha_i)$. If we identify
$C_i$ with the column vectors
\[
\begin{pmatrix}
A_i\\
s^{-1}B
\end{pmatrix}
\]
then the differential on $C_i$ is given by
\[
\begin{pmatrix}
d_{A_i} & 0\\
\alpha & d_{s^{-1}B}
\end{pmatrix}.
\]
We may now extend
\eqref{ref-2.7-11} to a diagram\footnote{This is of course just the TR3 axiom but we need the
precise nature of the maps.}
\begin{equation}
\label{ref-2.8-12}
\xymatrix{
C_i\ar[r]^p\ar[d]_{G_i}\ar@{.>}@/^1.5em/[rr]^q& A_i\ar@{.>}[dr]|{h_{i,t}} \ar[r]^{\alpha_i}\ar[d]_{F_i} & B\ar[d]^{F}\\
C_i\ar[r]_p\ar@{.>}@/_1.5em/[rr]_q& A_i \ar[r]_{\alpha_i} & B
}
\end{equation}
where $p,q$ are the projection maps and
\[
G_i=\begin{pmatrix}
F_i & 0\\
\int^1_0 h_{i,t}dt & F
\end{pmatrix}.
\]
We have
\begin{align*}
\alpha_i p&=dq,\\
qG_i-Fq&=\left(\int_0^1h_{i,t}dt\right) p.
\end{align*}
Set $V_i=s H^\ast(C_i)$, $G'_i=H^\ast(G_i)$.
Next, we construct a diagram of complexes
\begin{equation}
\label{ref-2.9-13}
\xymatrix{
s^{-1}V_i\ar@{.>}[rd]|{h'_i}\ar[r]^t\ar[d]_{G'_i} & C_i\ar[d]^{G_i}\\
s^{-1}V_i\ar[r]_t & C_i
}
\end{equation}
where $t$ is an arbitrary quasi-isomorphism $s^{-1}V_i\r C_i$, and $G'_it-tG_i=dh'_i$. We
may combine \eqref{ref-2.8-12} and \eqref{ref-2.9-13} to obtain
\[
\xymatrix{
s^{-1}V_i\ar@{-->}`l[d]`[dd]`[rr]_{-qh_i'}[rrd]\ar@{.>}[dr]|{ph_i'}\ar[r]^{pt}\ar[d]_{G'_i}\ar@{.>}@/^1.5em/[rr]^{qt}& A_i\ar@{.>}[dr]|{h_{i,t}} \ar[r]^{\alpha_i}\ar[d]_{F_i} & B\ar[d]^{F}\\
s^{-1}V_i\ar[r]_{pt}\ar@{.>}@/_1.5em/[rr]_{qt}& A_i \ar[r]_{\alpha_i} & B\\
&
}
\]
with the same relations as in \eqref{ref-2.2-6}. 
Hence, we may transform this diagram into
a diagram of the form \eqref{ref-2.5-8}:
\[
\xymatrix{
A_i \ar@{.>}`l[d]`[dd]`[rr]_{h_{i,t}}[rrd]\ar[r]^{c_{pt}}\ar[d]_{F_i}\ar@/^2em/[rr]^{\alpha_i}& A_i\langle V_i,pt\rangle\ar@{.>}[dr]|{\tilde{h}_{i,t}}\ar[r]^{\tilde{\alpha}_i}\ar[d]_{\tilde{F}_i}& B\ar[d]^{F}\\
A_i\ar[r]_{c_{pt}}\ar@/_2em/[rr]_{\alpha} & A_i\langle V_i,pt\rangle\ar[r]_{\tilde{\alpha_i}} & B\\
&
}
\]
Finally, we set
\begin{align*}
A_{i+1}&=A_i\langle  V_i,pt\rangle,\\
F_{i+1}&=\tilde{F}_i,\\
h_{i+1}&=\tilde{h}_i,\\
\beta_i&=c_{pt}.
\end{align*}
It remains to check several things.
\begin{enumerate}
\item 
$F_i$ acts locally finitely on $A_i$ (and hence so does $F'$ on $A$). We prove this by induction. The
case $i=0$ is clear. Now assume that $F_{i}'$ acts locally finitely on $A_{i}$
for some $i\ge 0$. 
Then by the long exact sequence for cohomology it follows that $G'_{i}$
acts locally finitely on $H^\ast(C_i)=s^{-1}V_i$.  Now
if we equip $A_{i+1}=A_i\langle V_i,pt\rangle$  with the ascending filtration
obtained from the grading on $A_i\langle V_i,pt\rangle$ defined by $|A_i|=0$, $|V_i|=1$, then we obtain
$\gr A_{i+1}=A_i\coprod_k T_{\Oscr} V_i$, and the induced action
of $\gr F_{i+1}=\gr \tilde{F}_i$ is given by $F_i\coprod G'_i$.
Since both $F_i$ and $G'_i$ act locally finitely, the same holds for $F_i\coprod G'_i$
and hence for $F_{i+1}$.
\item $F_i$ is an automorphism of $A_i$ (and hence $F'$ is an automorphism of $A$). 
This is again proved by induction. It is clearly true for $A_0$. Assume it
is true for some $A_i$. By the same argument as above we find that $\gr F_{i+1}$
is an automorphism of $\gr A_{i+1}$. It follows easily that $F_{i+1}$ is an automorphism of $A_{i+1}$.
\item $\alpha_i$ for $i>0$ is surjective on the level of cohomology. It suffices to prove this
for $A_1$, which is a simple verification.
\item The composition \eqref{ref-2.6-10} is zero on cohomology. To
see this note that it is sufficient to prove that the longer composition
\begin{equation}
\label{ref-2.10-14}
s^{-1}V_i\r s^{-1}(\operatorname{cone} \alpha_i)\r A_i\xrightarrow{\beta_i} A_{i+1}
\end{equation}
is zero on cohomology since the first map is a quasi-isomorphism. The fact
that \eqref{ref-2.10-14} is zero on cohomology is clear since it is just the
composition
\[
s^{-1}V_i\xrightarrow{pt}  A_i\r A_i\langle V,pt\rangle
\]
which was discussed in \S\ref{ref-2.2-5}.\qed
\end{enumerate}
\def\qed{}\end{proof}

\subsection{Theorem \ref{ref-1.2-2}  for modules}
The proof for modules is entirely analogous to the one for $\Oscr$-algebras. So
we will provide very few details. 

The key ingredient is an analogue of $A\langle V,i\rangle$ (see \S\ref{ref-2.2-5}). Let
$M$ be a module over $A$ and let $i:s^{-1}V\r M$ be a map of complexes. Then as graded $A$-module $M\langle V,i\rangle$ is 
equal to $M\bigoplus A\otimes_k V$.
We equip $M\langle V,i\rangle$  with the differential $\tilde{d}$ 
 given by
\begin{align*}
\tilde{d}m&=d_Mm,\\
\tilde{d}v&=d_V v+c_ii(s^{-1}v),
\end{align*}
for $m\in M$, $v\in V$, where $c_i:M\r M\langle V,i\rangle$ is the canonical map.  With this definition one may imitate
the above proof for the case of $\Oscr$-algebras. The only slight variation is that the diagram \eqref{ref-2.7-11} now takes the
form
\begin{equation}
\label{ref-2.11-15}
\xymatrix{
M_i\ar@{.>}[dr]|{h_i} \ar[r]^{\alpha_i}\ar[d]_{F_i} & N\ar[d]^{F}\\
{}_FM_i \ar[r]_{\alpha_i} & {}_FN
}
\end{equation}
where $h_i$ is a homotopy between $F\alpha_i$ and $\alpha_i F_i$.

\subsection{The asymmetric operad case}
\label{ref-2.5-16}
Assume now that $k$ is arbitrary but that $\Oscr$ is obtained from an asymmetric operad which we denote by $\Oscr$ as well. 
In this case the following different version of homotopy is more convenient.

If $\alpha,\beta:A\r B$ are morphisms of $\Oscr$-algebras then a homotopy
between $\alpha$ and $\beta$ will be a $(\alpha,\beta)$-derivation $h:A\r B$ of degree $-1$
such that $\beta-\alpha=dh$.
This notion
 corresponds
to a right homotopy associated the path object of upper triangular $2\times2$-matrices (see \cite[\S3.1]{kellerexact}). The latter is  characteristic free
but can only be defined for asymmetric operads.

We may now adapt the proof of Theorem \ref{ref-1.2-2} to this new notion of homotopy. Since the homotopies
are now constant, there is no $t$-parameter and consequently no integrals in the formulas. Everything else is the same.

\subsection{Proof of Theorem \ref{ref-1.1-1}}\label{proof-of-formality-sec}
Let $(B,F)$ be as in the statement of Theorem \ref{ref-1.1-1}. Let $(A,F')$ be as in Theorem \ref{ref-1.1-1}.
By the construction of $A$, we see that the generalized eigenvalues $\mu$ of $F'$ acting  on $A$ have the property 
$|\mu|=\zeta^{n}$ for some $n\in \ZZ$. Set
\[
A_{n,m}=\bigoplus_{|\mu|=\zeta^m} A_{n,\mu}
\]
It is easy to see that this defines a $\ZZ^2$-grading on $A$ with the differential sending $A_{n,m}$ to $A_{n+1,m}$.
The purity hypothesis implies that the cohomology of $A$ is supported on the diagonal $\{(m,m)\mid m\in\ZZ\}$. 
To conclude we may now use the standard subalgebra of $A$, quasi-isomorphic both to $A$ and to its cohomology 
(see \cite[Prop.\ 4]{schnurer}). Recall that this subalgebra $A'$ is given by
$$A'_{ij}=\begin{cases} A_{ij} & i<j,\\ \ker(d:A_{ij}\to A_{i+1,j}), & i=j,\\ 0 & i>j.\end{cases}$$

The argument for modules is similar.

\def\cprime{$'$} \def\cprime{$'$} \def\cprime{$'$}
\providecommand{\bysame}{\leavevmode\hbox to3em{\hrulefill}\thinspace}
\providecommand{\MR}{\relax\ifhmode\unskip\space\fi MR }
\providecommand{\MRhref}[2]{%
  \href{http://www.ams.org/mathscinet-getitem?mr=#1}{#2}
}
\providecommand{\href}[2]{#2}

\end{document}